%% file: Leopardi-Bent-functions-Cayley-graphs.tex
\documentclass[12pt,a4paper]{article}
\usepackage{amsmath}
\usepackage{amssymb}
\usepackage{amsthm}
\usepackage{xcolor}
\usepackage[colorlinks=true,allcolors=blue]{hyperref}
\usepackage{graphicx}
\usepackage[figurename={Figure}]{caption}
\usepackage{float}
\usepackage{sagetex}
\input{Leopardi-Bent-functions-Cayley-graphs.commands.tex}

\input{Leopardi-Bent-functions-Cayley-graphs.theorems.tex}

\title{Classifying bent functions by their Cayley graphs}

\author{
Paul~Leopardi
\thanks{University of Melbourne; Australian Government -- Bureau of Meteorology
\protect\url{mailto:paul.leopardi@gmail.com}}
}

\date{11 December 2018}

\begin{document}

\maketitle

\begin{abstract}
\input{Leopardi-Bent-functions-Cayley-graphs.abstract.tex}
\end{abstract}

\input{Leopardi-Bent-functions-Cayley-graphs.body.tex}
\end{document}

%% file: Leopardi-Bent-functions-Cayley-graphs.commands.tex
\newcommand{\mb}[1]{\mathbb{#1}}

\newcommand{\F}{\mb{F}}

\newcommand{\Z}{\mb{Z}}
\newcommand{\abs}[1]{\left| #1 \right|}

\newcommand{\To}{\rightarrow}
\newcommand{\Emph}[1]{\emph{\textcolor{blue}{#1}}}
\newcommand{\Cay}[1]{\operatorname{Cay}\left(#1\right)}
\newcommand{\diag}[1]{\operatorname{diag}\left(#1\right)}
\newcommand{\dual}[1]{\widetilde{#1}}
\newcommand{\support}[1]{\operatorname{supp}\left(#1\right)}
\newcommand{\weight}[1]{\operatorname{wt}\left(#1\right)}
\newcommand{\weightclass}[1]{\operatorname{wc}\left(#1\right)}
\newenvironment{proofof}[1]{\noindent\emph{Proof of #1.}}{\qed}

%% file: Leopardi-Bent-functions-Cayley-graphs.theorems.tex
\newtheorem{Lemma}{Lemma}

\newtheorem{Proposition}{Proposition}

\newtheorem{Theorem}{Theorem}

\newtheorem{Corollary}[Lemma]{Corollary}
\theoremstyle{remark}
\newtheorem*{remark}{Remark}
\newtheorem{Remark}{Remark}
\newtheorem*{remarks}{Remarks}
\theoremstyle{definition}
\newtheorem*{definition}{Definition}
\newtheorem{Definition}{Definition}
\newtheorem*{example}{Example}


%% file: Leopardi-Bent-functions-Cayley-graphs.abstract.tex
In 1999 Bernasconi and Codenotti noted that the Cayley graph of a bent function is strongly regular.
This paper describes the concept of extended Cayley equivalence of bent functions,
discusses some connections between bent functions, designs, and codes,
and explores the relationship between extended Cayley equivalence and extended affine equivalence.
SageMath scripts and CoCalc worksheets are used to compute and display some of these relationships,
for bent functions up to dimension 8.

%% file: Leopardi-Bent-functions-Cayley-graphs.body.tex
\section{Introduction}
\label{sec-Introduction}
Binary bent functions are important combinatorial objects.
Besides the well-known application of bent functions and their generalizations to cryptography
\cite{Ada97} \cite[4.1-4.6]{Tok15bent},
bent functions have well-studied connections to Hadamard difference sets \cite{Dil74},
symmetric designs with the symmetric difference property \cite{DilS87block,Kan75symplectic},
projective two-weight codes \cite{CalK1986,Din2015} and strongly regular graphs.

In two papers, Bernasconi and Codenotti \cite{BerC99}, and then Bernasconi, Codenotti and Vanderkam
\cite{BerCV01} explored some of the connections
between bent functions and strongly regular graphs.
While these papers established that the Cayley graph of a binary bent function (whose value at 0 is
0) is a strongly regular graph
with certain parameters, they leave open the question of which strongly regular graphs with these
parameters are so obtained.

In a recent paper \cite{Leo17Hurwitz},
the author found an example of two infinite series of bent functions whose
Cayley graphs have the same strongly regular parameters at each dimension,
but are not isomorphic if the dimension is 8 or more.

Kantor, in 1983 \cite{Kan83exponential}, showed that the numbers of non-isomorphic projective linear
two weight codes with certain parameters,
Hadamard difference sets, and symmetric designs with certain properties, grow at least exponentially
with dimension.
This result suggests that the number of strongly regular graphs obtained as Cayley graphs of bent
functions also increases at least exponentially with dimension.


A 2003 paper by Cameron \cite{Cam2003} considers random strongly regular graphs with given parameters,
and outlines some prerequisites for a theory of random strongly regular graphs, including that
``the number of non-isomorphic graphs is superexponential in the number of vertices.''

The goal of the current paper is to further explore the connections between bent functions, their
Cayley graphs, and related combinatorial objects,
and in particular to examine the relationship between various equivalence classes of bent
functions, in particular, the relationship between the extended affine
equivalence classes and equivalence classes defined by isomorphism of Cayley graphs.
As well as a theoretical study of bent functions of all dimensions, a computational study is conducted
into bent functions of dimension at most 8,
using SageMath \cite{SageMath7517} and CoCalc \cite{CoCalc}.

The methodology for this study is modelled on \emph{experimental mathematics}.
As stated by Borwein and Devlin \cite[pp. 4-5]{BorD2008}:
\begin{quotation}
What makes experimental mathematics different (as an enterprise)
from the classical conception and practice of mathematics is that the
experimental process is regarded not as a precursor to a proof, to be
relegated to private notebooks and perhaps studied for historical
purposes only after a proof has been obtained. Rather, experimentation
is viewed as a significant part of mathematics in its own right, to
be published, considered by others, and (of particular importance)
\emph{contributing to our overall mathematical knowledge.}
\end{quotation}

The theoretical results listed this paper serve a few purposes.
First, in order to classify bent functions by their Cayley graphs,
it helps to understand the relationship between Cayley equivalence and other concepts of equivalence
of bent functions, especially if this helps to cut down the search space needed for the
classification.
A similar consideration applies to the duals of bent functions.
Second, some of the empirical observations made in the classification of bent functions in small
dimensions can be explained by these theoretical results.
Third, these theoretical results can improve our understanding of the relationships between
bent functions, projective two-weight codes, strongly regular graphs, and
symmetric block designs with the symmetric difference property.
In what follows, known results with references are presented as propositions, or sometimes as remarks;
and results where the statement or the proof seems to be missing or obscure within the existing literature
are presented as lemmas or theorems, with proofs.

This paper makes no pretence at being an exhaustive survey, neither should it be construed that
the lemmas and theorems listed here make any claim to originality.
Even given the current excellent electronic search capabilities available, it would be folly to do so given
the extensive literature that has been generated on the study of bent functions since the 1960s.
Some recent surveys include books by Cusick and Stanica \cite{CusS2017},
Mesnager \cite{Mes2016} and Tokareva \cite{Tok15bent},
and the article by Carlet and Mesnager \cite{CarM2016four}.

The remainder of the paper is organized as follows.
Section~\ref{sec-Preliminaries} covers the concepts, definitions and known results used later in the paper.
Section~\ref{sec-Bent-graphs} discusses the relationship between bent functions and strongly regular graphs.
Section~\ref{sec-Equivalence} introduces various concepts of equivalence of bent functions.
Section~\ref{sec-Bent-designs} discusses the relationship between bent functions and block designs.
Section~\ref{sec-Code} describes the SageMath and CoCalc code that has been used to obtain
the computational results of this paper.
Section~\ref{sec-Discussion} puts the results of this paper in the context of questions that are still open.
The appendices contain the proof of one of the properties of quadratic bent functions,
and list some of the properties of the equivalence classes of bent functions for dimension up to 8.

\section{Preliminaries}
\label{sec-Preliminaries}

This section presents some of the key concepts used in the remainder of the paper.
We first examine Boolean functions, then define bent Boolean functions,
and finally explore the relationships between bent functions
and Hamming weights.

\subsection{Boolean functions}
\label{sec-Boolean-functions}

Here and in the remainder of the paper, $\F_2$ denotes the field of two elements,
also known as $GF(2)$. Models of $\F_2$ include integer arithmetic modulo 2
($\Z/2\Z$ also known as $\Z_2$) and Boolean algebra with ``exclusive or'' as addition and ``and''
as multiplication.
\paragraph*{Boolean functions and Reed-Muller codes.}
Any Boolean function $f : \F_2^n \To \F_2$ can be represented as a reduced polynomial in at most $n$ variables over $F_2$
\cite{Mul54, Rot76} \cite[Ch. III, Section 2]{Dil74}.
This is called the \Emph{algebraic normal form} of $f$ \cite[Chapter 5]{Rue1986}.

\begin{example}
In Sage, using \verb!sage.crypto.boolean_function!, define:
\begin{sageblock}
from sage.crypto.boolean_function import BooleanFunction
f = BooleanFunction([1,1,1,0])
a = f.algebraic_normal_form()
\end{sageblock}
The algebraic normal form of the Boolean function \verb!f! on $\F_2^2$
with variables $x_0$ and $x_1$ and truth table $\sage{[1,1,1,0]}$ (in lexicographic order)
is then \verb!a = !$\sage{a}$
\cite[Boolean Functions]{SageMath8418}.
\end{example}

\begin{Definition}
\label{def-Reed-Muller-codes}
\cite{Mul54} \cite[Ch. 13, Section 3]{MacS77} \cite[10.5.2]{Sti07combinatorial}
The \Emph{Reed-Muller code} $RM(r, n)$ consists of those Boolean functions $f : \F_2^n \To \F_2$
whose algebraic normal form has degree $r$.
\end{Definition}
\begin{remark}
Some texts use the notation $\mathcal{R}(r,n)$ or $RM(r,2^n)$ for $RM(r,n)$.

Each Reed-Muller code $RM(r,n)$ is a linear subspace of the vector space of Boolean functions $f : \F_2^n \To \F_2$.

The Reed-Muller code $RM(1,n)$ consists of the $2^{n+1}$ affine functions $f(x) = \langle c, x \rangle + \delta$
for $c \in \F_2^n,\ \delta \in \F_2$ \cite[Ch 14, Section 3]{MacS77} \cite[10.5.2]{Sti07combinatorial}.
\end{remark}


\paragraph*{Bent Boolean functions.}
Bent Boolean functions can be defined in a number of equivalent ways.
The definition used here involves the Walsh Hadamard Transform.
\begin{Definition}
\label{def-Walsh-Hadamard-transform}
\cite[Ch. III, Section 2]{Dil74} \cite[Ch. 2, Section 3]{MacS77}
The Walsh Hada\-mard transform of
a Boolean function $f : \F_2^n \To \F_2$ is
\begin{align*}
W_f(x)
&:=
\sum_{y \in \F_2^n} (-1)^{f(y) + \langle x, y \rangle}
\intertext{where}
\langle x , y \rangle
&:= \sum_{i=0}^{n-1} x_i y_i.
\end{align*}
\end{Definition}
\begin{example}
Using \verb!sage.crypto.boolean_function! in Sage, define:
\begin{sageblock}
from sage.crypto.boolean_function import BooleanFunction
f = BooleanFunction([1,1,1,0])
w = f.walsh_hadamard_transform()
\end{sageblock}
\begin{sagesilent}
from sage.misc.banner import require_version
\end{sagesilent}

The Walsh Hada\-mard transform of the Boolean function \verb!f! on $\F_2^2$ is then
\verb!w =! $\sage{w if require_version(8,2,0) else tuple(-v for v in w)}$
as a truth table in lexicographic order
\cite[Boolean Functions]{SageMath8418}.
\end{example}
\begin{remarks}~

\begin{enumerate}
\item
In versions of Sage before 8.2 the \verb!walsh_hadamard_transform! method has an incorrect sign
[Sage trac ticket \#23931].
\item
For Boolean functions $f : \F_{2^n} \To \F_2$,
where $\F_{2^n}$ is the finite field on $2^n$ elements,
the \emph{trace form} \cite[3.1]{Jac64}

is used to define the Walsh Hada\-mard transform.
\end{enumerate}
\end{remarks}

\begin{Definition}
\label{def-Bent-function}
A Boolean function $f : \F_2^{2m} \To \F_2$ is \Emph{bent}
if and only if its Walsh Hada\-mard transform has constant absolute value $2^{m}$ \cite[p. 74]{Dil74}
\cite[p. 300]{Rot76}.
\end{Definition}
\begin{example}
The Boolean function \verb!f! in the previous example is bent, since its
Walsh Hada\-mard transform has constant absolute value $\sage{abs(w[0])}$.
\end{example}

The remainder of this paper refers to bent Boolean functions simply as bent functions.
\begin{remarks}~

\begin{enumerate}
\item
Bent functions can also be characterized as those Boolean functions whose Hamming distance
from any affine Boolean function is the maximum possible \cite[Ch. 14 Theorem 6]{MacS77} \cite[Theorem 3.3]{MeiS90}.
\item
The property of being a bent function is invariant with respect to the non-degenerate symmetric bilinear form used to define
the Walsh Hada\-mard transform.
That is, for a non-degenerate symmetric bilinear form $\langle x, S y \rangle$, where $S$ is a non-degenerate symmetric matrix in $\F_2^{n \times n}$,
and a Boolean function $f : \F_{2^n} \To \F_2$, the Walsh Hada\-mard transform
\begin{align*}
W[S]_f(x)
&:=
\sum_{y \in \F_2^n} (-1)^{f(y) + \langle x, S y \rangle}
\end{align*}
has constant absolute value $2^{m}$ if and only if $f$ is bent as per Definition \ref{def-Bent-function} \cite{DinMTX2018cyclic}.
In this sense, given an isomorphism between $\F_2^{2m}$ and $\F_{2^{2m}}$ as vector spaces over $\F_2$,
the definition of a bent function on $\F_2^{2m}$ is equivalent to the definition on $\F_{2^{2m}}$.
\item
The fact that any symmetric matrix $S$ on $\F_2^n$ can be factorized $S=L^T L$, with the shape of $L$ depending on the rank of $S$ \cite{Lem1975matrix}
leads to the existence of a basis of $\F_{2^n}$ for which the trace form is diagonal.
This yields an explicit isomorphism between $\F_2^{2m}$ and $\F_{2^{2m}}$ as vector spaces over $\F_2$ for which the two equivalent definitions
of bent function coincide \cite[p. 860]{OlSW1975}.
\end{enumerate}
\end{remarks}

The characterization of bent functions given by Definition~\ref{def-Bent-function} immediately
implies the existence of dual functions:
\begin{Definition}
\label{def-dual-Bent-function}
For a bent function $f : \F_2^{2m} \To \F_2$, the function $\dual{f}$, defined by
\begin{align*}
(-1)^{\dual{f}(x)} &:= 2^{-m} W_f(x)
\end{align*}
is called the \Emph{dual} of $f$ \cite{CarDPS10self}.
\end{Definition}

\begin{Remark}
The function $\dual{f}$ is also a bent function on $\F_2^{2m}$ \cite[p. 427]{MacS77} \cite[p. 301]{Rot76}.
\end{Remark}

\subsection{Weights and weight classes}
\begin{Definition}
\label{def-weight}
The \Emph{Hamming weight} of a Boolean function is the cardinality of its \Emph{support} \cite[p. 8]{MacS77}.
For $f$ on $\F_2^n$
\begin{align*}
\support{f} &:= \{x \in \F_2^n \mid f(x)=1 \}, \quad \weight{f} := \abs{ \support{f} }.
\end{align*}
\end{Definition}

The remainder of this paper refers to Hamming weights simply as weights.

Since a bent function of a given dimension can have only one of two weights,
the weights can be used to define equivalence classes of bent functions 
here called \emph{weight classes}.
\begin{Definition}
\label{def-weight-class}
A bent function $f$ on $\F_2^{2m}$ has weight \cite[Theorem 6.2.10]{Dil74}
\begin{align*}
\weight{f} &= 2^{2 m - 1} - 2^{m-1} \quad (\text{weight class number~} \weightclass{f}=0),
\text{~or}
\\
\weight{f} &= 2^{2 m - 1} + 2^{m-1} \quad (\text{weight class number~} \weightclass{f}=1).
\end{align*}
\end{Definition}

\paragraph*{Weight classes and dual bent functions.}

We now note a connection between weight classes and dual bent functions that
makes it a little easier to reason about dual bent functions.
The following lemma expresses the dual bent function in terms of weight classes.
(See also MacWilliams and Sloane \cite[p. 414]{MacS77}.)
\begin{Lemma}
\label{lm-notes-9b}
For a bent function $f : \F_2^{2m} \To \F_2$, and $x \in \F_2^{2m}$,
\begin{align*}
\dual{f}(x)
&=
\weightclass{y \mapsto f(y) + \langle x, y \rangle}.
\end{align*}

\end{Lemma}

The proof of Lemma~\ref{lm-notes-9b} relies on the following lemma about weight classes.
\begin{Lemma}
\label{lm-notes-9a}
For a bent function $f : \F_2^{2m} \To \F_2$,
\begin{align*}
\weightclass{f}
&=
2^{-m} \weight{f} - 2^{m-1} + 2^{-1},
\intertext{so that}
\weight{f}
&=
2^{m} \weightclass{f} + 2^{2m-1} - 2^{m-1}.
\end{align*}

\end{Lemma}

\begin{proof}
If $\weight{f} = 2^{2 m - 1} - 2^{m-1}$ then
\begin{align*}
2^{-m} \weight{f} - 2^{m-1} + 2^{-1}
&=
2^{-m} (2^{2 m - 1} - 2^{m-1}) - 2^{m-1} + 2^{-1} = 0.
\end{align*}
If $\weight{f} = 2^{2 m - 1} + 2^{m-1}$ then
\begin{align*}
2^{-m} \weight{f} - 2^{m-1} + 2^{-1}
&=
2^{-m} (2^{2 m - 1} + 2^{m-1}) - 2^{m-1} + 2^{-1} = 1.
\end{align*}
\end{proof}

\begin{proofof}{Lemma~\ref{lm-notes-9b}}
Let $h(y) := y \mapsto f(y) + \langle x, y \rangle.$
Then
\begin{align*}
(-1)^{\dual{f}(x)}
&=
2^{-m} \sum_{y \in \F_2^{2m}} (-1)^{f(y) + \langle x, y \rangle}
\\
&=
2^{-m} \left( \sum_{f(y) + \langle x, y \rangle = 0} 1 - \sum_{f(y) + \langle x, y \rangle = 1} 1
\right)
\\
&=
2^{-m} \left( 2^{2m} - 2 \weight{h} \right)
=
2^m - 2^{1-m} \weight{h}
\\
&=
2^m - 2^{1-m} (2^{m} \weightclass{h} + 2^{2m-1} - 2^{m-1})
\\
&=
1 - 2 \weightclass{h} = (-1)^{\weightclass{h}},
\end{align*}
where we have used Lemma~\ref{lm-notes-9a}.
\end{proofof}

%
%

\section{Bent functions and strongly regular graphs}
\label{sec-Bent-graphs}
This section defines the Cayley graph of a Boolean function,
and explores the relationships between bent functions, projective two-weight codes, and strongly regular graphs.
\subsection{The Cayley graph of a Bent function}

The Cayley graph of a bent function $f$ with $f(0)=0$ is defined
in terms of the Cayley graph for a general Boolean function with $f$ with $f(0)=0$.
\paragraph*{The Cayley graph of a Boolean function.}
\begin{Definition}
\label{def-Cayley-graph}
For a Boolean function $f : \F_2^{2 m} \To \F_2$, with $f(0)=0$ we consider the simple undirected
\emph{Cayley graph} $\Cay{f}$  \cite[3.1]{BerC99}
where the vertex set $V(\Cay{f}) = \F_2^{2 m}$ and for $i,j \in \F_2^{2 m}$, the edge $(i,j)$ is in
the edge set $E(\Cay{f})$ if and only if $f(i+j)=1$.
\end{Definition}
Note especially that in contrast with the paper of Bernasconi and Codenotti \cite{BerC99},
this paper defines Cayley graphs only for Boolean functions $f$ with $f(0)=0$,
since the use of Definition~\ref{def-Cayley-graph} with a function $f$ for which $f(0)=1$ would
result in a graph with loops rather than a simple graph.

\paragraph*{Bent functions and strongly regular graphs.}
We repeat below in Proposition~\ref{pr-Cayley-bent-strongly-regular}
the result of Bernasconi and Codenotti \cite{BerC99}
that the Cayley graph of a bent function is strongly regular.
The following definition is used to fix the notation used in this paper.
\begin{Definition}
\label{def-strongly-regular-graph}
A simple graph $\Gamma$ of order $v$ is \Emph{strongly regular} \cite{Bos63,BroCN89,Sei79} with
parameters
$(v,k,\lambda,\mu)$ if
\begin{itemize}
 \item
each vertex has degree $k,$
 \item
each adjacent pair of vertices has $\lambda$ common neighbours, and
\item
each nonadjacent pair of vertices has $\mu$ common neighbours.
\end{itemize}
\end{Definition}
%


The following proposition summarizes some of the well-known properties of the Cayley graphs of bent functions.
\begin{Proposition}
\label{pr-Cayley-bent-strongly-regular}
The Cayley graph $\Cay{f}$ of a bent function $f$ on $\F_2^{2m}$
with $f(0)=0$ is a strongly regular graph with $\lambda = \mu$ \cite[Lemma 12]{BerC99}.

In addition, any Boolean function $f$ on $\F_2^{2m}$ with $f(0)=0$,
whose Cayley graph $\Cay{f}$ is a strongly regular graph with $\lambda = \mu$ is a bent function
\cite[Theorem 3]{BerCV01} \cite[Theorem 3.1]{Sta07}.

For a bent function $f$ on $\F_2^{2m}$,
the parameters of $\Cay{f}$ as a strongly regular graph
are \cite[Theorem 6.2.10]{Dil74} \cite[Theorem 3.2]{HuaY04}
\begin{align*}
(v,k,\lambda,\mu) = &(4^m, 2^{2 m - 1} - 2^{m-1}, 2^{2 m - 2} - 2^{m-1}, 2^{2 m - 2} - 2^{m-1})
\\
  \text{or} \quad &(4^m, 2^{2 m - 1} + 2^{m-1}, 2^{2 m - 2} + 2^{m-1}, 2^{2 m - 2} + 2^{m-1}).
\end{align*}
\end{Proposition}

%
\subsection{Bent functions, linear codes and strongly regular graphs}
Another well known way to obtain a strongly regular graph from a bent function is via a projective two-weight
code.
This is done via the following definitions.
\paragraph*{Projective two-weight binary codes.}

\begin{Definition}
\label{def-two-weight-codes}
\cite{BouFFWW2006, CalK1986, Del72weights, Din2015, Ton96uniformly}

A \Emph{two-weight binary code} with parameters $[n,k,d]$ is a $k$ dimensional subspace of $\F_2^n$
with
minimum Hamming distance $d$, such that the set of Hamming weights of the non-zero vectors has size
2.

Bouyukliev, Fack, Willems and Winne \cite[p. 60]{BouFFWW2006} define projective codes as follows.
``A \Emph{generator matrix} $G$ of a linear code $[n, k]$ code $C$ is any matrix
of rank $k$ (over $\F_2$) with rows from $C.$ \ldots
A linear $[n, k]$ code is called \Emph{projective} if no two columns of a generator matrix
$G$ are linearly dependent, i.e., if the columns of $G$ are pairwise different points in a
projective $(k-1)$-dimensional space.''

A \Emph{projective two-weight binary code} with parameters $[n, k, d]$ is thus a
two-weight binary code with these parameters which is also projective as an $[n, k]$ linear code.
%
%
%
\end{Definition}

\begin{Remark}
In the case of $\F_2$, no two columns of the generator matrix of a projective code are equal.
\end{Remark}

\paragraph*{From bent function to strongly regular graph via a projective two-weight code.}
There is a standard construction for obtaining a projective two-weight code from a bent function in such a way that
the code can be used to define a strongly regular graph.
This method is equivalent to the following definition.
\begin{Definition}
\label{def-boolean-linear-code}
\cite{Din2015}.
See also \cite[Definition 2A]{CalK1986}

Let $v=2^{2m}$, and let $X$ be the matrix in  $\F_2^{2m \times v}$
whose columns are the $v$ elements of $\F_2^{2m}$ in lexicographic order.
For a Boolean function $f : \F_2^{2m} \To \F_2$,
let $n=\weight{f}$, and let $Y$ be the matrix in $\F_2^{2m \times n}$
whose columns are the $n$ elements of $\support{f}$ in lexicographic order.
That is,
\begin{align*}
Y_{i,j} &= (y_j)_i, \quad \text{for~} i \in \{0,\ldots,v-1\},\ j \in \{0,\ldots,n-1\},
\end{align*}
with $y_j \in \F_2^{2m}$ being an element of $\support{f}$.

The linear code $C(f) \subset \F_2^n$ is then given by the span of the $2 m$ rows of $Y$ within $\F_2^n$,
considered as row vectors.
\end{Definition}
%

\begin{example}
In Sage, using \verb!sage.crypto.boolean_function! and
\newline
\verb!boolean_cayley_graphs.boolean_linear_code!, define:
\begin{sageblock}
from sage.crypto.boolean_function import BooleanFunction
from boolean_cayley_graphs.boolean_linear_code import (
     boolean_linear_code)
f = BooleanFunction([0,1,0,0,0,1,0,0,1,0,0,0,0,0,1,1])
dim = f.nvariables()
Cf = boolean_linear_code(dim, f)
\end{sageblock}
The linear code $Cf := C(f)$ of the Boolean function \verb!f! on $\F_2^4$ then has the following
generator matrix in echelon form:
\begin{align*}
\sage{Cf.generator_matrix().echelon_form()}.
\end{align*}
\cite{Leo16GitHub} \cite[Boolean Functions]{SageMath8418}.
\end{example}

\begin{remarks}~

\begin{enumerate}
 \item
The linear span of the rows of $Y$ is given by the set of rows of the matrix
\begin{align*}
M &:= X^T Y, \quad \text{where}
\\
X &\in \F_2^{2m \times v}
\end{align*}
is the matrix whose columns are all $v$ elements of $\F_2^{2m}$ in lexicographic order.
Thus $M \in \F_2^{v \times n}$.
 \item
Definition \ref{def-boolean-linear-code} is not identical to the generic construction described by
Ding \cite[Section III]{Din2015} since that definition is for
subsets $D \subset \F_v$, and uses the trace form, but for the case where $D$ is the support
of a Boolean function $f \colon \F_v \To F_2$, the two definitions are equivalent.

If the bilinear form $\langle x, y \rangle = x^T y$ is replaced in
the previous remark by a non-degenerate symmetric bilinear form $\langle x, S y \rangle = x^T S y$,
where $S$ is a non-degenerate symmetric matrix in $\F_2^{2m \times 2m}$,
then this would produce the same linear code, since the matrix $S$, being invertible,
produces a bijection on $\F_2^{2m}$ and therefore the matrix $X^T S$ differs from $X^T$ only by a permutation of rows.

Thus by similar arguments to those in Remarks 2 and 3 following Definition \ref{def-Bent-function},
the generic construction described by Ding \cite[Section III]{Din2015} is in this case equivalent
to Definition \ref{def-boolean-linear-code}.
\end{enumerate}
\end{remarks}

When $f$ is a bent function, the linear code $C(f)$
described by Definition \ref{def-boolean-linear-code} has the following properties.
%
\begin{Proposition}
\label{pr-bent-two-weight-code}
\cite[Corollary 10]{Din2015}

For a bent function $f : \F_2^{2m} \To \F_2$, the linear code $C(f)$
is a projective two-weight binary code, with
\begin{align*}
\begin{cases}
n = \weight{f} = 2^{2m-2} - 2^{m-2}, \quad k = 2m & \text{if~} \weightclass{f}=0.
\\
n = \weight{f} = 2^{2m-2} + 2^{m-2}, \quad k = 2m & \text{if~} \weightclass{f}=1.
\end{cases}
\end{align*}
The possible weights of non-zero code words are:
\begin{align*}
\begin{cases}
2^{2m-2}, 2^{2m-2} - 2^{m-1} & \text{if~} \weightclass{f}=0.
\\
2^{2m-2}, 2^{2m-2} + 2^{m-1} & \text{if~} \weightclass{f}=1.
\end{cases}
\end{align*}
\end{Proposition}

\begin{proof}
We examine $W_f$, the Walsh Hadamard transform of $f$.
\begin{align*}
W_f(x)
&=
\sum_{y \in \F_2^{2 m}} (-1)^{\langle x, y \rangle + f(y)}
=
\sum_{y \in \F_2^{2 m}} (-1)^{\langle x, y \rangle}
- 2\sum_{f(y)=1} (-1)^{\langle x, y \rangle}.
\end{align*}
But
\begin{align*}
\sum_{y \in \F_2^{2 m}} (-1)^{\langle x, y \rangle}
&=
\begin{cases}
4^m &(x=0)
\\
0 & \text{otherwise},
\end{cases}
\end{align*}
as per the Sylvester Hadamard matrices.

So, for $x \neq 0$,
\begin{align*}
W_f(x)
&=
- 2\sum_{f(y)=1} (-1)^{\langle x, y \rangle},
\intertext{so}
\sum_{f(y)=1} (-1)^{\langle x, y \rangle}
&=
\weight{f} - 2 \sum_{\substack{f(y)=1 \\ \langle x, y \rangle =1}} 1
=
- W_f(x)/2.
\intertext{But}
\sum_{\substack{f(y)=1 \\ \langle x, y \rangle =1}} 1
&=
\weight{x^T Y},
\end{align*}
the weight of the code word of $C(f)$ corresponding to $x$.
So
\begin{align*}
\weight{f} - 2 \weight{x^T Y}
&=
- W_f(x)/2,
\intertext{and therefore}
\weight{x^T Y}
&=
\weight{f}/2 + W_f(x)/4.
\end{align*}
We now examine the two possible weight class numbers of $f$.

If $\weightclass{f} = 0$ then $\weight{f} = 2^{2m-1}-2^{m-1}$.
For $x \neq 0$ there are two cases, depending on $\dual{f}(x)$:

If $\dual{f}(x) = 0$ then $W_f(x) = 2^m$, so
\begin{align*}
\weight{x^T Y}
&=
2^{2m-2} - 2^{m-2} + 2^{m-2}
=
2^{2m-2}.
\end{align*}

If $\dual{f}(x) = 1$ then $W_f(x) = -2^m$, so
\begin{align*}
\weight{x^T Y}
&=
2^{2m-2} - 2^{m-2} - 2^{m-2}
=
2^{2m-2} - 2^{m-1}.
\end{align*}

Similarly, if $\weightclass{f} = 1$ then $\weight{f} = 2^{2m-1}+2^{m-1}$,
and so for $x \neq 0$
\begin{align*}
\weight{x^T Y}
&=
\begin{cases}
2^{2m-2} + 2^{m-1} & (\dual{f}(x)=0)
\\
2^{2m-2}           & (\dual{f}(x)=1).
\end{cases}
\end{align*}
As a consequence, all $v = 2^{2m}$ rows of $M = X^T Y$ are distinct,
since the rows of $M$ consist of every linear combination of the rows of $Y$,
every linear combination of the rows of $Y$ is given by some $x^T Y$ where $x$ is a column of $X$,
and for every non-zero $x$, this linear combination has positive weight and is therefore non-zero.
Thus the dimension of $C(f)$ as a subspace of $F_2^n$ is $k=2m$.
\end{proof}

%
\paragraph*{From linear code to strongly regular graph.}
This paper uses the following non-standard definition to obtain a strongly regular graph from a
projective two-weight code.
\begin{Definition}
\label{def-R-f}
Given $f : \F_2^{2m} \To \F_2$, and linear code $C(f)$ defined as per Definition~\ref{def-boolean-linear-code},
define the graph $R(f)$ as follows.

Vertices of $R(f)$ are code words of $C(f)$.

For code words $v,w \in C(f)$, edge $(u,v) \in R(f)$ if and only if
\begin{align*}
\begin{cases}
\weight{u+v} = 2^{2m-2} - 2^{m-1} & (\text{if~}\weightclass{f}=0).
\\
\weight{u+v} = 2^{2m-2} + 2^{m-1} & (\text{if~}\weightclass{f}=1).
\end{cases}
\end{align*}

\end{Definition}
\begin{remarks}~

\begin{enumerate}
 \item
The standard definition uses the lower of the two weights in both cases above.
 \item
Since $C(f)$ is a projective two-weight binary code,
$R(f)$ is a strongly regular graph \cite[Theorem 2]{Del72weights} \cite[Theorem 16.22]{CamVL91}.
\end{enumerate}
\end{remarks}



\paragraph*{The graph $R(f)$ is the Cayley graph of the extended dual.}
The strongly regular graph $R(f)$ of bent function $f$, as per
\ref{def-R-f} has the following remarkable property.

\begin{Theorem}
For a bent function $f : \F_2^{2m} \To \F_2$, with $f(0)=0$,
\begin{align*}
R(f) \equiv \Cay{\dual{f} + \weightclass{f}}.
\end{align*}

\end{Theorem}
\begin{proof}
As a consequence of Lemma~\ref{lm-notes-9b}, $\weightclass{f} = \dual{f}(0)$,
so if $g(x) := \dual{f}(x) + \weightclass{f}$ then $g(0)=0$ and therefore the Cayley graph of $g$
is well defined.

From the proof of Proposition~\ref{pr-bent-two-weight-code}:

If $\weightclass{f} = 0$ then for $x \neq 0$
\begin{align*}
\weight{x^T Y}
&=
\begin{cases}
2^{2m-2}           & (\dual{f}(x)=0)
\\
2^{2m-2} - 2^{m-1} & (\dual{f}(x)=1).
\end{cases}
\end{align*}
Thus for $x \neq 0$, if $\dual{f}(x) + \weightclass{f} = 1$ if and only if $\weight{x^T Y} = 2^{2m-2} - 2^{m-1}$.

If $\weightclass{f} = 1$ then for $x \neq 0$
\begin{align*}
\weight{x^T Y}
&=
\begin{cases}
2^{2m-2} + 2^{m-1} & (\dual{f}(x)=0)
\\
2^{2m-2}           & (\dual{f}(x)=1).
\end{cases}
\end{align*}
Thus for $x \neq 0$, if $\dual{f}(x) + \weightclass{f} = 1$ if and only if $\weight{x^T Y} = 2^{2m-2} + 2^{m-1}$.

Thus in both cases, $R(f)$ as per Definition~\ref{def-R-f} is isomorphic to the Cayley graph $\Cay{\dual{f} + \weightclass{f}}$.
\end{proof}

\section{Equivalence of bent functions}
\label{sec-Equivalence}
The following concepts of equivalence of Boolean functions are used in this paper,
usually in the case where the Boolean functions are bent.

\paragraph*{Extended affine equivalence.}

\begin{Definition}
For Boolean functions $f,g : \F_2^n \To \F_2$,
$f$ is \Emph{extended affine equivalent} to $g$ \cite[Section 1.4]{Tok15bent} if and only if
\begin{align*}
g(x) &= f(A x + b) + \langle c, x \rangle + \delta
\end{align*}
for some $A \in GL(n,2)$, $b, c \in \F_2^n$, $\delta \in \F_2$.
\end{Definition}

The Boolean function $f$ is extended affine equivalent to $g$
if and only if $f$ and $g$ are in the same orbit
of the action of the \Emph{extended general affine group} $EGA(n, 2)$ on $\F_2^{\F_2^n}$, defined as follows.

\begin{Definition}
\begin{align*}
&EGA(n, 2) := \{ (A,b,c,\delta) \mid A \in GL(n,2),\ b, c \in \F_2^n,\ \delta \in \F_2 \}
\intertext{with}
&(A,b,c,\delta)(A',b',c',\delta') := (A A', A b' + b, A'^T c + c', \langle c, b' \rangle + \delta + \delta'),
\end{align*}
with action
\begin{align*}
(A,b,c,\delta)f(x) &:= f(A x + b) + \langle c, x \rangle + \delta,
\\
\left( (A,b,c,\delta)(A',b',c',\delta') f \right)
& := (A',b',c',\delta') \circ (A,b,c,\delta) f
\\
& = (A',b',c',\delta') \left( (A,b,c,\delta) f \right).
\end{align*}
\cite[Section 2]{Mai91}
\end{Definition}

\begin{Proposition}
\label{prop-EA-class-properties}
The extended affine (EA) equivalence classes of the Boolean functions $\F_2^n \To \F_2$,
that is, the orbits of these functions under $EGA(n, 2)$,
have the following well known and easily verified properties.
\begin{enumerate}
 \item
For a given $f \colon \F_2^n \To \F_2$,
the $2^{n+1}$ functions $x \mapsto f(x) + \langle c, x \rangle + \delta$ are all distinct.
Thus the EA equivalence class of $f$ consists of some number of complete cosets of the Reed-Muller code $RM(1,n)$
described in Section \ref{sec-Boolean-functions}.
 \item
Each general affine transformation $(A,b)f(x) := f(A x + b)$ preserves cosets of the Reed-Muller code $RM(1,n)$
in the sense that $(A,b)$ maps $f + RM(1,n)$ to $g + RM(1,n)$ where $g(x) = f(A x + b)$.
\end{enumerate}
\end{Proposition}
See also MacWilliams and Sloane \cite[Ch. 13]{MacS77}, and Maiorana \cite{Mai91}.
\paragraph*{General linear equivalence.}
\begin{Definition}
For Boolean functions $f,g : \F_2^n \To \F_2$,
$f$ is \Emph{general linear equivalent} to $g$ if and only if
\begin{align*}
g(x) &= f(A x)
\end{align*}
for some $A \in GL(n,2)$.
\end{Definition}

Thus $f$ is general linear equivalent to $g$
if and only if $f$ and $g$ are in the same orbit
of the action of the general linear group $GL(n, 2)$ on $\F_2^{\F_2^n}$, defined as follows.

\begin{Definition}
\begin{align*}
A f(x) &:= f(A x),
\\
(A A') f & := A' \circ A f = A' (A f).
\end{align*}
\end{Definition}
Some references for the study of the general linear equivalence of Boolean functions
include Harrison \cite{Har64}, Comerford \cite{Com80}, and Maiorana \cite[Section 2]{Mai91}.
\paragraph*{Extended translation equivalence.}

\begin{Definition}
For Boolean functions $f,g : \F_2^n \To \F_2$,
$f$ is \Emph{extended translation equivalent} to $g$ if and only if
\begin{align*}
g(x) &= f(x + b) + \langle c, x \rangle + \delta
\end{align*}
for $b, c \in \F_2^n$, $\delta \in \F_2$.
\end{Definition}

Thus $f$ is extended translation equivalent to $g$
if and only if $f$ and $g$ are in the same orbit
of the action of the \Emph{extended translation group} $ET(n, 2)$ on $\F_2^{\F_2^n}$, defined as follows.

\begin{Definition}
\begin{align*}
&ET(n, 2) := \{ (b,c,\delta) \mid \ b, c \in \F_2^n,\ \delta \in \F_2 \}
\intertext{with}
&(b,c,\delta)(b',c',\delta') := (b' + b, c + c', \langle c, b' \rangle + \delta + \delta'),
\end{align*}
with action
\begin{align*}
(b,c,\delta)f(x) &:= f(x + b) + \langle c, x \rangle + \delta,
\\
\left( (b,c,\delta)(b',c',\delta') \right) f
& := (b',c',\delta') \circ (b,c,\delta) f
\\
& = (b',c',\delta') \left( (b,c,\delta) f \right).
\end{align*}
\end{Definition}

\paragraph*{Cayley equivalence.}
\begin{Definition}
For Boolean functions $f, g : \F_2^n \To \F_2$, with $f(0)=g(0)=0$,
we call $f$ and $g$ \Emph{Cayley equivalent},
and write $f \equiv g$,
if and only if the graphs $\Cay{f}$ and $\Cay{g}$ are isomorphic.

Equivalently, $f \equiv g$ if and only if
there exists a bijection $\pi : \F_2^n \To \F_2^n$ such that
\begin{align*}
g(x+y) &= f \big(\pi(x)+\pi(y)\big) \quad \text{for all~} x,y \in \F_2^n.
\end{align*}
\end{Definition}
Remark: Note that the bijection $\pi$ is not necessarily linear on $\F_2^n$.
Examples of bent functions $f$ and $g$ where $f \equiv g$ but the bijection is not linear
are given in Section~\ref{sec-Empirical}.
\paragraph*{Extended Cayley equivalence.}
While Bernasconi and Codenotti \cite{BerC99} define Cayley graphs for Boolean functions with
$f(0)=1$ and allow Cayley graphs to have loops, this paper defines Cayley
graphs only for Boolean functions where $f(0)=0$.
This has the disadvantage that Cayley equivalence is an equivalence relation on half
of the Boolean functions rather than all of them.
To extend this equivalence relation to all Boolean functions,
we just declare the functions $f$ and $f+1$ to be ``extended'' Cayley equivalent,
resulting in the following definition.
\begin{Definition}
For Boolean functions $f, g : \F_2^n \To \F_2$,
if there exist $\delta, \epsilon \in \{0,1\}$ such that $f + \delta \equiv g + \epsilon$,
we call $f$ and $g$ \Emph{extended Cayley (EC) equivalent} and write $f \cong g$.
\end{Definition}
Extended Cayley equivalence is thus an equivalence relation on the set of all Boolean functions on
$\F_2^n$.
It is easy to verify that $f \cong g$ if and only if $f+f(0) \equiv g+g(0)$.


\subsection{Relationships between different concepts of equivalence}

As stated in the Introduction, in order to classify bent functions by their Cayley graphs,
it helps to understand the relationship between Cayley equivalence and other concepts of equivalence
of Boolean functions, especially if this helps to cut down the search space needed for the classification.
This section lists a few of these useful relationships.


\paragraph*{General linear equivalence implies Cayley equivalence.}

Firstly, general linear equivalence of Boolean functions implies Cayley equivalence.
Specifically, the following result applies.
\begin{Theorem}
\label{th-Linear-Cayley}
If $f$ is a Boolean function with $f(0)=0$ and $g(x) := f(A x)$ where $A \in GL(n,2)$,
then $f \equiv g$.
\end{Theorem}
\begin{proof}
\begin{align*}
g(x+y) &= f\big(A(x+y)\big) = f(A x + A y)\quad \text{for all~} x,y \in \F_2^n.
\end{align*}
\end{proof}
Thus, for bent functions, the following result holds.
\begin{Corollary}
\label{corr-bent-Linear-Cayley}
If $f$ is bent with $f(0)=0$ and $g(x) := f(A x)$ where $A \in GL(n,2)$,
then $g$ is bent with $g(0)=0$ and $f \equiv g$.
\end{Corollary}
Thus if $f$ is bent with $f(0)=0$, and $g$ is bent with $g(0)=0$, and $f \not\equiv g$,
then $f$ is not general linear equivalent to $g$.
This result immediately leads to another corollary.
Here, and later in this paper, we make use of the following terminology.
\begin{Definition}
A Boolean function $f \colon \F_2^n \To \F_2$ is said to be \Emph{prolific} if
there is no pair $b, c \in \F_2^n$ with $g(x) = f(x+b) + \langle c, x \rangle + f(b)$ such that $f \cong g$.
Thus the number of extended Cayley classes in the extended translation class of a prolific Boolean function
is $4^n$.
\end{Definition}

\begin{Corollary}
 \label{corr-no-Cayley-no-Linear}
If $f \colon \F^{2m} \To \F_2$ is bent with $f(0)=0$, and $f$ is prolific, then there is no triple
$A \in GL(2m,2)$, $b, c \in \F_2^{2m}$ with $f(Ax) = f(x+b) + \langle c, x \rangle + f(b)$.
\end{Corollary}


\paragraph*{Extended affine, translation, and Cayley equivalence.}

Secondly, if $f$ is a Boolean function,
and $h$ is a Boolean function $h$ that is extended affine equivalent to $f$,
then a Boolean function $g$ exists that is general linear equivalent to $h$
and extended translation equivalent to $f$:
\begin{Theorem}
\label{th-Affine-Translate-Linear}
For $A \in GL(n,2)$, $b, c \in \F_2^n$, $\delta \in \F_2$,
$f : \F_2^n \To \F_2$,
the function
\begin{align*}
h(x) &:= f(A x + b) + \langle c, x \rangle + \delta
\intertext{can be expressed as $h(x) = g(A x)$ where}
g(x) &:= f(x+b) + \langle (A^{-1})^T c, x \rangle + \delta.
\end{align*}
\end{Theorem}
\begin{proof}
Let $y:= A x$. Then
\begin{align*}
g(A x) = g(y)
&= f(y+b) + \langle (A^{-1})^T c, y \rangle + \delta
\\
&= f(y+b) + \langle c, A^{-1} y \rangle + \delta
\\
&= f(A x + b) + \langle c, x \rangle + \delta = h(x).
\end{align*}
\end{proof}

\begin{Corollary}
\label{corr-Affine-Translate-Cayley}
If $f$ is a bent Boolean function,
and a bent function $h$ is extended affine equivalent to $f$,
then a bent function $g$ can be found that is extended Cayley equivalent to $h$
and extended translation equivalent to $f$.
\end{Corollary}
\begin{proof}
Let $f$, $g$, and $h$ be as per Theorem \ref{th-Affine-Translate-Linear}.
If $f$ is bent, then so are $g$ and $h$.
Since, by Theorem \ref{th-Affine-Translate-Linear}, $g$ is general linear equivalent to $h$,
by Theorem \ref{th-Linear-Cayley}, $g$ is extended Cayley equivalent to $h$.
\end{proof}

As a consequence, to determine which strongly regular graphs occur, corresponding to each
extended Cayley equivalence classes within the extended affine
equivalence class of a bent function $f : \F_2^{2m} \To \F_2$ with $f(0)=0$,
we need only examine the extended translation equivalent functions of the form
\begin{align*}
f(x+b) + \langle c, x \rangle + f(b),
\end{align*}
for each $b, c \in \F_2^{2m}$.
This cuts down the required search space considerably.


\paragraph*{Quadratic bent functions have only two extended Cayley classes.}
Finally, in the case of quadratic bent functions, there is a complete classification in terms of
weight classes.
\begin{Theorem}
\label{th-Quadratic-Classes}
For each $m>0$, the extended affine equivalence class of quadratic bent functions
$q : \F_2^{2m} \To \F_2$ contains exactly two extended Cayley equivalence classes,
corresponding to the two possible weight classes of
$x \mapsto q(x+b) + \langle c, x \rangle + q(b)$.
\end{Theorem}

The proof of this theorem is given in Appendix~\ref{app-proof-of}.

\subsection{Relationships between duality of bent functions and different concepts of equivalence}

The following propositions are based on well known results,
but are useful in understanding the relationship
between the duality of bent functions and various concepts of equivalence.

Firstly, general linear equivalence of bent functions $f$ and $g$
implies general linear equivalence of their duals, $\dual{f}$ and $\dual{g}$,
which implies Cayley equivalence of $\dual{f}$ and $\dual{g}$.
\begin{Proposition}
\label{prop-dual-linear-equivalence}
\cite[Remark 6.2.7]{Dil74}

For a bent function $f : \F_2^{2m} \To \F_2$, and $A \in GL(2 m, 2)$, if
\begin{align*}
g(x) &:= f(A x)
\intertext{then}
\dual{g}(x) &= \dual{f}\big((A^T)^{-1} x \big),
\end{align*}
and therefore by Theorem \ref{th-Linear-Cayley}, $\dual{g} \equiv \dual{f}$.

If, in addition, $f=\dual{f}$ then $\dual{g} \equiv g$.
\end{Proposition}

Remark: Functions of the form
\begin{align*}
f(x) := \sum_{k=0}^{m-1} x_{2k} x_{2k+1}
\end{align*}
are self dual bent functions, $f=\dual{f}$ \cite[Remark 6.3.2]{Dil74}.
There are many other self dual bent functions \cite{CarDPS10self,FeuSSW2013}.

Secondly, the following proposition displays a relationship between the extended translation
class of a bent function $f$, and that of its dual $\dual{f}$.
\begin{Proposition}
\label{prop-dual-affine-equivalence}
\cite[Remark 6.2.7]{Dil74} \cite[Proposition 8.7]{Car10boolean}.
%

For a bent function $f$ on $\F_2^{2m}$, and $b,c \in \F_2^{2m}$,
if
\begin{align*}
g(x) &:= f(x+b) + \langle c, x \rangle
\intertext{then}
\dual{g}(x) &= \dual{f}(x+c) + \langle b, x \rangle + \langle b, c \rangle.
\end{align*}
\end{Proposition}

This result has an implication for the relationship between the set of bent functions within
an extended translation (ET) equivalence class, and the set of their duals.
Recall that a bent function is not necessarily extended affine (EA) equivalent to its dual
\cite{LanLM08Kasami}.
The following ``all or nothing'' property holds within an extended translation equivalence class of bent functions.
\begin{Corollary}
\label{cor-dual-ET-EC}
For bent functions $f, g$ on $\F_2^{2m}$,
if $f$ is EA equivalent to $\dual{f}$ and $g$ is ET equivalent to $f$,
then $\dual{g}$ is EA equivalent to $g$.
Thus, by Corollary~\ref{corr-Affine-Translate-Cayley},
the set of isomorphism classes of Cayley graphs of the \Emph{duals} of the bent functions in
the ET class of $f$ equals the set of isomorphism classes of Cayley graphs of
the bent functions themselves.

Conversely, for a bent function $f$ on $\F_2^{2m}$,
if there is any bent function $g$ that is ET equivalent to $f$,
such that $\dual{g}$ is not EA equivalent to $g$, then no bent function in the ET class is EA
equivalent to its dual, including $f$ itself.
\end{Corollary}


\section{Bent functions and block designs}
\label{sec-Bent-designs}

This section examines the relationships between bent functions and symmetric block designs.



%
\subsection{The two block designs of a bent function}

The first block design of a bent function $f$ on $\F_2^{2m}$ is obtained by interpreting
the adjacency matrix of $\Cay{f}$ as the incidence matrix of a block design.
In this case we do not need $f(0)=0$ \cite[p. 160]{DilS87block}.

The second block design of a bent function $f$ involves the
\Emph{symmetric difference property}, which was first investigated by Kantor
\cite[Section 5]{Kan75symplectic}.
\begin{Definition}
\label{def-Symmetric-difference-property}
\cite[p. 49]{Kan75symplectic}.

A symmetric block design $\mathcal{D}$ has the symmetric difference property (SDP)
if, for any three blocks, $B, C, D$ of $\mathcal{D},$ the symmetric difference
$B \bigtriangleup C \bigtriangleup D$ is either a block or the complement of a block.
\end{Definition}

This second block design is defined as follows.
\begin{Definition}
\label{def-SDP-design}
For a bent function  $f$ on $\F_2^{2m}$, define the matrix $M_D(f) \in \F_2^{2^{2m} \times 2^{2m}}$ where
\begin{align}
M_D(f)_{c,x} &:= f(x) + \langle c, x \rangle + \dual{f}(c),
\label{D-f-def}
\end{align}
and use it as the incidence matrix of a symmetric block design, which
we call it the \emph{SDP design} of $f$.
\end{Definition}

Kantor describes the special case where $f$ is quadratic
\cite[Section 5]{Kan75symplectic},
and Dillon and Schatz \cite{DilS87block} describe the general case.
See also Cameron and van Lint \cite[pp. 77-78 and Ex. 13, p. 152]{CamVL91}.

The following properties of SDP designs of bent functions are well known.
\begin{Proposition}
\label{prop-SDP-design}
\cite[p. 160]{DilS87block} \cite[Theorem 3.29]{Neu06bent}

For any bent function $f$ on $\F_2^{2m}$, the SDP design of $f$ has the symmetric difference property.
\end{Proposition}

\begin{Proposition}
\label{prop-SDP-design-affine-equivalence}
\cite[p. 161]{DilS87block} \cite{Kan83exponential}

For bent functions $f, g$ on $\F_2^{2m}$,
the two SDP designs $D(f)$ and $D(g)$ are isomorphic as symmetric block designs
if and only if $f$ and $g$ are affine equivalent.
\end{Proposition}

%
%

\paragraph*{Weight classes and the SDP design matrix.}


Definition~\ref{def-SDP-design} is different from
but equivalent to the one given by Dillon and Schatz \cite[p. 160]{DilS87block}:
\begin{Lemma}
\label{lm-SDP-design-rows}
\cite[3.29]{Neu06bent}

For any bent function $f$ on $\F_2^{2m}$, the rows of the incidence matrix $M_D(f)$
are given by the words of minimum weight in the code spanned by the support of $f$ and the Reed-Muller code $RM(1,2m)$.
\end{Lemma}
(Here we have used an ordering of the elements of $\F_2^{2m}$ to define an ordering of the columns of the incidence matrix.)


\begin{proof}
Firstly, as mentioned in Section \ref{sec-Boolean-functions},
the Reed-Muller code $RM(1,2m)$ consists of the words spanned by the affine functions on $Z_2^{2m}$.
Thus, the incidence matrix $M_{RM(1,2m)}$ is defined by
\begin{align*}
{M_{RM(1,2m)}}_{c,x} &:= \langle c, x \rangle + d,
\end{align*}
where $d \in \F_2$.

Therefore the incidence matrix of the code spanned by the support of $f$ and $RM(1,2m)$ is defined by
\begin{align*}
{M_{f,RM(1,2m)}}_{c,x} &:= f(x) + \langle c, x \rangle + d.
\end{align*}
Finally, from Lemma~\ref{lm-notes-9b} we know that
\begin{align*}
\weightclass{x \mapsto f(x) + \langle c, x \rangle}
&=
\dual{f}(c),
\intertext{so that}
\weightclass{x \mapsto f(x) + \langle c, x \rangle + \dual{f}(c)}
&=
0.
\end{align*}
\end{proof}

The following characterization of the SDP design of a bent function $f$ also relies on
Lemma~\ref{lm-notes-9b} for its proof.
We first define the matrix of weight classes corresponding to the extended translation class of $f$.
\begin{Definition}
\label{def-weight-class-matrix}

For a bent function $f : \F_2^{2m} \To \F_2$,
define the \Emph{weight class matrix} of $f$ by
\begin{align*}
M_{wc}(f)_{c,b}
&:=
\weightclass{x \mapsto f(x+b) + \langle c, x \rangle + f(b)}
\end{align*}
for $b,c \in \F_2^{2m}$.
\end{Definition}

\begin{Theorem}
\label{th-Dillon-Schatz}
The weight class matrix of $f$ as given by Definition~\ref{def-weight-class-matrix}
equals the incidence matrix of the SDP design of $f$.
Specifically,
\begin{align*}
M_{wc}(f)_{c,b}
&=
f(b) + \langle c, b \rangle + \dual{f}(c)
\\
&=
M_D(f)_{c,b},
\end{align*}
where $M_D(f)$ is defined by \eqref{D-f-def}.
\end{Theorem}

\begin{proof}
Let $g(x) := f(x+b) + \langle c, x \rangle + f(b)$.
Then by change of variable $y:=x+b$,
\begin{align*}
\weightclass{g}
&=
\weightclass{y \mapsto f(y) + \langle c, y \rangle + \langle c, b \rangle + f(b)}
\\
&=
\weightclass{y \mapsto f(y) + \langle c, y \rangle} + \langle c, b \rangle + f(b)
\\
&=
\dual{f}(c) + \langle c, b \rangle + f(b),
\end{align*}
as a consequence of Lemma~\ref{lm-notes-9b}.
\end{proof}

\section{SageMath and CoCalc code}
\label{sec-Code}
The computational results listed in this paper were obtained by the
use of code written in Sage \cite{JoyEtAl13Sage} \cite{SageMath7517} and Python.
This code base is called \texttt{Boolean-Cayley-graphs} and it is available both as a GitHub
repository \cite{Leo16GitHub} and as a public CoCalc \cite{CoCalc} folder
\cite{Leo17CoCalc}. 
For an introduction to other aspects of coding theory and cryptography in Sage,
see the article by Joyner et al. \cite{JoyEtAl13Sage}.

\paragraph*{Description of the Sage code.}

This section contains a brief description of some of the code included in
\texttt{Boolean-Cayley-graphs}.
More detailed documentation is being developed and this is intended to be included as part of the code base.
The code itself is subject to review and revision, and may change as a result of the advice of
those more experienced with Sage code.
The description in this section applies to the code base as it exists in 2018.

The code base is structured as a set of Sage script files. 
These in turn use Python scripts, found in a subdirectory called \texttt{Boolean\_Cayley\_graphs}.
The Python code is used to define a number of useful Python classes.

The key class is \texttt{BentFunctionCayleyGraphClassification}.
This class is used to store the classification of Cayley graphs within the extended translation
class of a given bent function $f$, as well as the classification of Cayley graphs of the duals of
each function in the extended translation class.
The class therefore contains the algebraic normal form of the given bent function,
a list of graphs
stored as strings obtained via the \texttt{graph6\_string} \cite{McKP13nauty}
method of the \texttt{Graph} class, and two matrices,
used to store the list indices corresponding to
the Cayley graph for each bent function in the extended translation class, and the dual of each bent
function, respectively.
The class also contains the weight class matrix
corresponding to the given bent function.

The class is initialized by enumerating the bent functions of the form
$x \mapsto f(x+b) + \langle c, x \rangle + f(b)$,
and determining the Cayley graph of each.
For each Cayley graph, the \texttt{Graph} method \texttt{canonical\_label} is used
to invoke the Bliss package \cite{JunK07Bliss,JunK11conflict} to calculate the canonical label
of the graph, and then \texttt{graph6\_string} is used to obtain a string.
Each new graph is compared for isomorphism to each of the graphs in the current list,
by simply comparing the string against each of the existing strings.
If the new graph is not isomorphic to any existing graph, it is added to the list.
Each list of pairwise non-isomorphic graphs can be checked by a function called \texttt{check\_graph\_class\_list}
which uses the Nauty package to check the non-isomorphism \cite{McKP13nauty,McKP14practical}.

It is the efficiency of the Bliss canonical labelling algorithm, and the speed of its implementation, that makes this approach feasible.
Even so, for an 8 dimensional bent function, the initialization of its Cayley graph classification
can take more than 24 hours on an Intel\textregistered Core\texttrademark i7 CPU 870 running at 2.93 GHz.
For this reason, each computed classification is saved, and a class method (\texttt{load\_mangled})
is provided to load existing saved classifications.

\paragraph*{History of the Sage code.}

The Sage code originated in 2015 as a series of worksheets on Sage\-Math\-Cloud (now CoCalc).
While these were useful for investigating extended Cayley classes for bent functions in up to 6 dimensions,
they were too slow to use for bent functions in 8 dimensions.

The \texttt{Boolean-Cayley-graphs} GitHub project \cite{Leo16GitHub} and public Sage\-Math\-Cloud folder \cite{Leo17CoCalc} were begun in 2016
with the intention of refactoring the code to make it fast enough to use for bent functions in 8 dimensions
up to degree 3.
The use of canonical labelling made this possible.

Further improvements were made in 2017 to enable the classification of any bent function in 8 dimensions or less
to be computed in a reasonable time on a commodity personal computer.
In late 2017, code was added so that the Cayley graph classifications could be accessed via a relational database \cite{Leo18Database},
with implementations using SQLite3 \cite{SQLite} and PostgreSQL \cite{PostgreSQL}.
Also, parallel versions of the classification functions were written using MPI4Py,
and used on the NCI Raijin supercomputer to complete the classifications for CAST-128 and compute
the classifications of the $\mathcal{PS}^{(+)}$ bent functions in 8 dimensions.
In 2018, the code was continually improved, and computation of the
classifications of the $\mathcal{PS}^{(-)}$ bent functions in 8 dimensions was begun.

\section{Discussion}
\label{sec-Discussion}
The investigation of the extended Cayley classes of bent functions is just beginning, and there are many open questions.

The following questions have been settled only for dimensions 2, 4 and 6.
\small{}
\begin{enumerate}
\item
How many extended Cayley classes are there for each dimension?
Are there ``Exponential numbers'' of classes  \cite{Kan83exponential}?
\item
In $n$ dimensions,
which extended translation classes contain the maximum number, $4^n$, of different extended Cayley classes?
\item
Which extended Cayley classes overlap more than one extended translation class?
\item
Which bent functions are Cayley equivalent to their dual?
\end{enumerate}
\normalsize{}

In 8 dimensions, what are the extended affine and extended Cayley classes of bent functions of degree 4 \cite{LanL11counting}?
What are the key properties of the strongly regular graphs obtained as  Cayley graphs of bent functions, e.g., what is the distribution of clique numbers, clique polynomials, etc.? How do these distributions vary with the type of bent function, e.g., how do the distributions for the $\mathcal{PS}^{(-)}$ and $\mathcal{PS}^{(+)}$ partial spread bent functions differ from each other and from the general case?

Finally, how does the concept of extended Cayley classes of bent functions generalize to bent functions
over number fields of prime order $p \neq 2$ \cite{CheTZ11}?



\paragraph*{Acknowledgements.}
This work was begun in 2014 while the author was a Visiting Fellow at the Australian National University,
continued while the author was a Visiting Fellow and a Casual Academic at the University of Newcastle, Australia,
and concluded while the author was an Honorary Fellow at the University of Melbourne,
and an employee of the Bureau of Meteorology.

Thanks to Christine Leopardi for her hospitality at Long Beach.
Thanks to
Robert Craigen,
Joanne Hall,
Kathy Horadam,
David Joyner,
Philippe Langevin,
William Martin,
Padraig {\'O} Cath{\'a}in,
Judy-anne Osborn,
Dima Pasechnik and
William Stein
for valuable questions, discussions and advice;
David Joyner and Caroline Melles for suggestions for improvements based on the first draft of this paper;
Nathan Clisby, for the opportunity to become an Honorary Fellow at the University of Melbourne;
Kodlu, a member of MathOverflow, for asking Philippe Langevin about the affine classification
of bent functions of degree 4 in 8 dimensions;
Gray Chan for a citation tool for IETF RFCs;
and finally, thanks to the authors of SageMath, Bliss, and Nauty for these valuable tools
without which I would not have been able to conduct this research.

\newpage

\appendix

\section{Proof of Theorem \ref{th-Quadratic-Classes}}
\label{app-proof-of}

The proof of Theorem \ref{th-Quadratic-Classes} relies on a number of supporting lemmas,
which are stated and proved here.
\begin{Lemma}
\label{lm-notes-5}
Let $q(x) := x^T L x$ where $L \in \F_2^{2 m \times 2 m}$,
\begin{align*}
L
&:=
\left[
\begin{array}{cc}
0 & I
\\
0 & 0
\end{array}
\right],
\intertext{so that}
q(x) &= \sum_{k=0}^{m-1} x_k x_{m+k}.
\end{align*}

Let $f(x) := q(x+b) + \langle c,x \rangle + q(b)$.
Then there exists $c' \in \F_2^{2m}$ such that
\begin{align*}
f(x)
&=
q(x) + \langle c',x \rangle.
\end{align*}

\end{Lemma}

\begin{proof}
\begin{align*}
q(x) = x^T L x, \quad \text{so~}
q(x+b)
&=
(x^T+b^T) L (x+b)
\\
&= q(x) + x^T L b + b^T L x + q(b)
\\
&= q(x) + \langle (L + L^T) b, x \rangle + q(b),
\intertext{and therefore}
q(x+b) + \langle c, x \rangle + q(b)
&=
q(x) + \langle (L+L^T) b + c, x \rangle.
\end{align*}

\end{proof}

\begin{Lemma}
\label{lm-notes-3}
Let $Z \in \F_2^{2 m \times 2 m}$ be symmetric with zero diagonal.
In other words, $Z = Z^T$, $\diag{Z} = 0$.
Then for any $M \in \F_2^{2 m \times 2 m}$,
\begin{align*}
x^T (M + Z) x  &= x^T M x
\end{align*}
for all $x \in \F_2^{2 m}$.
\end{Lemma}

\begin{proof}
Let $Z$, $x$ be as above.
Then
\begin{align*}
x^T Z x
&=
\sum_{i=0}^{2m-1} \sum_{j=0}^{2m-1} x_i Z_{i,j} x_j
\\
&=
\sum_{i=0}^{2m-1} \sum_{j<i} x_i Z_{i,j} x_j\, +
\sum_{i=0}^{2m-1} x_i Z_{i,i} x_i\, +
\sum_{i=0}^{2m-1} \sum_{j>i} x_i Z_{i,j} x_j
\\
&=
\sum_{i=0}^{2m-1} \sum_{j<i} x_i (Z_{i,j} + Z_{j,i})
= 0.
\intertext{Therefore}
x^T (M + Z) x  &= x^T M x + x^T Z x = x^T M x.
\end{align*}
\end{proof}

\begin{Lemma}
\label{lm-notes-4}
Let $q$ be defined as per Lemma \ref{lm-notes-5}.
Then for all $c \in Z_2^{2 m}$ with $q(c)=0$, there exists $A \in GL(2 m, 2)$ such that
\begin{align*}
q(A x) &= q(x) + \langle c, x \rangle.
\end{align*}
\end{Lemma}

\begin{proof}
Let $C \in \F_2^{2 m \times 2 m}$ be such that $C_{i,j} = \delta_{i,j} c_i$, where $\delta$ is the
\Emph{Dirac delta}: $\delta_{i,j}=1$ if $i=j$ and $0$ otherwise.
In other words $\diag{C} = c$.
Then
\begin{align*}
\langle c, x \rangle
&=
\sum_{i=0}^{2m-1} c_i x_i
\\
&=
\sum_{i=0}^{2m-1} x_i c_i x_i
=
x^T C x.
\end{align*}
Therefore, by Lemma \ref{lm-notes-3},
\begin{align*}
q(x) + \langle c, x \rangle
&=
x^T (L + Z + C) x,
\end{align*}
where $Z \in \F_2^{2 m \times 2 m}$ is symmetric with zero diagonal.

For such $Z$, let $S := Z + C$.
We want to find $A \in \F_2^{2 m \times 2 m}$ such that $q(A x) = q(x) + \langle c, x \rangle.$
In other words,
\begin{align*}
q(A x)
&=
(A x)^T L (A x)
=
x^T A^T L A x
=
x^T (L + S) x.
\end{align*}
This will be true if $A^T L A = L + S.$

Let
\begin{align*}
A
&:=
\left[
\begin{array}{cc}
A_{0,0} & A_{0,1}
\\
A_{1,0} & A_{1,1}
\end{array}
\right],
\quad
S
&:=
\left[
\begin{array}{cc}
S_{0,0} & S_{0,1}
\\
S_{0,1}^T & S_{1,1}
\end{array}
\right]
=:
\left[
\begin{array}{cc}
Z_{0,0} + C_{0,0} & Z_{0,1}
\\
Z_{0,1}^T & Z_{1,1} + C_{1,1}
\end{array}
\right].
\end{align*}
Since
\begin{align*}
L A
&=
\left[
\begin{array}{cc}
0 & I
\\
0 & 0
\end{array}
\right]
\left[
\begin{array}{cc}
A_{0,0} & A_{0,1}
\\
A_{1,0} & A_{1,1}
\end{array}
\right]
=
\left[
\begin{array}{cc}
A_{1,0} & A_{1,1}
\\
0 & 0
\end{array}
\right],
\end{align*}
we require that
\begin{align*}
A^T L A
&=
\left[
\begin{array}{cc}
A_{0,0} & A_{1,0}
\\
A_{0,1} & A_{1,1}
\end{array}
\right]
\left[
\begin{array}{cc}
A_{1,0} & A_{1,1}
\\
0 & 0
\end{array}
\right]
\\
&=
\left[
\begin{array}{cc}
A_{0,0}^T A_{1,0} & A_{0,0}^T A_{1,1}
\\
A_{0,1}^T A_{1,0} & A_{0,1}^T A_{1,1}
\end{array}
\right]
\\
&=
L + S
=
\left[
\begin{array}{cc}
S_{0,0} & I + S_{0,1}
\\
S_{0,1}^T & S_{1,1}
\end{array}
\right],
\end{align*}
and therefore
\begin{align*}
A_{0,0}^T A_{1,0}
&=
S_{0,0},
\quad
A_{0,0}^T A_{1,1}
=
I + S_{0,1},
\\
A_{0,1}^T A_{1,0}
&=
S_{0,1}^T,
\quad
A_{0,1}^T A_{1,1}
=
S_{1,1}.
\end{align*}
If $S_{0,1}=0$ and $A_{0,0}=I$ then
$A_{1,0}=S_{0,0}$, $A_{1,1}=I$ and $A_{0,1}=S_{1,1}$.
In this case, we have $A_{0,1}^T A_{1,0} = S_{0,1}^T = 0$,
i.e. $S_{1,1} S_{0,0} = 0$, and
\begin{align*}
A
&=
\left[
\begin{array}{cc}
I & S_{1,1}
\\
S_{0,0} & I
\end{array}
\right],
\intertext{so that}
A^T L A
&=
\left[
\begin{array}{cc}
I & S_{0,0}
\\
S_{1,1} & I
\end{array}
\right]
\left[
\begin{array}{cc}
S_{0,0} & I
\\
0 & 0
\end{array}
\right]
\\
&=
\left[
\begin{array}{cc}
S_{0,0} & I
\\
0 & S_{1,1}
\end{array}
\right]
\\
&=
L + S.
\end{align*}

Also
\begin{align*}
S
&=
\left[
\begin{array}{cc}
Z_{0,0} + C_{0,0} & 0
\\
0 & Z_{1,1} + C_{1,1}
\end{array}
\right].
\end{align*}

Since $q(c)=0$ we have
\begin{align*}
q(c)
&=
\sum_{k=0}^{m-1} c_k c_{m+k}
=
0.
\end{align*}
Let $K := \{ k \mid c_k c_{m+k} = 1 \}$.
Then we must have $\abs{K} = 2 r$ for some integer $r \geqslant 0$, i.e. $\abs{K}$ is even.
We therefore arbitrarily group the elements of $K$ into pairs $(i_p, j_p)$ for $p=0,\ldots,r-1$,
and define the matrix $T \in \F_2^{m \times m}$ by
\begin{align*}
T_{i,j}
&:=
\sum_{p=0}^{r-1} (\delta_{i,i_p} \delta_{j,j_p} + \delta_{i,j_p} \delta_{j,i_p}),
\end{align*}
so that
\begin{align*}
\begin{cases}
T_{i_p,j_p}
=
T_{j_p,i_p}
=
1
&\text{for~} p \in \{0,\ldots,r-1\},
\\
T_{i,j} = 0
&\text{otherwise.}
\end{cases}
\end{align*}
Since the $r$ pairs $(i_p, j_p)$ partition the set $K$,
the matrix $T$ has at most one non-zero in each row and column.

Recalling that
\begin{align*}
(T^2)_{i,j}
&=
\sum_{k=0}^{m-1} T_{i,k} T_{k,j},
\end{align*}
we see that the general term $T_{i,k} T_{k,j}$ of this sum is non-zero only if either
\begin{align*}
\begin{cases}
i = j = i_p,&\text{and}\ k=j_p,\ \text{or}
\\
i = j = j_p,&\text{and}\ k=i_p,
\end{cases}
\end{align*}
for some $p \in \{0,\ldots,r-1\}$, with all $2r$ of these cases being mutually exclusive.
So $T^2$ is diagonal with $2r$ non-zeros at the elements of $K$.

But $C_{1,1} C_{0,0}$ is diagonal, and $(C_{1,1} C_{0,0})_{i,i} = c_{m+i} c_i$.
Therefore
\begin{align}
T^2 &= C_{1,1} C_{0,0}.
\label{eq-t-2}
\end{align}

Now, let $Z_{0,0}=Z_{1,1}=T$. Then $S_{0,0} = T + C_{0,0}$, $S_{1,1} = T + C_{1,1}$, and
\begin{align*}
S_{1,1} S_{0,0}
&=
(T + C_{1,1})(T + C_{0,0})
=
T^2 + T C_{0,0} + C_{1,1} T + C_{1,1} C_{0,0}
\\
&=
T C_{0,0} + C_{1,1} T,
\end{align*}
where in the last step, we have used \eqref{eq-t-2}.

Now,
\begin{align*}
(T C_{0,0} + C_{1,1} T)_{i,j}
&=
\sum_{k=0}^{m-1} T_{i,k} (C_{0,0})_{k,j} + (C_{1,1})_{i,k} T_{k,j}
\\
&=
T_{i,j} (C_{0,0})_{j,j} + (C_{1,1})_{i,i} T_{i,j}
\\
&=
T_{i,j} \left( c_j + c_{m+i} \right).
\end{align*}
As above, $T_{i,j}$ is non-zero only when $(i,j)=(i_p,j_p)$ or $(i,j)=(j_p,i_p)$
for some $p \in \{0,\ldots,r-1\}$, but in all those cases $c_j=c_{m+j}=1$.

Therefore
\begin{align*}
S_{1,1} S_{0,0} &= T C_{0,0} + C_{1,1} T = 0.
\end{align*}
Similarly, $S_{0,0} S_{1,1} = 0$, and therefore
\begin{align*}
A^2
&=
\left[
\begin{array}{cc}
I & S_{1,1}
\\
S_{0,0} & I
\end{array}
\right]
\left[
\begin{array}{cc}
I & S_{1,1}
\\
S_{0,0} & I
\end{array}
\right]
\\
&=
\left[
\begin{array}{cc}
I + S_{1,1} S_{0,0} & S_{1,1} + S_{1,1}
\\
S_{0,0} + S_{0,0} & I + S_{0,0} S_{1,1}
\end{array}
\right]
=
\left[
\begin{array}{cc}
I & 0
\\
0 & I
\end{array}
\right].
\end{align*}

We have therefore shown that
\begin{align}
A
&:=
\left[
\begin{array}{cc}
I & T + C_{1,1}
\\
T + C_{0,0} & I
\end{array}
\right],
\quad
S
:=
\left[
\begin{array}{cc}
T + C_{0,0} & 0
\\
0 & T + C_{1,1}
\end{array}
\right]
\label{eq-a-s-def}
\end{align}
is a solution to $A^T L A = L + S$ with $A \in GL(2 m, 2)$.

Finally, given $c$ with $q(c)=0$, the matrix $A$ as defined by \eqref{eq-a-s-def} is such that
$q(A x) = q(x) + \langle c, x \rangle$.
\end{proof}

\begin{Lemma}
\label{lm-notes-6}
For $k \in \{0,\ldots,m-1\}$ define $e^{(k)}$ by
\begin{align}
e_i^{(k)} &:= \delta_{i,k} + \delta_{i,m+k}
\label{eq-e-def}
\end{align}
for $i \in \{0,\ldots,2 m - 1\}$.

Let $h(x) := q(x) + \langle e^{(0)}, x \rangle$, where $q$ is defined as per Lemma \ref{lm-notes-5}.
Then for any $c'$ such that $q(c')=1$, there exists $B \in GL(2 m, 2)$ such that
\begin{align}
h(B x) &= q(x) + \langle c',x \rangle.
\label{eq-h-B-x}
\end{align}
\end{Lemma}

\begin{proof}
Let $K'=\{k \mid c'_k c'_{m+k} = 1\}$. Since $q(c')=1$, $\abs{K'}$ is odd.
Choose any $\ell \in K'$, and let $c := c' + e^{(\ell)}$.
Then $c_{\ell} = c_{m+\ell} = 0$ and $q(c)=0$.

Now let $h^{(\ell)}(x) := q(x) + \langle e^{(\ell)}, x \rangle$.
We calculate
\begin{align*}
h^{(\ell)}(A x)
&=
q(A x) + \langle e^{(\ell)}, A x \rangle
=
q(x) + \langle c, x \rangle + \langle A^T e^{(\ell)}, x \rangle
\\
&=
q(x) + \langle c + A^T e^{(\ell)}, x \rangle
\end{align*}
for $A$ given by the proof of Lemma \ref{lm-notes-4}.

If we let $K := \{ k \mid c_k c_{m+k} = 1 \}$, we see that $K = K' \setminus \{\ell\}$.
Applying the other definitions and techniques used in the proof of Lemma \ref{lm-notes-4},
we see that since $c_{\ell} = c_{m+\ell} = 0$ and $K$ does not contain $\ell$,
column $\ell$ of each of $S_{0,0} := T + C_{0,0}$
and $S_{1,1} := T + C_{1,1}$ is $0$, and therefore columns $\ell$ and $m + \ell$ of
\begin{align*}
A^T + I
&:=
\left[
\begin{array}{cc}
I & T + C_{0,0}
\\
T + C_{1,1} & I
\end{array}
\right]
\end{align*}
are both $0$.
Therefore $A^T e^{(\ell)} = e^{(\ell)}$, and therefore
\begin{align*}
h^{(\ell)}(A x)
&=
q(x) + \langle c', x \rangle.
\end{align*}

\end{proof}

\begin{Lemma}
\label{lm-notes-6-b}
For distinct $k,\ell \in \{0,\ldots,m-1\}$ let $e^{(k)}, e^{(\ell)}$ be defined as per Lemma
\ref{lm-notes-6}.
Let $h(x) := q(x) + \langle e^{(k)}, x \rangle$, where $q$ is defined as per Lemma \ref{lm-notes-5}.
Then there exists $A \in GL(2 m, 2)$ such that
\begin{align}
h(A x)
&=
q(x) + \langle e^{(\ell)},x \rangle.
\end{align}
\end{Lemma}

\begin{proof}
The matrix $A$ is the permutation matrix for the the permutation $(k\ \ell)(m+k\ m+\ell)$ (defined
using cycle notation.)
\end{proof}

\begin{Lemma}
\label{lm-notes-7}
Let $q$ be defined as per Lemma \ref{lm-notes-5}.
Then for all $c, c' \in Z_2^{2 m}$ with $q(c)=q(c')=1$, there exists $A \in GL(2 m, 2)$ such that
if $h(x) := q(x) + \langle c, x \rangle$, then
\begin{align*}
h(A x) &= q(x) + \langle c', x \rangle.
\end{align*}
\end{Lemma}

\begin{proof}
This is a consequence of Lemmas \ref{lm-notes-6} and \ref{lm-notes-6-b}.
\end{proof}

\begin{proofof}{Theorem \ref{th-Quadratic-Classes}}
It is well known that all quadratic bent functions are contained in one Extended Affine equivalence
class.
As a consequence of Corollary \ref{corr-Affine-Translate-Cayley}, without loss of generality, we need
only examine
the Extended Translation equivalence class of the quadratic function $q$ as defined in Lemma
\ref{lm-notes-5}.

As a result of Lemma \ref{lm-notes-5}, we actually need only examine functions of the form
$f(x) = q(x) + \langle c,x \rangle$
for some $c \in \F_2^{2m}$.
Lemma \ref{lm-notes-4} implies that all such functions for which $q(c)=0$ are Cayley equivalent to
$q$.
Lemma \ref{lm-notes-7} implies that any two such functions $q(x) + \langle c, x \rangle$ and $q(x)\, +
\langle c', x \rangle$
with $q(c)=q(c')=1$ are Cayley equivalent to each other.

The functions where $q(c)=0$ are not Cayley equivalent to the functions where $q(c)=1$ because
Lemma \ref{lm-notes-9b} implies that
\begin{align*}
\weightclass{x \mapsto q(x) + \langle c,x \rangle}
&=
\dual{q}(c) = q(c),
\end{align*}
since $q$ is self-dual.
\end{proofof}


\section{Computational results for low dimensions}
\label{sec-Empirical}
This section lists some properties of bent functions and their extended affine (EA) classes, extended translation (ET) classes, and extended Cayley classes
that have been computed for 2, 4, 6 and 8 dimensions.
The computations were made using Sage \cite{SageMath7517} and CoCalc \cite{CoCalc}.
Larger scale computations, involving millions of ET classes,
were conducted on the Raijin supercomputer of the National Computational Infrastructure.
Sage and Python code for these computations are available on GitHub \cite{Leo16GitHub} and CoCalc \cite{Leo17CoCalc}.
Some CoCalc worksheets also illustrate these and related computations \cite{Leo17CoCalc}.
The Sage and Python code is briefly described in Section \ref{sec-Code}.

In the tables below, each bent function is defined by its algebraic normal form, and each Cayley class is described by
its number within the ET class of the bent function (from 0, in the order in which Sage identified non-isomorphic graphs),
followed by three properties of the Cayley graph: its parameters as a strongly regular graph,
the 2-rank of its adjacency matrix \cite{Brov92}, and its clique polynomial \cite{HoeL94}.
The 2-rank is included for comparison with Tonchev's tables of 2-weight codes in 6 dimensions \cite{Ton96uniformly,Ton07codes}.
The clique polynomial is included for interest's sake, and to illustrate the variety of strongly regular graphs
that exist with the same parameters, even for low dimensions.

The plots below are produced by the function \texttt{sage.}\texttt{plot.}\texttt{matrix\_plot},
with \texttt{gist\_stern} as the \texttt{colormap}.
Thus the smallest number is coloured black and the largest number is coloured white.

The plotted matrices all contain non-negative integers.
The weight class matrices are defined by Definition~\ref{def-weight-class-matrix}, and are $\{0,1\}$ matrices,
so their matrix plots are therefore black and white, with black representing 0 and white representing 1.
The other matrices record the number of the Cayley class within the
ET class, starting from 0, as per corresponding table of extended Cayley classes.

Some highlights of the computational results include:
\begin{enumerate}
\item Verification of Theorem \ref{th-Quadratic-Classes}: the quadratic bent functions have two Cayley classes
 corresponding to the two weight classes.
\item In 6 dimensions, identification of the Cayley classes corresponding to
Tonchev's tables of 2-weight codes \cite{Ton96uniformly,Ton07codes}.
\item In 6 and 8 dimensions, extended Cayley equivalence between a quadratic bent function
and a bent function of degree 3.
In each case, the isomorphism between Cayley graphs is not a linear function on $\F_2^{2m}$.
\item In 8 dimensions, the result that two of Braeken's extended affine classes of bent functions of degree
at most 3 \cite{Bra06thesis, Tok15bent} are actually the same class.
Thus the list only contains 9 distinct classes and not 10.
\item In 8 dimensions, 8 of the 256 bent functions used for the S-boxes of the CAST-128 cypher \cite{RFC2144,Ada97}
are exceptional in the sense that,
for each of these 8 bent functions, the $65\,536$ bent functions in the extended translation class,
and their  $65\,536$ duals do not yield $131\,072$ distinct Cayley graphs.
In contrast, for the remaining 248 bent functions, these $131\,072$ Cayley graphs are all non-isomorphic.
\end{enumerate}

\subsection{Bent functions in 2 dimensions}
The bent functions on $\F_2^2$ consist of one EA class, containing the ET class: $[f_{2,1}]$
where $f_{2,1}(x) := x_0 x_1$ is self dual.
The ET class contains two extended Cayley classes as per Table~\ref{tab-c2_1_EC_classes}.
Note that the Cayley graph for class 1 is $K_4$, which is not considered to be strongly regular, by convention.
\begin{table}[!bhpt] 
\small{
\begin{align*}
\def\arraystretch{1.2}
\begin{array}{|cccl|}
\hline
\text{Class} &
\text{Parameters} &
\text{2-rank} &
\text{Clique polynomial}
\\
\hline
0 &
(4, 1, 0, 0) &
4 &
\begin{array}{l}
2t^{2} + 4t + 1
\end{array}
\\
1 &
K_4 &
4 &
\begin{array}{l}
t^{4} + 4t^{3} + 6t^{2} + 4t + 1
\end{array}
\\
\hline
\end{array}
\end{align*}
}
\caption{$[f_{2,1}]$ extended Cayley classes.}
\label{tab-c2_1_EC_classes}
\end{table}

\begin{figure}[!ht]
\centering
\begin{minipage}{.48\textwidth}
  \centering
  \includegraphics[width=.9\linewidth]{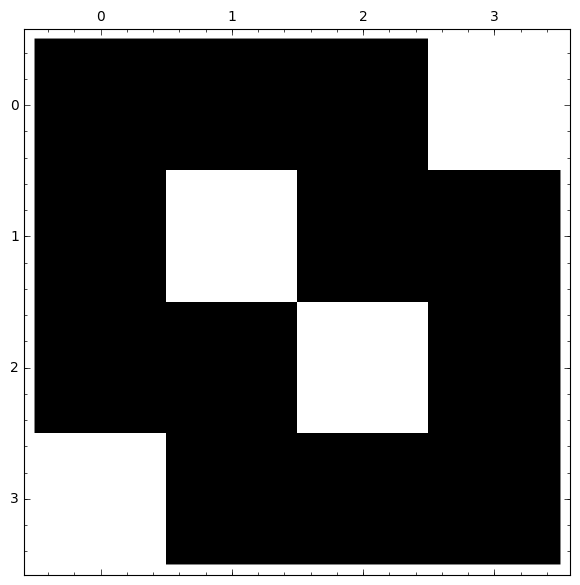}
  \captionof{figure}{$[f_{2,1}]$: weight classes. ~~~~\\~~~~}
  \label{fig:c2_1_weight_class_matrix}
\end{minipage}%
~~~~
\begin{minipage}{.48\textwidth}
  \centering
  \includegraphics[width=.9\linewidth]{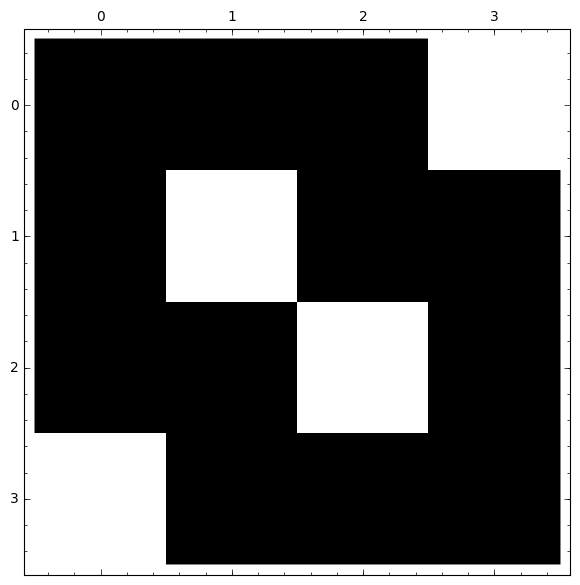}
  \captionof{figure}{$[f_{2,1}]$: extended Cayley classes.}
  \label{fig:c2_1_bent_cayley_graph_index_matrix}
\end{minipage}
\end{figure}
As expected from Theorem~\ref{th-Quadratic-Classes},
the two extended Cayley classes correspond to the two weight classes,
as shown in Figures~\ref{fig:c2_1_weight_class_matrix} and~\ref{fig:c2_1_bent_cayley_graph_index_matrix}.

\subsection{Bent functions in 4 dimensions}
The bent functions on $\F_2^4$ consist of one EA class, containing the ET class $[f_{4,1}]$ where
$f_{4,1}(x) := x_0 x_1 + x_2 x_3$ is self dual.
The ET class contains two extended Cayley classes as per Table~\ref{tab-c4_1_EC_classes}.
\begin{table}[!bhpt] 
\small{
\begin{align*}
\def\arraystretch{1.2}
\begin{array}{|cccl|}
\hline
\text{Class} &
\text{Parameters} &
\text{2-rank} &
\text{Clique polynomial}
\\
\hline
0 &
(16, 6, 2, 2) &
6 &
\begin{array}{l}
8t^{4} + 32t^{3} + 48t^{2} + 16t + 1
\end{array}
\\
1 &
(16, 10, 6, 6) &
6 &
\begin{array}{l}
16t^{5} + 120t^{4} + 160t^{3}
\,+
\\
 80t^{2} + 16t + 1
\end{array}
\\
\hline
\end{array}
\end{align*}
}
\caption{$[f_{4,1}]$ extended Cayley classes.}
\label{tab-c4_1_EC_classes}
\end{table}

%

\begin{figure}[!bhpt] 
\centering
\begin{minipage}{.48\textwidth}
  \centering
  \includegraphics[width=.9\linewidth]{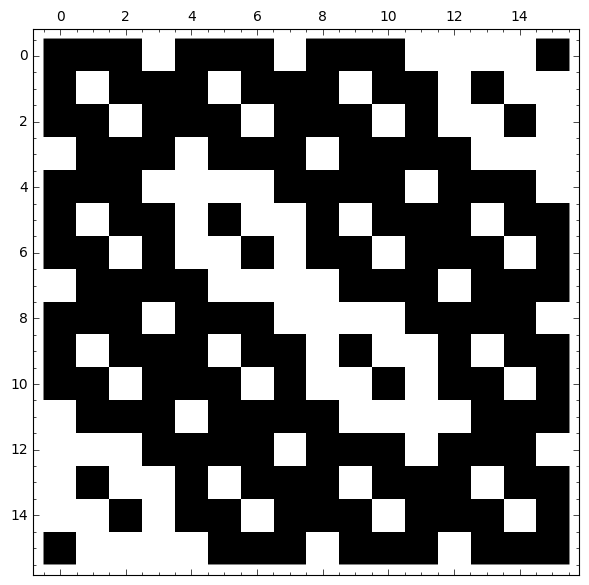}
  \captionof{figure}{$[f_{4,1}]$: weight classes. ~~~~\\~~~~}
  \label{fig:c4_1_weight_class_matrix}
\end{minipage}%
~~~~
\begin{minipage}{.48\textwidth}
  \centering
  \includegraphics[width=.9\linewidth]{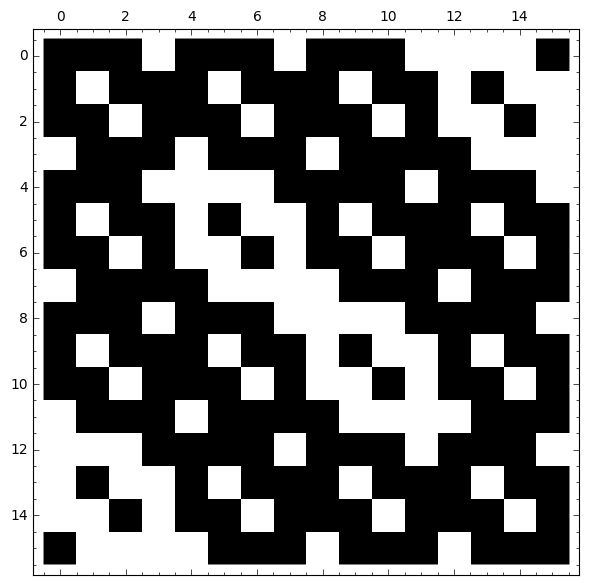}
  \captionof{figure}{$[f_{4,1}]$: extended Cayley classes.}
  \label{fig:c4_1_bent_cayley_graph_index_matrix}
\end{minipage}
\end{figure}
The two extended Cayley classes correspond to the two weight classes,
as shown in Figures~\ref{fig:c4_1_weight_class_matrix} and~\ref{fig:c4_1_bent_cayley_graph_index_matrix}.

\newpage
\subsection{Bent functions in 6 dimensions}
\paragraph*{Extended affine classes.}
The bent functions on $\F_2^6$ consist of four
EA classes, containing the ET classes as listed in Table~\ref{tab-c6_ET_classes}
\cite[p. 303]{Rot76} \cite[Section 7.2]{Tok15bent}.
\begin{table}[!bhpt] 
\small{
\begin{align*}
\def\arraystretch{1.2}
\begin{array}{|cl|}
\hline
\text{Class} &
\text{Representative}
\\
\hline
\,[f_{6,1}] & f_{6,1} :=
\begin{array}{l}
x_{0} x_{1} + x_{2} x_{3} + x_{4} x_{5}
\end{array}
\\
\,[f_{6,2}] & f_{6,2} :=
\begin{array}{l}
x_{0} x_{1} x_{2} + x_{0} x_{3} + x_{1} x_{4} + x_{2} x_{5}
\end{array}
\\
\,[f_{6,3}] & f_{6,3} :=
\begin{array}{l}
x_{0} x_{1} x_{2} + x_{0} x_{1} + x_{0} x_{3} + x_{1} x_{3} x_{4} + x_{1} x_{5}\, +
\\
x_{2} x_{4} + x_{3} x_{4}
\end{array}
\\
\,[f_{6,4}] & f_{6,4} :=
\begin{array}{l}
x_{0} x_{1} x_{2} + x_{0} x_{3} + x_{1} x_{3} x_{4} + x_{1} x_{5} + x_{2} x_{3} x_{5}\, +
\\
x_{2} x_{3} + x_{2} x_{4} + x_{2} x_{5} + x_{3} x_{4} + x_{3} x_{5}
\end{array}
\\
\hline
\end{array}
\end{align*}
}
\caption{6 dimensions: ET classes.}
\label{tab-c6_ET_classes}
\end{table}

In 1996, Tonchev classified the binary projective two-weight $[27,21,3]$ and $[35,6,16]$ codes
listing them in Tables 1 and 2, respectively, of his paper \cite{Ton96uniformly}.
These tables are repeated as Tables 1.155 and 1.156 in Chapter VII.1 of the Handbook of
Combinatorial Designs, Second Edition \cite{Ton07codes},
with a different numbering.
For each of the codes listed in these two tables, the characteristics of the corresponding
strongly regular graph is also listed.

In the classification given below, the Cayley graph of each Cayley class is matched by isomorphism
with a strongly regular graph corresponding to one
or more of Tonchev's projective two-weight codes, or the complement of such a graph.
Tonchev's strongly regular graphs were checked using the function
\verb!strongly_regular_from_two_weight_code!, which uses the smaller of the two weights to create the graph
\cite{SageMath7517}.
%
\paragraph*{ET class $[f_{6,1}]$.}
This is the ET class of the bent function
$f_{6,1}(x) := x_0 x_1 + x_2 x_3 + x_4 x_5.$
This function is quadratic and self-dual.

The ET class contains two extended Cayley classes as per Table~\ref{tab-c6_1_EC_classes}.

\begin{table}[!bhpt] 
\small{}
\begin{align*}
\def\arraystretch{1.2}
\begin{array}{|cccl|}
\hline
\text{Class} &
\text{Parameters} &
\text{2-rank} &
\text{Clique polynomial}
\\
\hline
0 &
(64, 28, 12, 12) &
8 &
\begin{array}{l}
64t^{8} + 512t^{7} + 1792t^{6} + 3584t^{5}
\,+
\\
 5376t^{4} + 3584t^{3} + 896t^{2} + 64t + 1
\end{array}
\\
1 &
(64, 36, 20, 20) &
8 &
\begin{array}{l}
2304t^{6} + 13824t^{5} + 19200t^{4} + 7680t^{3}
\,+
\\
 1152t^{2} + 64t + 1
\end{array}
\\
\hline
\end{array}
\end{align*}
\caption{$[f_{6,1}]$ extended Cayley classes.}
\label{tab-c6_1_EC_classes}
\end{table}

The Cayley graphs for classes 0 and 1 are isomorphic to those those obtained from
Tonchev's projective two-weight codes \cite{Ton07codes} as per Table~\ref{tab-c6_1_codes}.

\begin{table}[!bhpt] 
\small{
\begin{align*}
\def\arraystretch{1.2}
\begin{array}{|ccl|}
\hline
\text{Class} &
\text{Parameters} & \text{Reference}
\\
\hline
0 & [35,6,16] & \text{Table 1.156 1, 2 (complement)}
\\
1 & [27,6,12] & \text{Table 1.155 1 }
\\
\hline
\end{array}
\end{align*}
}
\caption{$[f_{6,1}]$ Two-weight projective codes.}
\label{tab-c6_1_codes}
\end{table}


The two extended Cayley classes correspond to the two weight classes,
as shown in Figures~\ref{fig:c6_1_weight_class_matrix} and~\ref{fig:c6_1_bent_cayley_graph_index_matrix}.

\begin{figure}[!bhpt] 
\centering
\begin{minipage}{.48\textwidth}
  \centering
  \includegraphics[width=.9\linewidth]{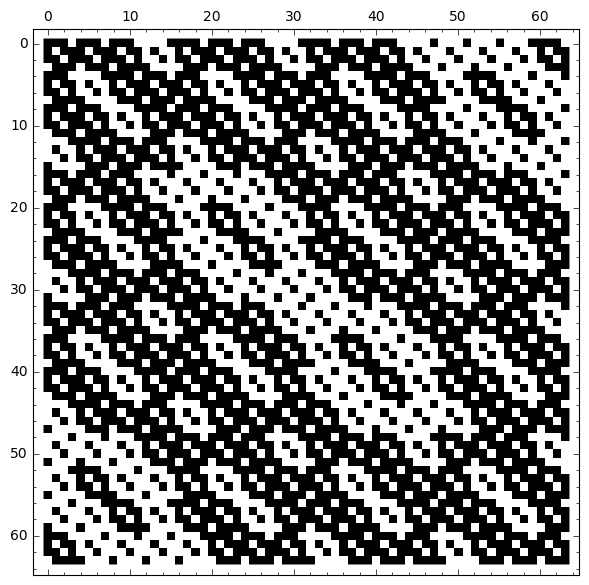}
  \captionof{figure}{$[f_{6,1}]$: weight classes. ~~~~\\~~~~}
  \label{fig:c6_1_weight_class_matrix}
\end{minipage}%
~~~~
\begin{minipage}{.48\textwidth}
  \centering
  \includegraphics[width=.9\linewidth]{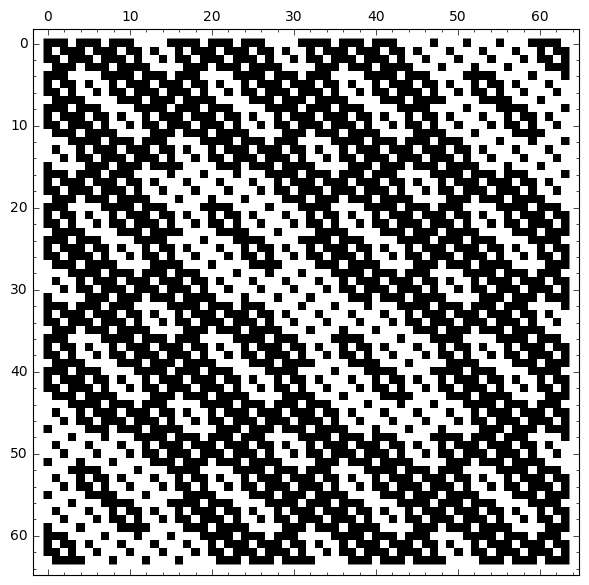}
  \captionof{figure}{$[f_{6,1}]$: extended Cayley classes.}
  \label{fig:c6_1_bent_cayley_graph_index_matrix}
\end{minipage}
\end{figure}
Remark: The sequence of Figures~\ref{fig:c2_1_weight_class_matrix}, \ref{fig:c4_1_weight_class_matrix},
and \ref{fig:c6_1_weight_class_matrix} displays a fractal-like self-similar quality.

\paragraph*{ET class $[f_{6,2}]$.}
This is the ET class of the bent function
$f_{6,2}(x) := x_{0} x_{1} x_{2} + x_{0} x_{3} + x_{1} x_{4} + x_{2} x_{5}$.

The ET class contains three extended Cayley classes as per Table~\ref{tab-c6_2_EC_classes}.

\begin{table}[!bhpt] 
\small{}
\begin{align*}
\def\arraystretch{1.2}
\begin{array}{|cccl|}
\hline
\text{Class} &
\text{Parameters} &
\text{2-rank} &
\text{Clique polynomial}
\\
\hline
0 &
(64, 28, 12, 12) &
8 &
\begin{array}{l}
64t^{8} + 512t^{7} + 1792t^{6} + 3584t^{5}
\,+
\\
 5376t^{4} + 3584t^{3} + 896t^{2} + 64t + 1
\end{array}
\\
1 &
(64, 28, 12, 12) &
8 &
\begin{array}{l}
256t^{6} + 1536t^{5} + 4352t^{4} + 3584t^{3}
\,+
\\
 896t^{2} + 64t + 1
\end{array}
\\
2 &
(64, 36, 20, 20) &
8 &
\begin{array}{l}
192t^{8} + 1536t^{7} + 8960t^{6} + 19968t^{5}
\,+
\\
 20224t^{4} + 7680t^{3} + 1152t^{2} + 64t + 1
\end{array}
\\
\hline
\end{array}
\end{align*}
\caption{$[f_{6,2}]$ extended Cayley classes.}
\label{tab-c6_2_EC_classes}
\end{table}

The Cayley graph for class 0 is isomorphic to graph 0 of ET class $[f_{6,1}]$,
This reflects the fact that $f_{6,1} \equiv f_{6,2}$, even though these two functions are not
EA equivalent.
This is therefore an example of an isomorphism between Cayley graphs of bent functions on
$\F_2^6$ that is not a linear function.

The Cayley graph for class 0 is also isomorphic to the complement of Royle's $(64,35,18,20)$ strongly regular graph $X$
\cite{Roy08normal}.

%
%

The Cayley graphs for classes 0 to 2 are isomorphic to those those obtained from
Tonchev's projective two-weight codes \cite{Ton07codes} as per Table~\ref{tab-c6_2_codes}.

\begin{table}[!bhpt] 
\small{
\begin{align*}
\def\arraystretch{1.2}
\begin{array}{|ccl|}
\hline
\text{Class} &
\text{Parameters} & \text{Reference}
\\
\hline
0 & [35,6,16] & \text{Table 1.156 1, 2 (complement)}
\\
1 & [35,6,16] & \text{Table 1.156 3 (complement)}
\\
2 & [27,6,12] & \text{Table 1.155 2 }
\\
\hline
\end{array}
\end{align*}
}
\caption{$[f_{6,2}]$ Two-weight projective codes.}
\label{tab-c6_2_codes}
\end{table}

The three extended Cayley classes are distributed between the two weight classes,
as shown in Figures~\ref{fig:c6_2_weight_class_matrix} and~\ref{fig:c6_2_bent_cayley_graph_index_matrix}.

\begin{figure}[!bhpt] 
\centering
\begin{minipage}{.48\textwidth}
  \centering
  \includegraphics[width=.9\linewidth]{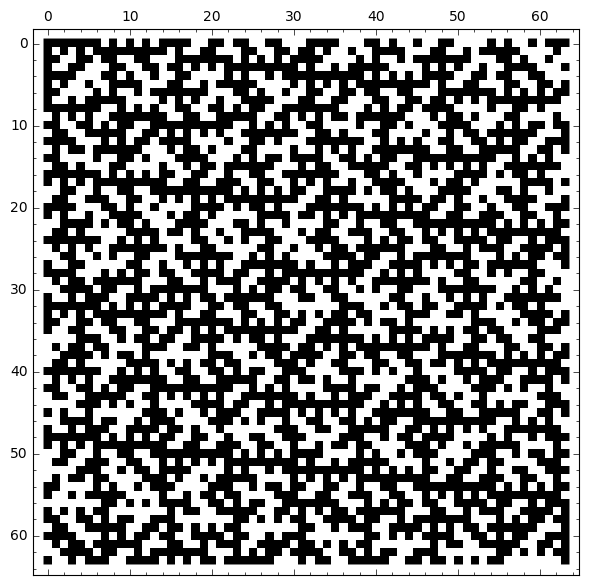}
  \captionof{figure}{$[f_{6,2}]$: weight classes. ~~~~\\~~~~}
  \label{fig:c6_2_weight_class_matrix}
\end{minipage}%
~~~~
\begin{minipage}{.48\textwidth}
  \centering
  \includegraphics[width=.9\linewidth]{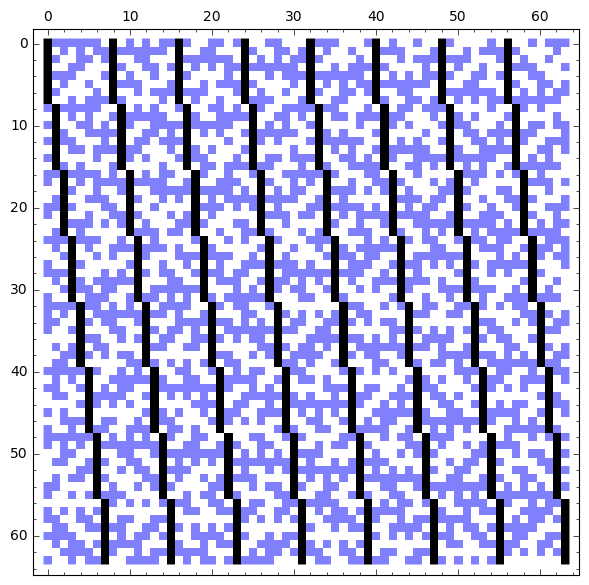}
  \captionof{figure}{$[f_{6,2}]$: extended Cayley classes.}
  \label{fig:c6_2_bent_cayley_graph_index_matrix}
\end{minipage}
\end{figure}

%
\paragraph*{ET class $[f_{6,3}]$.}
This is the ET class of the bent function
\begin{align*}
f_{6,3}(x) &= x_{0} x_{1} x_{2} + x_{0} x_{1} + x_{0} x_{3} + x_{1} x_{3} x_{4}
\\
           &+ x_{1} x_{5} + x_{2} x_{4} + x_{3} x_{4}.
\end{align*}

The ET class contains four extended Cayley classes as per Table~\ref{tab-c6_3_EC_classes}.

\begin{table}[!bhpt] 
\small{}
\begin{align*}
\def\arraystretch{1.2}
\begin{array}{|cccl|}
\hline
\text{Class} &
\text{Parameters} &
\text{2-rank} &
\text{Clique polynomial}
\\
\hline
0 &
(64, 28, 12, 12) &
12 &
\begin{array}{l}
32t^{8} + 256t^{7} + 896t^{6} + 2048t^{5} + 4608t^{4}
\,+
\\
 3584t^{3} + 896t^{2} + 64t + 1
\end{array}
\\
1 &
(64, 36, 20, 20) &
12 &
\begin{array}{l}
160t^{8} + 1280t^{7} + 9344t^{6} + 21504t^{5}
\,+
\\
 20480t^{4} + 7680t^{3} + 1152t^{2} + 64t + 1
\end{array}
\\
2 &
(64, 28, 12, 12) &
12 &
\begin{array}{l}
64t^{6} + 1024t^{5} + 4096t^{4} + 3584t^{3}
\,+
\\
 896t^{2} + 64t + 1
\end{array}
\\
3 &
(64, 36, 20, 20) &
12 &
\begin{array}{l}
160t^{8} + 1664t^{7} + 9792t^{6} + 21504t^{5}
\,+
\\
 20480t^{4} + 7680t^{3} + 1152t^{2} + 64t + 1
\end{array}
\\
\hline
\end{array}
\end{align*}
\caption{$[f_{6,3}]$ extended Cayley classes.}
\label{tab-c6_3_EC_classes}
\end{table}

The Cayley graphs for classes 0 to 3 are isomorphic to those those obtained from
Tonchev's projective two-weight codes \cite{Ton07codes} as per Table~\ref{tab-c6_3_codes}.

\begin{table}[!bhpt] 
\small{
\begin{align*}
\def\arraystretch{1.2}
\begin{array}{|ccl|}
\hline
\text{Class} &
\text{Parameters} & \text{Reference}
\\
\hline
0 & [35,6,16] & \text{Table 1.156 4 (complement)}
\\
1 & [27,6,12] & \text{Table 1.155 3 }
\\
2 & [35,6,16] & \text{Table 1.156 5 (complement)}
\\
3 & [27,6,12] & \text{Table 1.155 4 }
\\
\hline
\end{array}
\end{align*}
}
\caption{$[f_{6,3}]$ Two-weight projective codes.}
\label{tab-c6_3_codes}
\end{table}

The four extended Cayley classes are distributed between the two weight classes,
as shown in Figures~\ref{fig:c6_3_weight_class_matrix} and~\ref{fig:c6_3_bent_cayley_graph_index_matrix}.

\begin{figure}[!ht] 
\centering
\begin{minipage}{.48\textwidth}
  \centering
  \includegraphics[width=.9\linewidth]{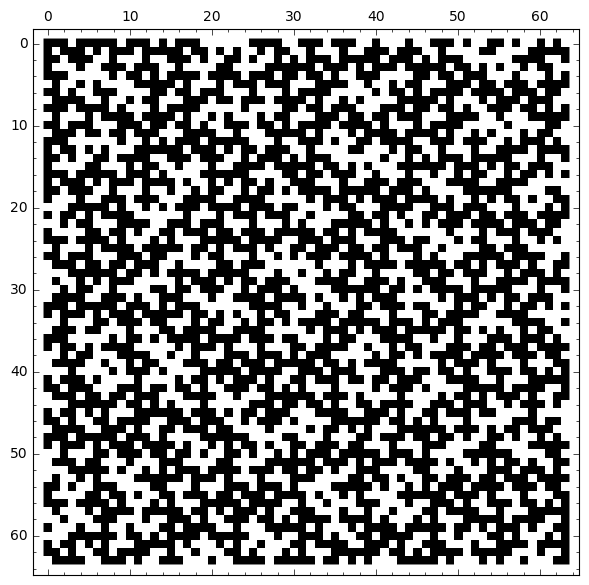}
  \captionof{figure}{$[f_{6,3}]$: weight classes. ~~~~\\~~~~}
  \label{fig:c6_3_weight_class_matrix}
\end{minipage}%
~~~~
\begin{minipage}{.48\textwidth}
  \centering
  \includegraphics[width=.9\linewidth]{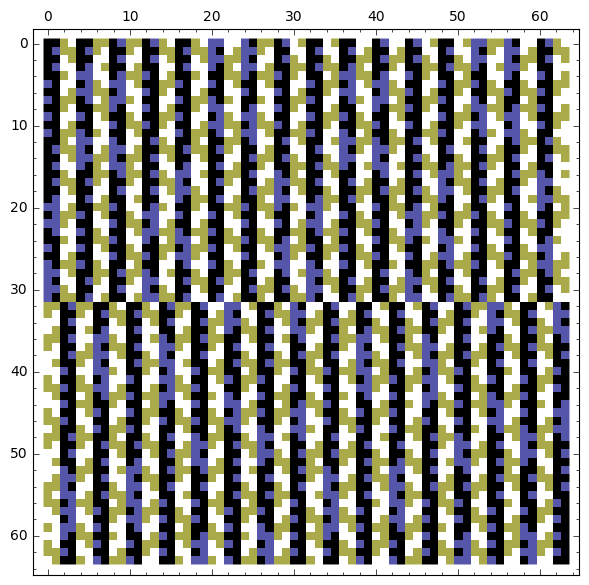}
  \captionof{figure}{$[f_{6,3}]$: extended Cayley classes.}
  \label{fig:c6_3_bent_cayley_graph_index_matrix}
\end{minipage}
\end{figure}

\paragraph*{ET class $[f_{6,4}]$.}
This is the ET class of the bent function
\begin{align*}
f_{6,4}(x) &= x_{0} x_{1} x_{2} + x_{0} x_{3} + x_{1} x_{3} x_{4} + x_{1} x_{5} + x_{2} x_{3} x_{5}
\\
           &+ x_{2} x_{3} + x_{2} x_{4} + x_{2} x_{5} + x_{3} x_{4} + x_{3} x_{5}.
\end{align*}

The ET class contains three extended Cayley classes as per Table~\ref{tab-c6_4_EC_classes}.

\begin{table}[!bhpt] 
\small{}
\begin{align*}
\def\arraystretch{1.2}
\begin{array}{|cccl|}
\hline
\text{Class} &
\text{Parameters} &
\text{2-rank} &
\text{Clique polynomial}
\\
\hline
0 &
(64, 28, 12, 12) &
14 &
\begin{array}{l}
32t^{8} + 256t^{7} + 896t^{6} + 1792t^{5} + 4480t^{4}
\,+
\\
 3584t^{3} + 896t^{2} + 64t + 1
\end{array}
\\
1 &
(64, 28, 12, 12) &
14 &
\begin{array}{l}
16t^{8} + 128t^{7} + 448t^{6} + 1280t^{5} + 4224t^{4}
\,+
\\
 3584t^{3} + 896t^{2} + 64t + 1
\end{array}
\\
2 &
(64, 36, 20, 20) &
14 &
\begin{array}{l}
176t^{8} + 1408t^{7} + 9664t^{6} + 22272t^{5}
\,+
\\
 20608t^{4} + 7680t^{3} + 1152t^{2} + 64t + 1
\end{array}
\\
\hline
\end{array}
\end{align*}
\caption{$[f_{6,4}]$ extended Cayley classes.}
\label{tab-c6_4_EC_classes}
\end{table}

The Cayley graphs for classes 0 to 2 are isomorphic to those those obtained from
Tonchev's projective two-weight codes \cite{Ton07codes} as per Table~\ref{tab-c6_4_codes}.

\begin{table}[!bhpt] 
\small{
\begin{align*}
\def\arraystretch{1.2}
\begin{array}{|ccl|}
\hline
\text{Class} &
\text{Parameters} & \text{Reference}
\\
\hline
0 & [35,6,16] & \text{Table 1.156 7 (complement)}
\\
1 & [35,6,16] & \text{Table 1.156 6 (complement)}
\\
2 & [27,6,12] & \text{Table 1.155 5 }
\\
\hline
\end{array}
\end{align*}
}
\caption{$[f_{6,4}]$ Two-weight projective codes.}
\label{tab-c6_4_codes}
\end{table}


The three extended Cayley classes are distributed between the two weight classes,
as shown in Figures~\ref{fig:c6_4_weight_class_matrix} and~\ref{fig:c6_4_bent_cayley_graph_index_matrix}.

\begin{figure}[!hpt] 
\centering
\begin{minipage}{.48\textwidth}
  \centering
  \includegraphics[width=.9\linewidth]{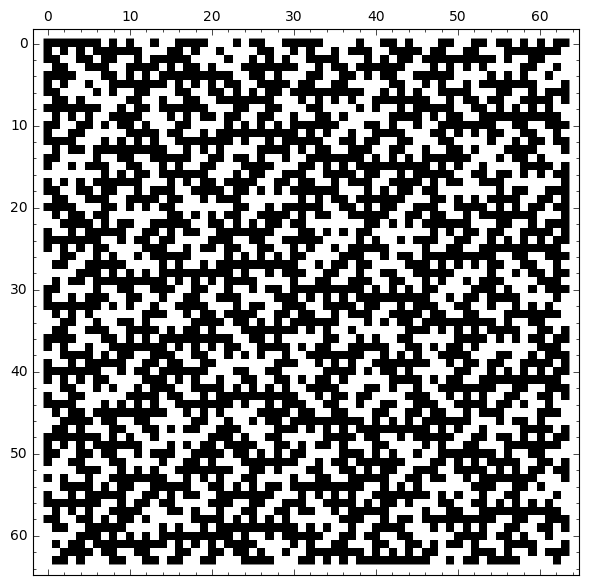}
  \captionof{figure}{$[f_{6,4}]$: weight classes. ~~~~\\~~~~}
  \label{fig:c6_4_weight_class_matrix}
\end{minipage}%
~~~~
\begin{minipage}{.48\textwidth}
  \centering
  \includegraphics[width=.9\linewidth]{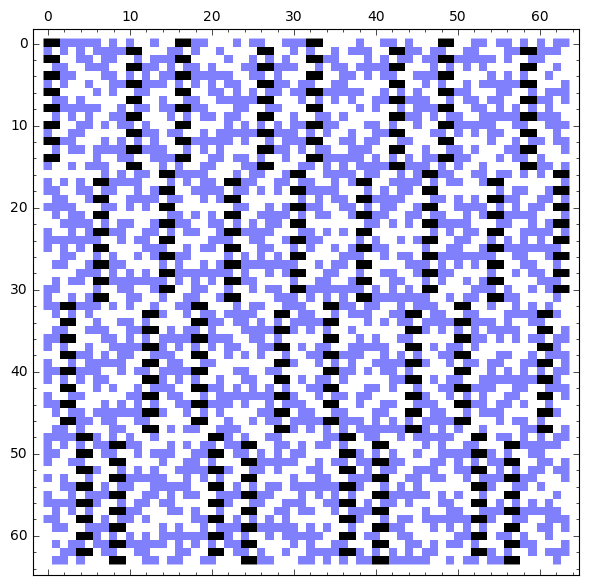}
  \captionof{figure}{$[f_{6,4}]$: extended Cayley classes.}
  \label{fig:c6_4_bent_cayley_graph_index_matrix}
\end{minipage}
\end{figure}
\newpage
\subsection{Bent functions in 8 dimensions}

There are
$99\,270\,589\,265\,934\,370\,305\,785\,861\,242\,880 \approx 2^{106}$ bent functions in 8 dimensions,
according to Langevin and Leander \cite{LanL11counting}.
%
%
The number of EA classes has not yet been published,
let alone a list of representative bent functions.
The lists of EA classes of bent functions that have so far been published include those
for the bent functions of degree at most 3 \cite[Section 5.5.2]{Bra06thesis} \cite[Section 7.3]{Tok15bent},
and the partial spread bent functions \cite{Lan10psf,LanH11counting}.
The bent functions used in the S-boxes of the CAST-128 encryption algorithm \cite{Ada97,RFC2144}
are also representatives of disjoint EA classes.
\paragraph*{Extended affine classes of degree at most 3.}
According to a list contained in Braeken's PhD thesis \cite[Section 5.5.2]{Bra06thesis},
and repeated in Tokareva's table \cite[Section 7.3]{Tok15bent},
the bent functions on $\F_2^8$, of degree at most 3, consist of 10
EA classes, whose representatives are listed in Table~\ref{tab-c8_ET_classes}.
\begin{table}[!bhpt] 
\small{}
\begin{align*}
\def\arraystretch{1.2}
\begin{array}{|cl|}
\hline
\text{Class} &
\text{Representative}
\\
\hline
\,[f_{ 8 , 1 }] & f_{ 8 , 1 } :=
\begin{array}{l}
x_{0} x_{1} + x_{2} x_{3} + x_{4} x_{5} + x_{6} x_{7}
\end{array}
\\
\,[f_{ 8 , 2 }] & f_{ 8 , 2 } :=
\begin{array}{l}
x_{0} x_{1} x_{2} + x_{0} x_{3} + x_{1} x_{4} + x_{2} x_{5} + x_{6} x_{7}
\end{array}
\\
\,[f_{ 8 , 3 }] & f_{ 8 , 3 } :=
\begin{array}{l}
x_{0} x_{1} x_{2} + x_{0} x_{6} + x_{1} x_{3} x_{4} + x_{1} x_{5} + x_{2} x_{3} + x_{4} x_{7}
\end{array}
\\
\,[f_{ 8 , 4 }] & f_{ 8 , 4 } :=
\begin{array}{l}
x_{0} x_{1} x_{2} + x_{0} x_{2} + x_{0} x_{4} + x_{1} x_{3} x_{4} + x_{1} x_{5} + x_{2} x_{3}\, +
x_{6} x_{7}
\end{array}
\\
\,[f_{ 8 , 5 }] & f_{ 8 , 5 } :=
\begin{array}{l}
x_{0} x_{1} x_{2} + x_{0} x_{6} + x_{1} x_{3} x_{4} + x_{1} x_{4} + x_{1} x_{5} + x_{2} x_{3} x_{5}
+ x_{2} x_{4} + x_{3} x_{7}
\end{array}
\\
\,[f_{ 8 , 6 }] & f_{ 8 , 6 } :=
\begin{array}{l}
x_{0} x_{1} x_{2} + x_{0} x_{2} + x_{0} x_{3} + x_{1} x_{3} x_{4} + x_{1} x_{6} + x_{2} x_{3} x_{5}
+ x_{2} x_{4} + x_{5} x_{7}
\end{array}
\\
\,[f_{ 8 , 7 }] & f_{ 8 , 7 } :=
\begin{array}{l}
x_{0} x_{1} x_{2} + x_{0} x_{1} + x_{0} x_{2} + x_{0} x_{3} + x_{1} x_{3} x_{4} + x_{1} x_{4}\, +
x_{1} x_{5}\, +
\\
x_{2} x_{3} x_{5} + x_{2} x_{4} + x_{6} x_{7}
\end{array}
\\
\,[f_{ 8 , 8 }] & f_{ 8 , 8 } :=
\begin{array}{l}
x_{0} x_{1} x_{2} + x_{0} x_{5} + x_{1} x_{3} x_{4} + x_{1} x_{6} + x_{2} x_{3} x_{5} + x_{2} x_{4}
+ x_{3} x_{7}
\end{array}
\\
\,[f_{ 8 , 9 }] & f_{ 8 , 9 } :=
\begin{array}{l}
x_{0} x_{1} x_{6} + x_{0} x_{3} + x_{1} x_{4} + x_{2} x_{3} x_{6} + x_{2} x_{5} + x_{3} x_{4}\, +
x_{4} x_{5} x_{6} + x_{6} x_{7}
\end{array}
\\
\,[f_{ 8 , 10 }] & f_{ 8 , 10 } :=
\begin{array}{l}
x_{0} x_{1} x_{2} + x_{0} x_{3} x_{6} + x_{0} x_{4} + x_{0} x_{5} + x_{1} x_{3} x_{4} + x_{1} x_{6}
+ x_{2} x_{3} x_{5}\, +
\\
x_{2} x_{4} + x_{3} x_{7}
\end{array}
\\
\hline
\end{array}
\end{align*}
\normalsize{}
\caption{8 dimensions to degree 3: ET classes.}
\label{tab-c8_ET_classes}
\end{table}
We here examine the corresponding ET classes in detail.
\newpage
\paragraph*{ET class $[f_{8,1}]$.}
This is the ET class of the bent function
\small{}
\begin{align*}
f_{ 8 , 1 } &=
\begin{array}{l}
x_{0} x_{1} + x_{2} x_{3} + x_{4} x_{5} + x_{6} x_{7}.
\end{array}
\end{align*}
\normalsize{}
This function is quadratic and self-dual.
The ET class contains two extended Cayley classes as per Table~\ref{tab-c8_1_EC_classes}.

\begin{table}[!bhpt] 
\small{}
\begin{align*}
\def\arraystretch{1.2}
\begin{array}{|cccl|}
\hline
\text{Class} &
\text{Parameters} &
\text{2-rank} &
\text{Clique polynomial}
\\
\hline
0 &
(256, 120, 56, 56) &
10 &
\begin{array}{l}
245760t^{9} + 3317760t^{8} + 8847360t^{7}
\,+
\\
 10321920t^{6} + 6193152t^{5} + 2007040t^{4}
\,+
\\
 286720t^{3} + 15360t^{2} + 256t + 1
\end{array}
\\
1 &
(256, 136, 72, 72) &
10 &
\begin{array}{l}
417792t^{8} + 3342336t^{7} + 11698176t^{6}
\,+
\\
 11698176t^{5} + 3760128t^{4} + 417792t^{3}
\,+
\\
 17408t^{2} + 256t + 1
\end{array}
\\
\hline
\end{array}
\end{align*}
\caption{$[f_{8,1}]$ extended Cayley classes.}
\label{tab-c8_1_EC_classes}
\end{table}

As expected from Theorem~\ref{th-Quadratic-Classes},
the two extended Cayley classes correspond to the two weight classes,
as shown in Figures~\ref{fig:c8_1_weight_class_matrix} and~\ref{fig:c8_1_bent_cayley_graph_index_matrix}.

Remark: The fractal-like self-similar quality of Figures~\ref{fig:c2_1_weight_class_matrix}, \ref{fig:c4_1_weight_class_matrix},
and \ref{fig:c6_1_weight_class_matrix} continues with Figure~\ref{fig:c8_1_weight_class_matrix}.

\begin{figure}[!bhpt] 
\centering
\begin{minipage}{.48\textwidth}
  \centering
  \includegraphics[width=.9\linewidth]{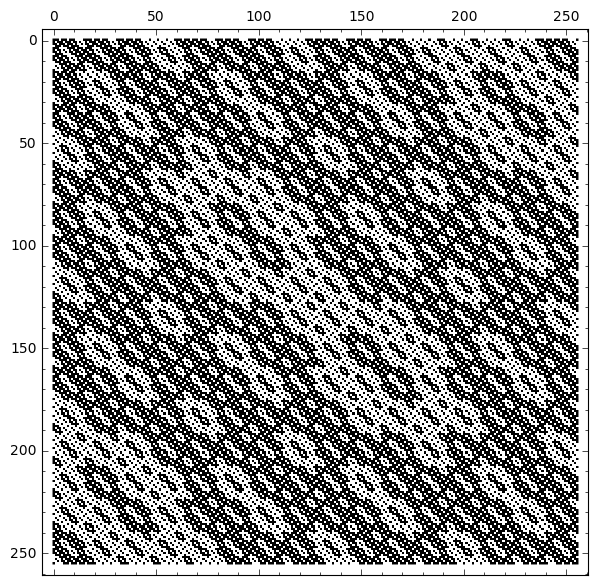}
  \captionof{figure}{$[f_{8,1}]$: weight classes. ~~~~\\~~~~}
  \label{fig:c8_1_weight_class_matrix}
\end{minipage}%
~~~~
\begin{minipage}{.48\textwidth}
  \centering
  \includegraphics[width=.9\linewidth]{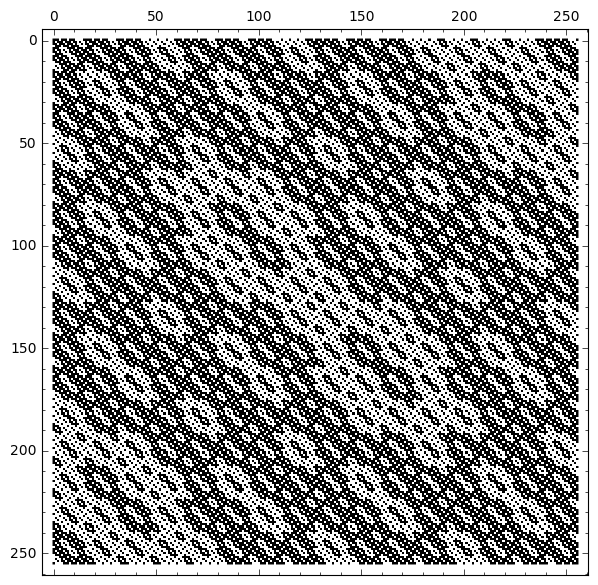}
  \captionof{figure}{$[f_{8,1}]$: extended Cayley classes.}
  \label{fig:c8_1_bent_cayley_graph_index_matrix}
\end{minipage}
\end{figure}

%
\paragraph*{ET class $[f_{8,2}]$.}
This is the ET class of the bent function
\small{}
\begin{align*}
f_{ 8 , 2 } &=
\begin{array}{l}
x_{0} x_{1} x_{2} + x_{0} x_{3} + x_{1} x_{4} + x_{2} x_{5} + x_{6} x_{7}.
\end{array}
\end{align*}
\normalsize{}
The ET class contains four extended Cayley classes as per Table~\ref{tab-c8_2_EC_classes}.

\begin{table}[!bhpt] 
\small{}
\begin{align*}
\def\arraystretch{1.2}
\begin{array}{|cccl|}
\hline
\text{Class} &
\text{Parameters} &
\text{2-rank} &
\text{Clique polynomial}
\\
\hline
0 &
(256, 120, 56, 56) &
10 &
\begin{array}{l}
245760t^{9} + 3317760t^{8} + 8847360t^{7}
\,+
\\
 10321920t^{6} + 6193152t^{5} + 2007040t^{4}
\,+
\\
 286720t^{3} + 15360t^{2} + 256t + 1
\end{array}
\\
1 &
(256, 120, 56, 56) &
10 &
\begin{array}{l}
49152t^{9} + 663552t^{8} + 2555904t^{7}
\,+
\\
 5079040t^{6} + 4620288t^{5} + 1875968t^{4}
\,+
\\
 286720t^{3} + 15360t^{2} + 256t + 1
\end{array}
\\
2 &
(256, 136, 72, 72) &
10 &
\begin{array}{l}
327680t^{9} + 4055040t^{8} + 13828096t^{7}
\,+
\\
 22183936t^{6} + 14319616t^{5} + 3891200t^{4}
\,+
\\
 417792t^{3} + 17408t^{2} + 256t + 1
\end{array}
\\
3 &
(256, 136, 72, 72) &
10 &
\begin{array}{l}
417792t^{8} + 3342336t^{7} + 11698176t^{6}
\,+
\\
 11698176t^{5} + 3760128t^{4} + 417792t^{3}
\,+
\\
 17408t^{2} + 256t + 1
\end{array}
\\
\hline
\end{array}
\end{align*}
\caption{$[f_{8,2}]$ extended Cayley classes.}
\label{tab-c8_2_EC_classes}
\end{table}

\begin{figure}[!ht] 
\centering
\begin{minipage}{.48\textwidth}
  \centering
  \includegraphics[width=.9\linewidth]{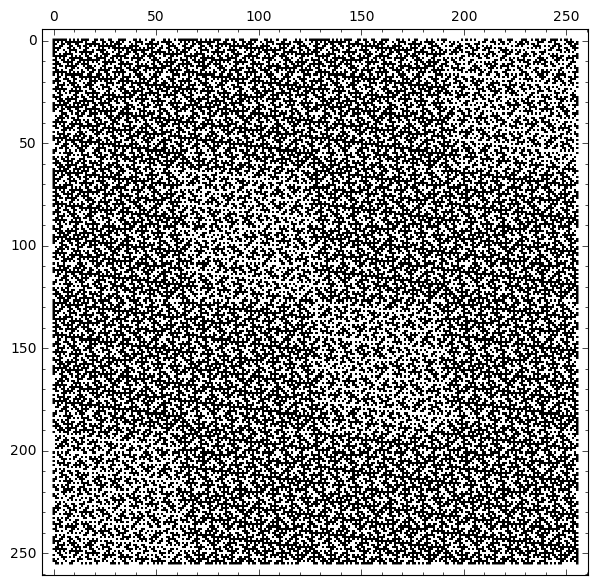}
  \captionof{figure}{$[f_{8,2}]$: weight classes. ~~~~\\~~~~}
  \label{fig:c8_2_weight_class_matrix}
\end{minipage}%
~~~~
\begin{minipage}{.48\textwidth}
  \centering
  \includegraphics[width=.9\linewidth]{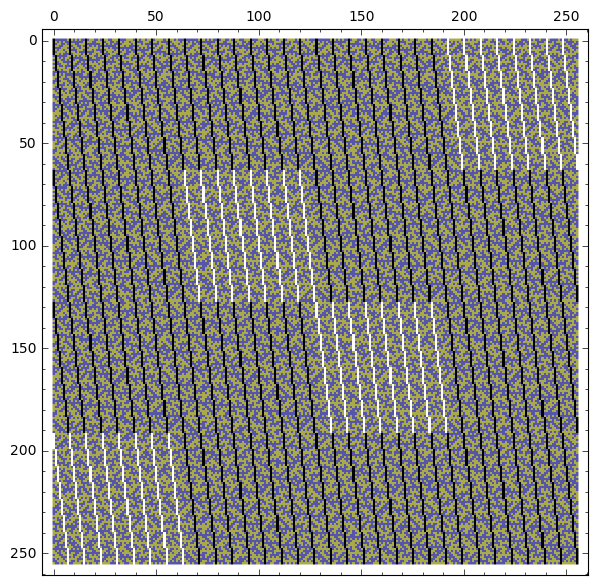}
  \captionof{figure}{$[f_{8,2}]$: extended Cayley classes.}
  \label{fig:c8_2_bent_cayley_graph_index_matrix}
\end{minipage}
\end{figure}
~
\newpage
The Cayley graph for class 0 is isomorphic to graph 0 of ET class $[f_{8,1}]$,
This reflects the fact that $f_{8,1} \equiv f_{8,2}$, even though these two functions are not
EA equivalent.
This is therefore an example of an isomorphism between Cayley graphs of bent functions on
$\F_2^8$ that is not a linear function.

The four extended Cayley classes are distributed between the two weight classes,
as shown in Figures~\ref{fig:c8_2_weight_class_matrix} and~\ref{fig:c8_2_bent_cayley_graph_index_matrix}.

%
\paragraph*{ET class $[f_{8,3}]$.}
This is the ET class of the bent function
\small{}
\begin{align*}
f_{ 8 , 3 } &=
\begin{array}{l}
x_{0} x_{1} x_{2} + x_{0} x_{6} + x_{1} x_{3} x_{4} + x_{1} x_{5} + x_{2} x_{3} + x_{4} x_{7}.
\end{array}
\end{align*}
\normalsize{}
The ET class contains six extended Cayley classes as per Table~\ref{tab-c8_3_EC_classes}.

\begin{table}[!bhpt] 
%
\small{}
\begin{align*}
\def\arraystretch{1.2}
\begin{array}{|cccl|}
\hline
\text{Class} &
\text{Parameters} &
\text{2-rank} &
\text{Clique polynomial}
\\
\hline
0 &
(256, 120, 56, 56) &
12 &
\begin{array}{l}
81920t^{9} + 1368064t^{8} + 4653056t^{7}
\,+
\\
 7176192t^{6} + 5406720t^{5} + 1941504t^{4}
\,+
\\
 286720t^{3} + 15360t^{2} + 256t + 1
\end{array}
\\
1 &
(256, 136, 72, 72) &
12 &
\begin{array}{l}
294912t^{9} + 6299648t^{8} + 21692416t^{7}
\,+
\\
 27951104t^{6} + 15630336t^{5} + 3956736t^{4}
\,+
\\
 417792t^{3} + 17408t^{2} + 256t + 1
\end{array}
\\
2 &
(256, 120, 56, 56) &
12 &
\begin{array}{l}
16384t^{9} + 221184t^{8} + 1277952t^{7}
\,+
\\
 3768320t^{6} + 4227072t^{5} + 1843200t^{4}
\,+
\\
 286720t^{3} + 15360t^{2} + 256t + 1
\end{array}
\\
3 &
(256, 136, 72, 72) &
12 &
\begin{array}{l}
262144t^{9} + 4399104t^{8} + 16220160t^{7}
\,+
\\
 24281088t^{6} + 14974976t^{5} + 3923968t^{4}
\,+
\\
 417792t^{3} + 17408t^{2} + 256t + 1
\end{array}
\\
4 &
(256, 120, 56, 56) &
12 &
\begin{array}{l}
49152t^{9} + 729088t^{8} + 2686976t^{7}
\,+
\\
 5079040t^{6} + 4620288t^{5} + 1875968t^{4}
\,+
\\
 286720t^{3} + 15360t^{2} + 256t + 1
\end{array}
\\
5 &
(256, 136, 72, 72) &
12 &
\begin{array}{l}
196608t^{9} + 3399680t^{8} + 13172736t^{7}
\,+
\\
 21659648t^{6} + 14319616t^{5} + 3891200t^{4}
\,+
\\
 417792t^{3} + 17408t^{2} + 256t + 1
\end{array}
\\
\hline
\end{array}
\end{align*}
\caption{$[f_{8,3}]$ extended Cayley classes.}
\label{tab-c8_3_EC_classes}
\end{table}

The six extended Cayley classes are distributed between the two weight classes,
as shown in Figures~\ref{fig:c8_3_weight_class_matrix} and~\ref{fig:c8_3_bent_cayley_graph_index_matrix}.

\begin{figure}[!bhpt] 
\centering
\begin{minipage}{.48\textwidth}
  \centering
  \includegraphics[width=.9\linewidth]{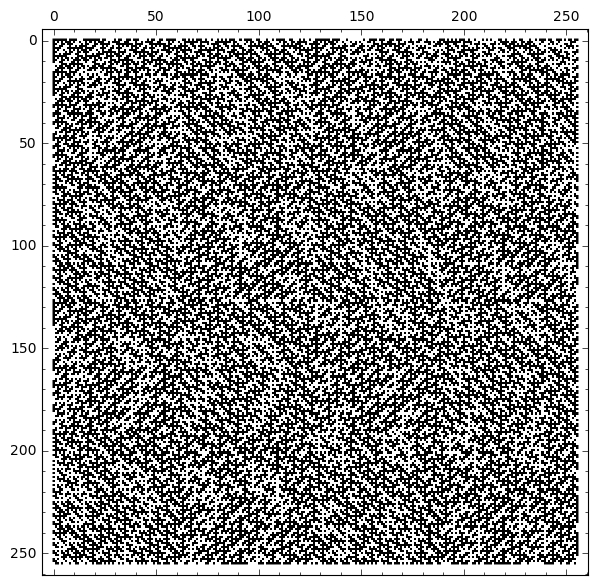}
  \captionof{figure}{$[f_{8,3}]$: weight classes. ~~~~\\~~~~}
  \label{fig:c8_3_weight_class_matrix}
\end{minipage}%
~~~~
\begin{minipage}{.48\textwidth}
  \centering
  \includegraphics[width=.9\linewidth]{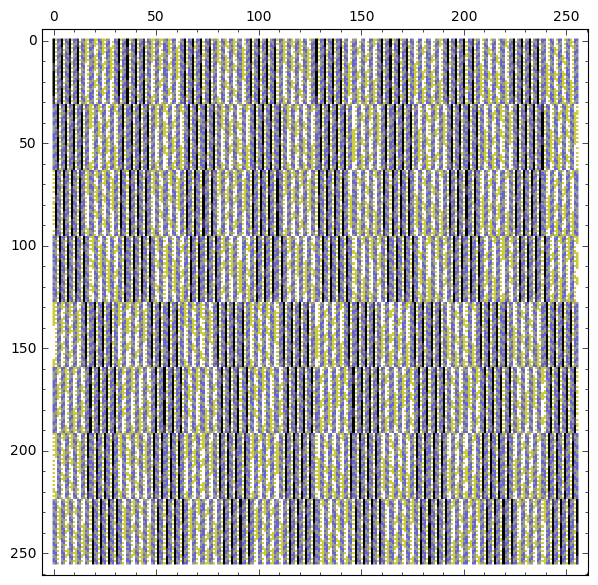}
  \captionof{figure}{$[f_{8,3}]$: extended Cayley classes.}
  \label{fig:c8_3_bent_cayley_graph_index_matrix}
\end{minipage}
\end{figure}
~
\newpage
%
\paragraph*{ET class $[f_{8,4}]$.}
This is the ET class of the bent function
\small{}
\begin{align*}
f_{ 8 , 4 } &=
\begin{array}{l}
x_{0} x_{1} x_{2} + x_{0} x_{2} + x_{0} x_{4} + x_{1} x_{3} x_{4} + x_{1} x_{5} + x_{2} x_{3}\, +
x_{6} x_{7}.
\end{array}
\end{align*}
\normalsize{}
The ET class contains six extended Cayley classes as per Table~\ref{tab-c8_4_EC_classes}.

The six extended Cayley classes are distributed between the two weight classes,
as shown in Figures~\ref{fig:c8_4_weight_class_matrix} and~\ref{fig:c8_4_bent_cayley_graph_index_matrix}.

\begin{table}[!bhpt] 
\small{}
\begin{align*}
\def\arraystretch{1.2}
\begin{array}{|cccl|}
\hline
\text{Class} &
\text{Parameters} &
\text{2-rank} &
\text{Clique polynomial}
\\
\hline
0 &
(256, 120, 56, 56) &
14 &
\begin{array}{l}
69632t^{9} + 1099776t^{8} + 3784704t^{7}
\,+
\\
 6160384t^{6} + 5013504t^{5} + 1908736t^{4}
\,+
\\
 286720t^{3} + 15360t^{2} + 256t + 1
\end{array}
\\
1 &
(256, 136, 72, 72) &
14 &
\begin{array}{l}
225280t^{9} + 4319232t^{8} + 16203776t^{7}
\,+
\\
 24313856t^{6} + 14974976t^{5} + 3923968t^{4}
\,+
\\
 417792t^{3} + 17408t^{2} + 256t + 1
\end{array}
\\
2 &
(256, 120, 56, 56) &
14 &
\begin{array}{l}
1536t^{10} + 15360t^{9} + 209920t^{8}
\,+
\\
 1280000t^{7} + 3751936t^{6} + 4227072t^{5}
\,+
\\
 1843200t^{4} + 286720t^{3} + 15360t^{2} + 256t + 1
\end{array}
\\
3 &
(256, 136, 72, 72) &
14 &
\begin{array}{l}
7680t^{10} + 230400t^{9} + 4228096t^{8}
\,+
\\
 16058368t^{7} + 24166400t^{6} + 14974976t^{5}
\,+
\\
 3923968t^{4} + 417792t^{3} + 17408t^{2} + 256t + 1
\end{array}
\\
4 &
(256, 136, 72, 72) &
14 &
\begin{array}{l}
110592t^{9} + 2344960t^{8} + 10305536t^{7}
\,+
\\
 18939904t^{6} + 13664256t^{5} + 3858432t^{4}
\,+
\\
 417792t^{3} + 17408t^{2} + 256t + 1
\end{array}
\\
5 &
(256, 120, 56, 56) &
14 &
\begin{array}{l}
20480t^{9} + 337920t^{8} + 1556480t^{7}
\,+
\\
 3932160t^{6} + 4227072t^{5} + 1843200t^{4}
\,+
\\
 286720t^{3} + 15360t^{2} + 256t + 1
\end{array}
\\
\hline
\end{array}
\end{align*}
\caption{$[f_{8,4]}$ extended Cayley classes.}
\label{tab-c8_4_EC_classes}
\end{table}

\begin{figure}[!bhpt] 
\centering
\begin{minipage}{.48\textwidth}
  \centering
  \includegraphics[width=.9\linewidth]{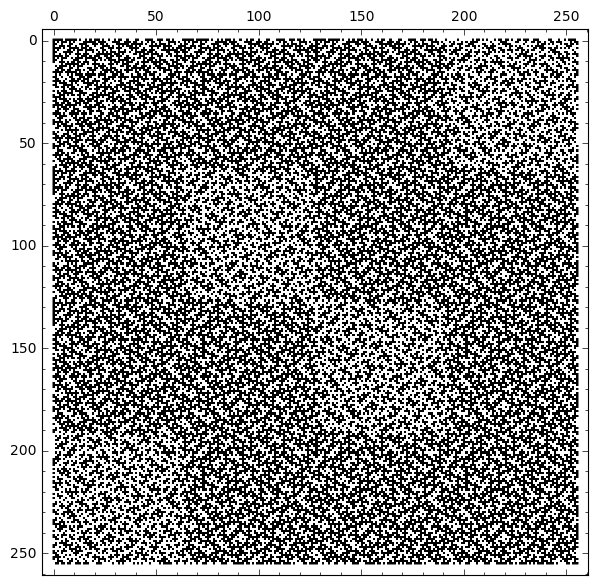}
  \captionof{figure}{$[f_{8,4}]$: weight classes. ~~~~\\~~~~}
  \label{fig:c8_4_weight_class_matrix}
\end{minipage}%
~~~~
\begin{minipage}{.48\textwidth}
  \centering
  \includegraphics[width=.9\linewidth]{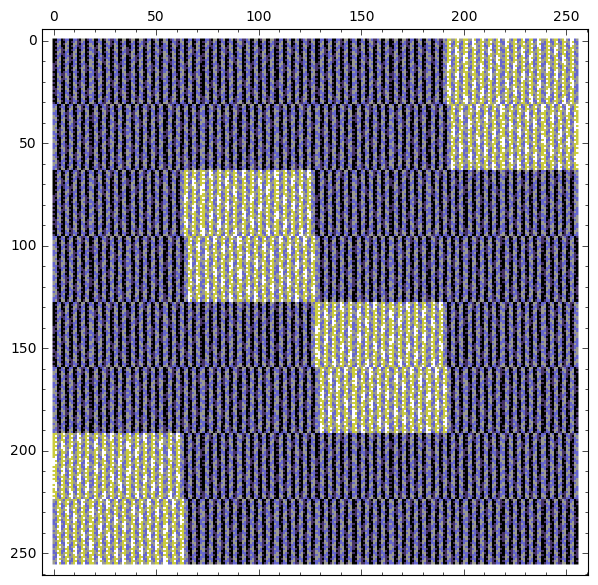}
  \captionof{figure}{$[f_{8,4}]$: extended Cayley classes.}
  \label{fig:c8_4_bent_cayley_graph_index_matrix}
\end{minipage}
\end{figure}
~
\paragraph*{ET class $[f_{8,5}]$.}
This is the ET class of the bent function
\small{}
\begin{align*}
f_{ 8 , 5 } &=
\begin{array}{l}
x_{0} x_{1} x_{2} + x_{0} x_{6} + x_{1} x_{3} x_{4} + x_{1} x_{4} + x_{1} x_{5} + x_{2} x_{3} x_{5}
+ x_{2} x_{4} + x_{3} x_{7}.
\end{array}
\end{align*}
\normalsize{}
The ET class contains 9 extended Cayley classes as per
Table~\ref{tab-c8_5_EC_classes}.

\begin{table}[!bhpt] 
\small{}
\begin{align*}
\def\arraystretch{1.2}
\begin{array}{|cccl|}
\hline
\text{Class} &
\text{Parameters} &
\text{2-rank} &
\text{Clique polynomial}
\\
\hline
0 &
(256, 120, 56, 56) &
14 &
\begin{array}{l}
32768t^{9} + 731136t^{8} + 3096576t^{7}
\,+
\\
 5767168t^{6} + 5013504t^{5} + 1908736t^{4}
\,+
\\
 286720t^{3} + 15360t^{2} + 256t + 1
\end{array}
\\
1 &
(256, 120, 56, 56) &
14 &
\begin{array}{l}
28672t^{9} + 534528t^{8} + 2211840t^{7}
\,+
\\
 4718592t^{6} + 4620288t^{5} + 1875968t^{4}
\,+
\\
 286720t^{3} + 15360t^{2} + 256t + 1
\end{array}
\\
2 &
(256, 136, 72, 72) &
14 &
\begin{array}{l}
159744t^{9} + 4753408t^{8} + 19021824t^{7}
\,+
\\
 26804224t^{6} + 15630336t^{5} + 3956736t^{4}
\,+
\\
 417792t^{3} + 17408t^{2} + 256t + 1
\end{array}
\\
3 &
(256, 120, 56, 56) &
14 &
\begin{array}{l}
24576t^{9} + 526336t^{8} + 2342912t^{7}
\,+
\\
 4849664t^{6} + 4620288t^{5} + 1875968t^{4}
\,+
\\
 286720t^{3} + 15360t^{2} + 256t + 1
\end{array}
\\
4 &
(256, 136, 72, 72) &
14 &
\begin{array}{l}
90112t^{9} + 2795520t^{8} + 12402688t^{7}
\,+
\\
 21168128t^{6} + 14319616t^{5} + 3891200t^{4}
\,+
\\
 417792t^{3} + 17408t^{2} + 256t + 1
\end{array}
\\
5 &
(256, 120, 56, 56) &
14 &
\begin{array}{l}
16384t^{9} + 284672t^{8} + 1392640t^{7}
\,+
\\
 3735552t^{6} + 4227072t^{5} + 1843200t^{4}
\,+
\\
 286720t^{3} + 15360t^{2} + 256t + 1
\end{array}
\\
6 &
(256, 136, 72, 72) &
14 &
\begin{array}{l}
131072t^{9} + 3577856t^{8} + 15319040t^{7}
\,+
\\
 23855104t^{6} + 14974976t^{5} + 3923968t^{4}
\,+
\\
 417792t^{3} + 17408t^{2} + 256t + 1
\end{array}
\\
7 &
(256, 120, 56, 56) &
14 &
\begin{array}{l}
1536t^{10} + 19456t^{9} + 279552t^{8}
\,+
\\
 1394688t^{7} + 3751936t^{6} + 4227072t^{5}
\,+
\\
 1843200t^{4} + 286720t^{3} + 15360t^{2} + 256t + 1
\end{array}
\\
8 &
(256, 136, 72, 72) &
14 &
\begin{array}{l}
5632t^{10} + 148480t^{9} + 3621888t^{8}
\,+
\\
 15206400t^{7} + 23773184t^{6} + 14974976t^{5}
\,+
\\
 3923968t^{4} + 417792t^{3} + 17408t^{2} + 256t + 1
\end{array}
\\
\hline
\end{array}
\end{align*}
\caption{$[f_{8,5}]$ extended Cayley classes.}
\label{tab-c8_5_EC_classes}
\end{table}
\newpage
The 9 extended Cayley classes are distributed between the two weight classes,
as shown in Figures~\ref{fig:c8_5_weight_class_matrix} and~\ref{fig:c8_5_bent_cayley_graph_index_matrix}.

\begin{figure}[!bhpt] 
\centering
\begin{minipage}{.48\textwidth}
  \centering
  \includegraphics[width=.9\linewidth]{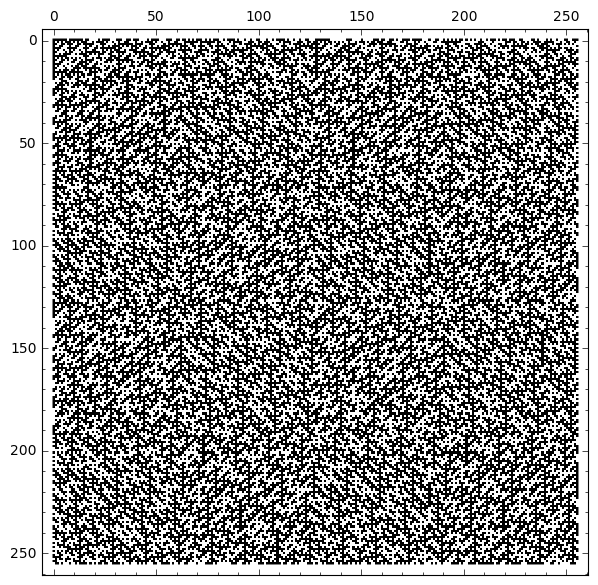}
  \captionof{figure}{$[f_{8,5}]$: weight classes. ~~~~\\~~~~}
  \label{fig:c8_5_weight_class_matrix}
\end{minipage}%
~~~~
\begin{minipage}{.48\textwidth}
  \centering
  \includegraphics[width=.9\linewidth]{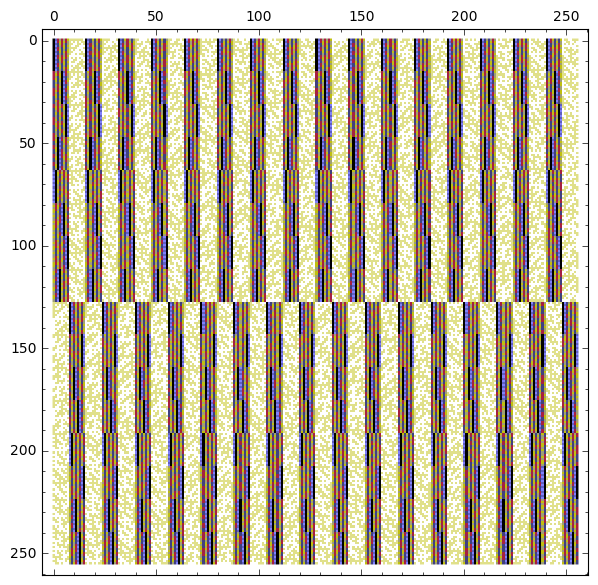}
  \captionof{figure}{$[f_{8,5}]$: extended Cayley classes.}
  \label{fig:c8_5_bent_cayley_graph_index_matrix}
\end{minipage}
\end{figure}
~
\paragraph*{ET class $[f_{8,6}]$.}
This is the ET class of the bent function
\small{}
\begin{align*}
f_{ 8 , 6 } &=
\begin{array}{l}
x_{0} x_{1} x_{2} + x_{0} x_{2} + x_{0} x_{3} + x_{1} x_{3} x_{4} + x_{1} x_{6} + x_{2} x_{3} x_{5}
+ x_{2} x_{4} + x_{5} x_{7}.
\end{array}
\end{align*}
\normalsize{}
The ET class contains 9 extended Cayley classes.

\begin{figure}[!hb] 
\centering
\begin{minipage}{.48\textwidth}
  \centering
  \includegraphics[width=.9\linewidth]{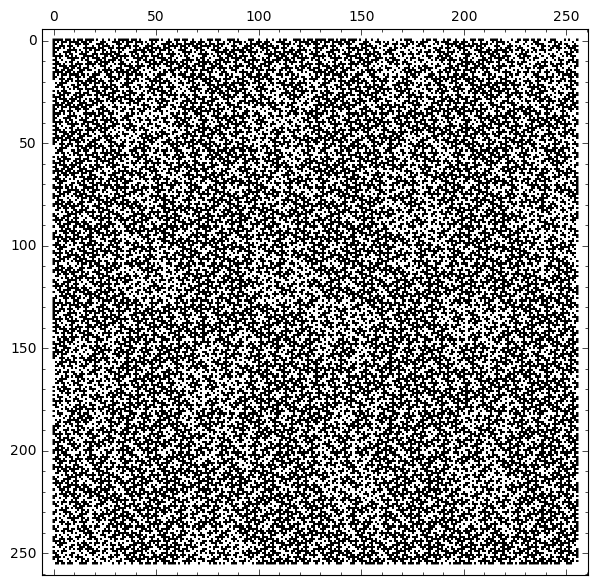}
  \captionof{figure}{$[f_{8,6}]$: weight classes. ~~~~\\~~~~}
  \label{fig:c8_6_weight_class_matrix}
\end{minipage}%
~~~~
\begin{minipage}{.48\textwidth}
  \centering
  \includegraphics[width=.9\linewidth]{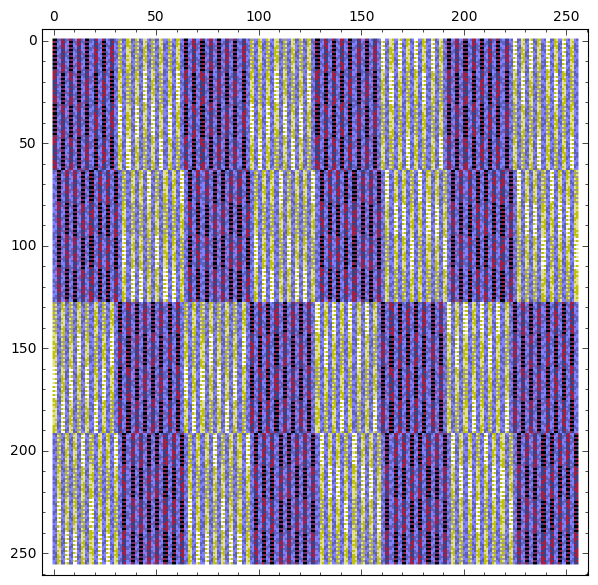}
  \captionof{figure}{$[f_{8,6}]$: extended Cayley classes.}
  \label{fig:c8_6_bent_cayley_graph_index_matrix}
\end{minipage}
\end{figure}

The 9 extended Cayley classes are distributed between the two weight classes,
as shown in Figures~\ref{fig:c8_6_weight_class_matrix} and~\ref{fig:c8_6_bent_cayley_graph_index_matrix}.

The 9 Cayley graphs corresponding to the 9 classes are isomorphic to those for the extended Cayley classes for $[f_{8,5}]$.
The corresponding Cayley classes have the same frequency within each of these two ET classes.
This correspondence is shown in Table~\ref{tab-c8_5-c8_6_EC_classes}.

Figures~\ref{fig:re8_5_bent_cayley_graph_index_matrix} and~\ref{fig:re8_5_bent_cayley_graph_index_matrix}
show the 9 extended Cayley classes of each of $[f_{8,5}]$ and $[f_{8,6}]$ with corresponding Cayley classes
given the same colour.

\begin{table}[!bhpt] 
\small{}
\begin{align*}
\def\arraystretch{1.2}
\begin{array}{|ccr|}
\hline
[f_{8,5}] &
[f_{8,6}] &
\text{Frequency}
\\
\hline
  0 &    0 &  4096
\\
  1 &    1 &  6144
\\
  2 &    2 &  6144
\\
  3 &    5 &  2048
\\
  4 &    8 &  2048
\\
  5 &    6 &  6144
\\
  6 &    7 &  6144
\\
  7 &    3 & 16384
\\
  8 &    4 & 16384
\\
\hline
\end{array}
\end{align*}
\caption{Correspondence between $[f_{8,5}]$ and $[f_{8,6}]$ extended Cayley classes.}
\label{tab-c8_5-c8_6_EC_classes}
\end{table}

\begin{figure}[!bhpt] 
\centering
\begin{minipage}{.48\textwidth}
  \centering
  \includegraphics[width=.9\linewidth]{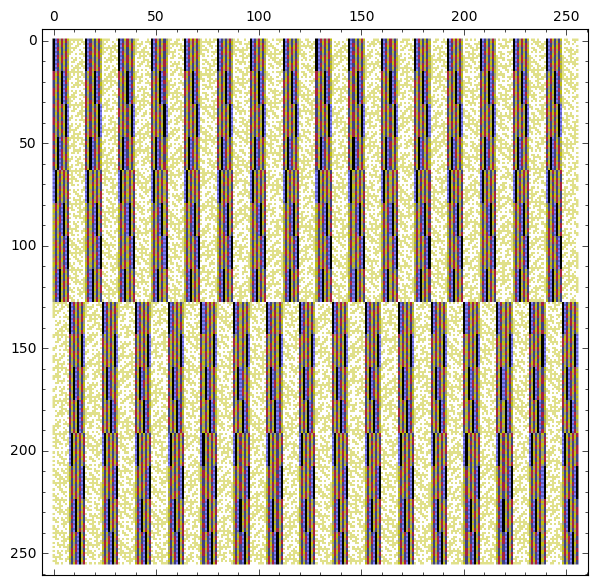}
  \captionof{figure}{$[f_{8,5}]$: extended Cayley classes (recoloured).}
  \label{fig:re8_5_bent_cayley_graph_index_matrix}
\end{minipage}
~~
\begin{minipage}{.48\textwidth}
  \centering
  \includegraphics[width=.9\linewidth]{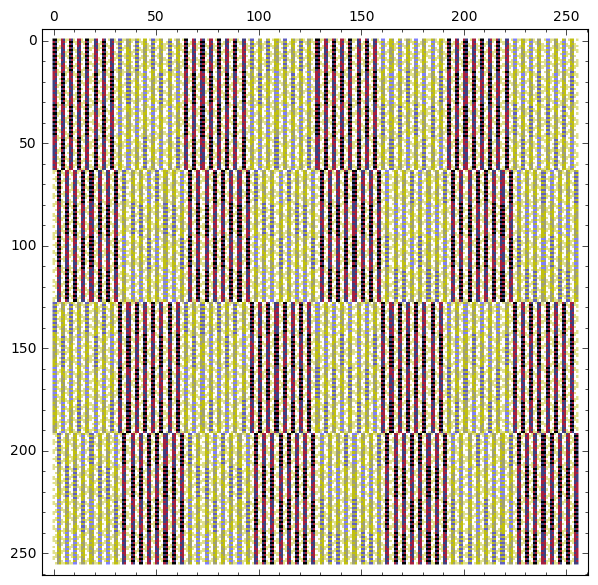}
  \captionof{figure}{$[f_{8,6}]$: extended Cayley classes (recoloured).}
  \label{fig:re8_6_bent_cayley_graph_index_matrix}
\end{minipage}
\end{figure}
\newpage
The explanation for the correspondence between the Cayley classes of $f_{8,5}$ and $f_{8,6}$
is quite simple. The functions $f_{8,5}$ and $f_{8,6}$ are EA equivalent,
in fact general linear equivalent, and therefore Braeken's list of EA equivalence classes
\cite[Section 5.5.2]{Bra06thesis} contains an error.

\begin{Theorem}
\label{th-f8-5-f8-6-linearly-equiv}
Functions $f_{8,5}$ and $f_{8,6}$ are general linear equivalent.
\end{Theorem}

\begin{proof}

Apply the permutation $\pi := (x_0\ x_5\ x_4)(x_1\ x_2\ x_3)(x_6\ x_7)$ to
\small{
\begin{align*}
f_{8,5}
&=
x_{0} x_{1} x_{2} + x_{0} x_{6} + x_{1} x_{3} x_{4} + x_{1} x_{4} + x_{1} x_{5} + x_{2} x_{3} x_{5} + x_{2} x_{4} + x_{3} x_{7}
\intertext{\normalsize{to obtain}}
\pi(f_{8,5})
&=
x_{5} x_{2} x_{3} + x_{5} x_{7} + x_{2} x_{1} x_{0} + x_{2} x_{0} + x_{2} x_{4} + x_{3} x_{1} x_{4} + x_{3} x_{0} + x_{1} x_{6}
\\
&=
x_{0} x_{1} x_{2} + x_{0} x_{2} + x_{0} x_{3} + x_{1} x_{3} x_{4} + x_{1} x_{6} + x_{2} x_{3} x_{5} + x_{2} x_{4} + x_{5} x_{7}
\\
&= f_{8,6}.
\end{align*}
}\normalsize{}
\end{proof}

\paragraph*{ET class $[f_{8,7}]$.}
This is the ET class of the bent function
\small{}
\begin{align*}
f_{ 8 , 7 } &=
\begin{array}{l}
x_{0} x_{1} x_{2} + x_{0} x_{1} + x_{0} x_{2} + x_{0} x_{3} + x_{1} x_{3} x_{4} + x_{1} x_{4}\, +
x_{1} x_{5}\, +
\\
x_{2} x_{3} x_{5} + x_{2} x_{4} + x_{6} x_{7}.
\end{array}
\end{align*}
\normalsize{}
The ET class contains six extended Cayley classes as per Table~\ref{tab-c8_7_EC_classes}.

The six extended Cayley classes are distributed between the two weight classes,
as shown in Figures~\ref{fig:c8_7_weight_class_matrix} and~\ref{fig:c8_7_bent_cayley_graph_index_matrix}.

\begin{table}[!bhpt] 
\small{}
\begin{align*}
\def\arraystretch{1.2}
\begin{array}{|cccl|}
\hline
\text{Class} &
\text{Parameters} &
\text{2-rank} &
\text{Clique polynomial}
\\
\hline
0 &
(256, 120, 56, 56) &
16 &
\begin{array}{l}
29696t^{9} + 655360t^{8} + 2789376t^{7}
\,+
\\
 5332992t^{6} + 4816896t^{5} + 1892352t^{4}
\,+
\\
 286720t^{3} + 15360t^{2} + 256t + 1
\end{array}
\\
1 &
(256, 120, 56, 56) &
16 &
\begin{array}{l}
20480t^{9} + 409600t^{8} + 1837056t^{7}
\,+
\\
 4235264t^{6} + 4423680t^{5} + 1859584t^{4}
\,+
\\
 286720t^{3} + 15360t^{2} + 256t + 1
\end{array}
\\
2 &
(256, 136, 72, 72) &
16 &
\begin{array}{l}
143360t^{9} + 3981312t^{8} + 16697344t^{7}
\,+
\\
 25108480t^{6} + 15302656t^{5} + 3940352t^{4}
\,+
\\
 417792t^{3} + 17408t^{2} + 256t + 1
\end{array}
\\
3 &
(256, 136, 72, 72) &
16 &
\begin{array}{l}
64512t^{9} + 2316288t^{8} + 10932224t^{7}
\,+
\\
 19783680t^{6} + 13991936t^{5} + 3874816t^{4}
\,+
\\
 417792t^{3} + 17408t^{2} + 256t + 1
\end{array}
\\
4 &
(256, 136, 72, 72) &
16 &
\begin{array}{l}
92160t^{9} + 2979840t^{8} + 13608960t^{7}
\,+
\\
 22388736t^{6} + 14647296t^{5} + 3907584t^{4}
\,+
\\
 417792t^{3} + 17408t^{2} + 256t + 1
\end{array}
\\
5 &
(256, 120, 56, 56) &
16 &
\begin{array}{l}
6144t^{9} + 124928t^{8} + 944128t^{7}
\,+
\\
 3219456t^{6} + 4030464t^{5} + 1826816t^{4}
\,+
\\
 286720t^{3} + 15360t^{2} + 256t + 1
\end{array}
\\
\hline
\end{array}
\end{align*}

\caption{$[f_{8,7}]$ extended Cayley classes.}
\label{tab-c8_7_EC_classes}
\end{table}

\begin{figure}[!bhpt] 
\centering
\begin{minipage}{.48\textwidth}
  \centering
  \includegraphics[width=.9\linewidth]{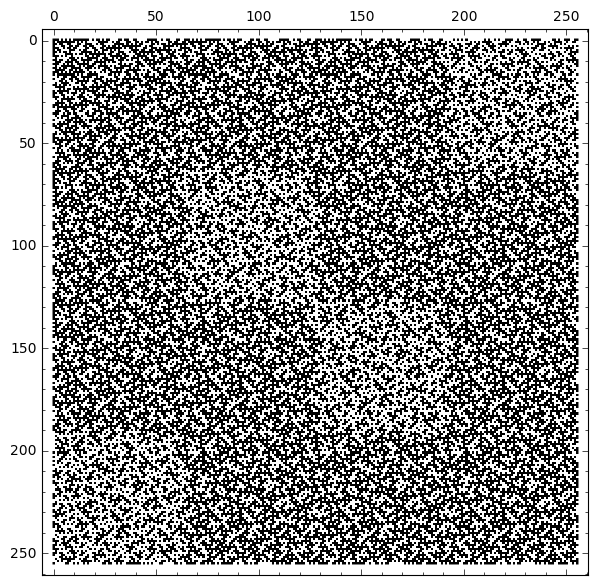}
  \captionof{figure}{$[f_{8,7}]$: weight classes. ~~~~\\~~~~}
  \label{fig:c8_7_weight_class_matrix}
\end{minipage}%
~~~~
\begin{minipage}{.48\textwidth}
  \centering
  \includegraphics[width=.9\linewidth]{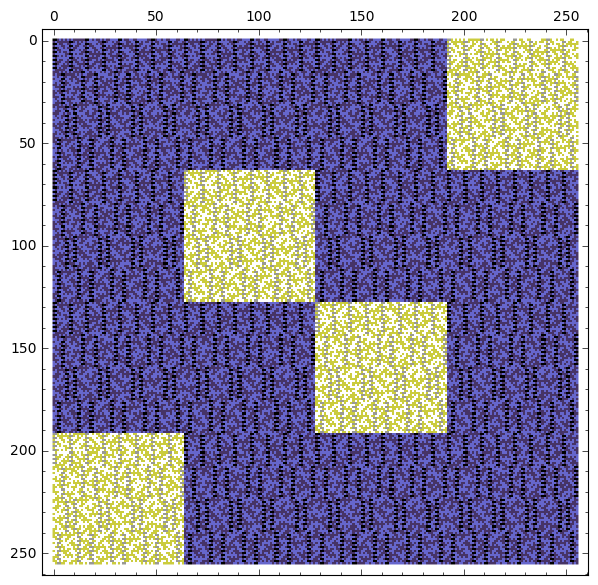}
  \captionof{figure}{$[f_{8,7}]$: extended Cayley classes.}
  \label{fig:c8_7_bent_cayley_graph_index_matrix}
\end{minipage}
\end{figure}

\newpage
\paragraph*{ET class $[f_{8,8}]$.}
This is the ET class of the bent function
\small{}
\begin{align*}
f_{ 8 , 8 } &=
\begin{array}{l}
x_{0} x_{1} x_{2} + x_{0} x_{5} + x_{1} x_{3} x_{4} + x_{1} x_{6} + x_{2} x_{3} x_{5} + x_{2} x_{4}
+ x_{3} x_{7}.
\end{array}
\end{align*}
\normalsize{}
The ET class contains six extended Cayley classes as per Table~\ref{tab-c8_8_EC_classes}.

The six extended Cayley classes are distributed between the two weight classes,
as shown in Figures~\ref{fig:c8_8_weight_class_matrix} and~\ref{fig:c8_8_bent_cayley_graph_index_matrix}.

\begin{table}[!bhpt] 
\small{}
\begin{align*}
\def\arraystretch{1.2}
\begin{array}{|cccl|}
\hline
\text{Class} &
\text{Parameters} &
\text{2-rank} &
\text{Clique polynomial}
\\
\hline
0 &
(256, 120, 56, 56) &
14 &
\begin{array}{l}
32768t^{9} + 712704t^{8} + 3014656t^{7}
\,+
\\
 5734400t^{6} + 5013504t^{5} + 1908736t^{4}
\,+
\\
 286720t^{3} + 15360t^{2} + 256t + 1
\end{array}
\\
1 &
(256, 120, 56, 56) &
14 &
\begin{array}{l}
24576t^{9} + 466944t^{8} + 2064384t^{7}
\,+
\\
 4685824t^{6} + 4620288t^{5} + 1875968t^{4}
\,+
\\
 286720t^{3} + 15360t^{2} + 256t + 1
\end{array}
\\
2 &
(256, 136, 72, 72) &
14 &
\begin{array}{l}
172032t^{9} + 5332992t^{8} + 20283392t^{7}
\,+
\\
 27295744t^{6} + 15630336t^{5} + 3956736t^{4}
\,+
\\
 417792t^{3} + 17408t^{2} + 256t + 1
\end{array}
\\
3 &
(256, 136, 72, 72) &
14 &
\begin{array}{l}
147456t^{9} + 3858432t^{8} + 15990784t^{7}
\,+
\\
 24150016t^{6} + 14974976t^{5} + 3923968t^{4}
\,+
\\
 417792t^{3} + 17408t^{2} + 256t + 1
\end{array}
\\
4 &
(256, 120, 56, 56) &
14 &
\begin{array}{l}
16384t^{9} + 270336t^{8} + 1376256t^{7}
\,+
\\
 3768320t^{6} + 4227072t^{5} + 1843200t^{4}
\,+
\\
 286720t^{3} + 15360t^{2} + 256t + 1
\end{array}
\\
5 &
(256, 136, 72, 72) &
14 &
\begin{array}{l}
163840t^{9} + 3858432t^{8} + 15532032t^{7}
\,+
\\
 23887872t^{6} + 14974976t^{5} + 3923968t^{4}
\,+
\\
 417792t^{3} + 17408t^{2} + 256t + 1
\end{array}
\\
\hline
\end{array}
\end{align*}
\caption{$[f_{8,8}]$ extended Cayley classes.}
\label{tab-c8_8_EC_classes}
\end{table}

\begin{figure}[!bhpt] 
\centering
\begin{minipage}{.48\textwidth}
  \centering
  \includegraphics[width=.9\linewidth]{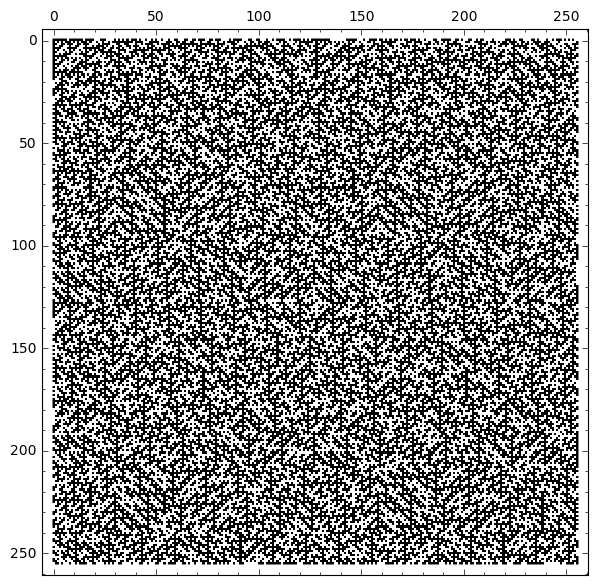}
  \captionof{figure}{$[f_{8,8}]$: weight classes. ~~~~\\~~~~}
  \label{fig:c8_8_weight_class_matrix}
\end{minipage}%
~~~~
\begin{minipage}{.48\textwidth}
  \centering
  \includegraphics[width=.9\linewidth]{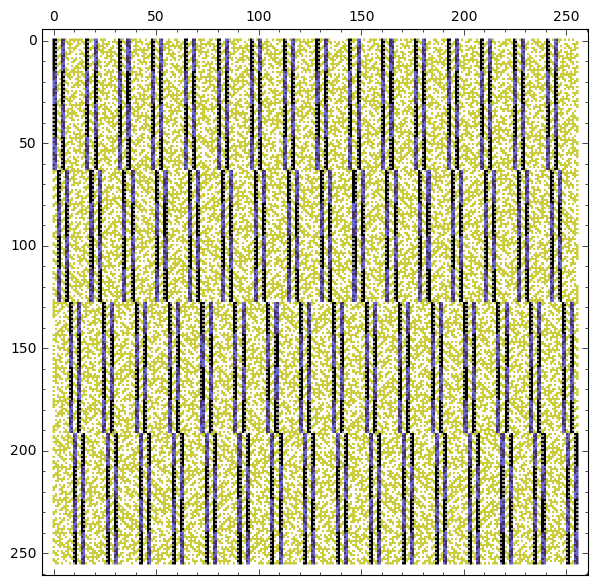}
  \captionof{figure}{$[f_{8,8}]$: extended Cayley classes.}
  \label{fig:c8_8_bent_cayley_graph_index_matrix}
\end{minipage}
\end{figure}

\newpage
\paragraph*{ET class $[f_{8,9}]$.}
This is the ET class of the bent function
\small{}
\begin{align*}
f_{ 8 , 9 } &=
\begin{array}{l}
x_{0} x_{1} x_{6} + x_{0} x_{3} + x_{1} x_{4} + x_{2} x_{3} x_{6} + x_{2} x_{5} + x_{3} x_{4}\, +
x_{4} x_{5} x_{6} + x_{6} x_{7}.
\end{array}
\end{align*}
\normalsize{}
The ET class contains 8 extended Cayley classes as per
Table~\ref{tab-c8_9_EC_classes}.
In 4 of these 8 classes, each bent function is extended Cayley equivalent to its dual.
In the remaining 4 extended Cayley classes, the dual of every bent function in the class has a Cayley graph
that is isomorphic to that of one other class. That is, the 4 Cayley classes form two duality pairs of classes.
This correspondence between Cayley classes and Cayley classes of duals is shown in Table~\ref{tab-c8_9-dual-EC_classes}.

\begin{figure}[!bhpt] 
\centering
\begin{minipage}{.48\textwidth}
  \centering
  \includegraphics[width=.9\linewidth]{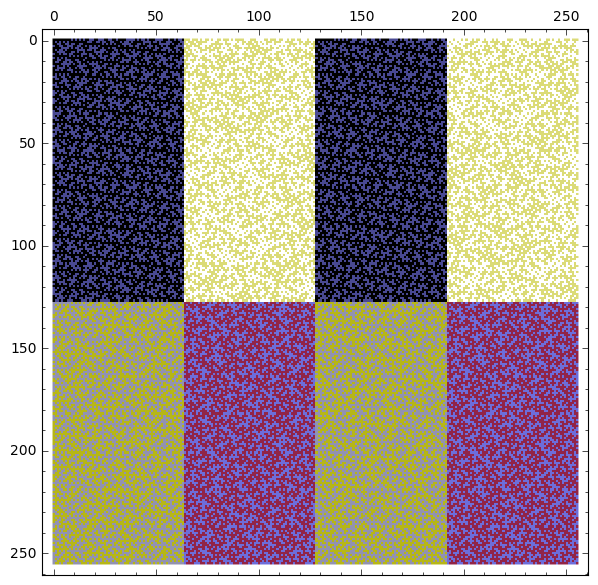}
  \captionof{figure}{$[f_{8,9}]$: 8 extended Cayley classes ~~ ~~~~ ~~~~ ~~~~~~~~~}
  \label{fig:c8_9_bent_cayley_graph_index_matrix}
\end{minipage}
~~
\begin{minipage}{.48\textwidth}
  \centering
  \includegraphics[width=.9\linewidth]{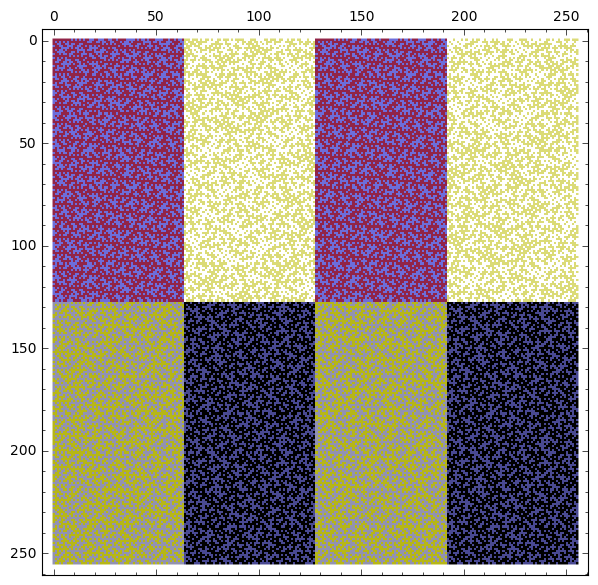}
  \captionof{figure}{$[f_{8,9}]$: 8 extended Cayley classes of dual bent functions}
  \label{fig:c8_9_dual_cayley_graph_index_matrix}
\end{minipage}%
\end{figure}

The 8 extended Cayley classes are distributed
as shown in Figures~\ref{fig:c8_9_bent_cayley_graph_index_matrix} (classes) and~\ref{fig:c8_9_dual_cayley_graph_index_matrix}
(classes of duals).

\begin{table}[!bhpt] 
\small{}
\begin{align*}
\def\arraystretch{1.2}
\begin{array}{|cccl|}
\hline
\text{Class} &
\text{Parameters} &
\text{2-rank} &
\text{Clique polynomial}
\\
\hline
0 &
(256, 120, 56, 56) &
16 &
\begin{array}{l}
45056t^{9} + 780288t^{8} + 2998272t^{7}
\,+
\\
 5505024t^{6} + 4816896t^{5} + 1892352t^{4}
\,+
\\
 286720t^{3} + 15360t^{2} + 256t + 1
\end{array}
\\
1 &
(256, 120, 56, 56) &
16 &
\begin{array}{l}
45056t^{9} + 780288t^{8} + 2998272t^{7}
\,+
\\
 5505024t^{6} + 4816896t^{5} + 1892352t^{4}
\,+
\\
 286720t^{3} + 15360t^{2} + 256t + 1
\end{array}
\\
2 &
(256, 136, 72, 72) &
16 &
\begin{array}{l}
184320t^{9} + 3852288t^{8} + 14893056t^{7}
\,+
\\
 23003136t^{6} + 14647296t^{5} + 3907584t^{4}
\,+
\\
 417792t^{3} + 17408t^{2} + 256t + 1
\end{array}
\\
3 &
(256, 136, 72, 72) &
16 &
\begin{array}{l}
184320t^{9} + 3852288t^{8} + 14893056t^{7}
\,+
\\
 23003136t^{6} + 14647296t^{5} + 3907584t^{4}
\,+
\\
 417792t^{3} + 17408t^{2} + 256t + 1
\end{array}
\\
4 &
(256, 120, 56, 56) &
16 &
\begin{array}{l}
105984t^{8} + 976896t^{7} + 3440640t^{6}
\,+
\\
 4128768t^{5} + 1835008t^{4} + 286720t^{3}
\,+
\\
 15360t^{2} + 256t + 1
\end{array}
\\
5 &
(256, 136, 72, 72) &
16 &
\begin{array}{l}
9216t^{10} + 264192t^{9} + 4468224t^{8}
\,+
\\
 16803840t^{7} + 24772608t^{6} + 15138816t^{5}
\,+
\\
 3932160t^{4} + 417792t^{3} + 17408t^{2} + 256t + 1
\end{array}
\\
6 &
(256, 120, 56, 56) &
16 &
\begin{array}{l}
9216t^{9} + 124416t^{8} + 976896t^{7}
\,+
\\
 3440640t^{6} + 4128768t^{5} + 1835008t^{4}
\,+
\\
 286720t^{3} + 15360t^{2} + 256t + 1
\end{array}
\\
7 &
(256, 136, 72, 72) &
16 &
\begin{array}{l}
193536t^{9} + 4449792t^{8} + 16803840t^{7}
\,+
\\
 24772608t^{6} + 15138816t^{5} + 3932160t^{4}
\,+
\\
 417792t^{3} + 17408t^{2} + 256t + 1
\end{array}
\\
\hline
\end{array}
\end{align*}
\caption{$[f_{8,9}]$ extended Cayley classes.}
\label{tab-c8_9_EC_classes}
\end{table}

\begin{table}
\small{}
\begin{align*}
\def\arraystretch{1.2}
\begin{array}{|ccc|}
\hline
[f_{8,9}] &
[f_{8,9}] &
\text{Frequency}
\\
&
\text{duals} &
\\
\hline
0 &   1 & 9216
\\
1 &   0 & 9216
\\
2 &   3 & 7168
\\
3 &   2 & 7168
\\
4 &   4 & 8192
\\
5 &   5 & 8192
\\
6 &   6 & 8192
\\
7 &   7 & 8192
\\
\hline
\end{array}
\end{align*}
\caption{Correspondence between $[f_{8,9}]$ extended Cayley classes and $[f_{8,9}]$ dual extended Cayley classes.}
\label{tab-c8_9-dual-EC_classes}
\end{table}
\newpage
\paragraph*{ET class $[f_{8,10}]$.}
This is the ET class of the bent function
\small{}
\begin{align*}
f_{ 8 , 10 } :=
\begin{array}{l}
x_{0} x_{1} x_{2} + x_{0} x_{3} x_{6} + x_{0} x_{4} + x_{0} x_{5} + x_{1} x_{3} x_{4} + x_{1} x_{6}
+ x_{2} x_{3} x_{5}\, +
\\
x_{2} x_{4} + x_{3} x_{7}.
\end{array}
\end{align*}
\normalsize{}
The ET class contains 10 extended Cayley classes as per
Table~\ref{tab-c8_10_EC_classes}.
In 4 of these 10 classes, each bent function is extended Cayley equivalent to its dual.
In the remaining 6 extended Cayley classes, the dual of every bent function in the class has a Cayley graph
that is isomorphic to that of one other class. That is, the 6 Cayley classes form three duality pairs of classes.
This correspondence between Cayley classes and Cayley classes of duals is shown in Table~\ref{tab-c8_10-dual-EC_classes}.

The 10 extended Cayley classes are distributed
as shown in Figures~\ref{fig:c8_10_bent_cayley_graph_index_matrix} (classes) and~\ref{fig:c8_10_dual_cayley_graph_index_matrix}
(classes of duals).
\begin{table}[!bhpt] 
\small{}
\begin{align*}
\def\arraystretch{1.2}
\begin{array}{|cccl|}
\hline
\text{Class} &
\text{Parameters} &
\text{2-rank} &
\text{Clique polynomial}
\\
\hline
0 &
(256, 120, 56, 56) &
16 &
\begin{array}{l}
16384t^{9} + 464896t^{8} + 2310144t^{7}
\,+
\\
 5046272t^{6} + 4816896t^{5} + 1892352t^{4}
\,+
\\
 286720t^{3} + 15360t^{2} + 256t + 1
\end{array}
\\
1 &
(256, 120, 56, 56) &
16 &
\begin{array}{l}
16384t^{9} + 464896t^{8} + 2310144t^{7}
\,+
\\
 5046272t^{6} + 4816896t^{5} + 1892352t^{4}
\,+
\\
 286720t^{3} + 15360t^{2} + 256t + 1
\end{array}
\\
2 &
(256, 120, 56, 56) &
16 &
\begin{array}{l}
12288t^{9} + 301056t^{8} + 1589248t^{7}
\,+
\\
 4128768t^{6} + 4423680t^{5} + 1859584t^{4}
\,+
\\
 286720t^{3} + 15360t^{2} + 256t + 1
\end{array}
\\
3 &
(256, 120, 56, 56) &
16 &
\begin{array}{l}
12288t^{9} + 301056t^{8} + 1589248t^{7}
\,+
\\
 4128768t^{6} + 4423680t^{5} + 1859584t^{4}
\,+
\\
 286720t^{3} + 15360t^{2} + 256t + 1
\end{array}
\\
4 &
(256, 136, 72, 72) &
16 &
\begin{array}{l}
110592t^{9} + 4159488t^{8} + 17285120t^{7}
\,+
\\
 25296896t^{6} + 15302656t^{5} + 3940352t^{4}
\,+
\\
 417792t^{3} + 17408t^{2} + 256t + 1
\end{array}
\\
5 &
(256, 136, 72, 72) &
16 &
\begin{array}{l}
110592t^{9} + 4159488t^{8} + 17285120t^{7}
\,+
\\
 25296896t^{6} + 15302656t^{5} + 3940352t^{4}
\,+
\\
 417792t^{3} + 17408t^{2} + 256t + 1
\end{array}
\\
6 &
(256, 120, 56, 56) &
16 &
\begin{array}{l}
2048t^{9} + 167424t^{8} + 1091584t^{7}
\,+
\\
 3440640t^{6} + 4128768t^{5} + 1835008t^{4}
\,+
\\
 286720t^{3} + 15360t^{2} + 256t + 1
\end{array}
\\
7 &
(256, 136, 72, 72) &
16 &
\begin{array}{l}
7168t^{10} + 143360t^{9} + 3804672t^{8}
\,+
\\
 15886336t^{7} + 24313856t^{6} + 15138816t^{5}
\,+
\\
 3932160t^{4} + 417792t^{3} + 17408t^{2} + 256t + 1
\end{array}
\\
8 &
(256, 120, 56, 56) &
16 &
\begin{array}{l}
9216t^{9} + 181760t^{8} + 1091584t^{7}
\,+
\\
 3440640t^{6} + 4128768t^{5} + 1835008t^{4}
\,+
\\
 286720t^{3} + 15360t^{2} + 256t + 1
\end{array}
\\
9 &
(256, 136, 72, 72) &
16 &
\begin{array}{l}
107520t^{9} + 3790336t^{8} + 15886336t^{7}
\,+
\\
 24313856t^{6} + 15138816t^{5} + 3932160t^{4}
\,+
\\
 417792t^{3} + 17408t^{2} + 256t + 1
\end{array}
\\
\hline
\end{array}
\end{align*}
\caption{$[f_{8,10}]$ extended Cayley classes.}
\label{tab-c8_10_EC_classes}
\end{table}

\begin{table}
\small{}
\begin{align*}
\def\arraystretch{1.2}
\begin{array}{|ccc|}
\hline
[f_{8,10}] &
[f_{8,10}] &
\text{Frequency}
\\
&
\text{duals} &
\\
\hline
  0 &    1 & 2048
\\
  1 &    0 & 2048
\\
  2 &    3 & 7168
\\
  3 &    2 & 7168
\\
  4 &    5 & 7168
\\
  5 &    4 & 7168
\\
  6 &    6 & 8192
\\
  7 &    7 & 8192
\\
  8 &    8 & 8192
\\
  9 &    9 & 8192
\\
\hline
\end{array}
\end{align*}
\caption{Correspondence between $[f_{8,10}]$ extended Cayley classes and $[f_{8,10}]$ dual extended Cayley classes.}
\label{tab-c8_10-dual-EC_classes}
\end{table}

\begin{figure}[!bhpt] 
\centering
\begin{minipage}{.48\textwidth}
  \centering
  \includegraphics[width=.9\linewidth]{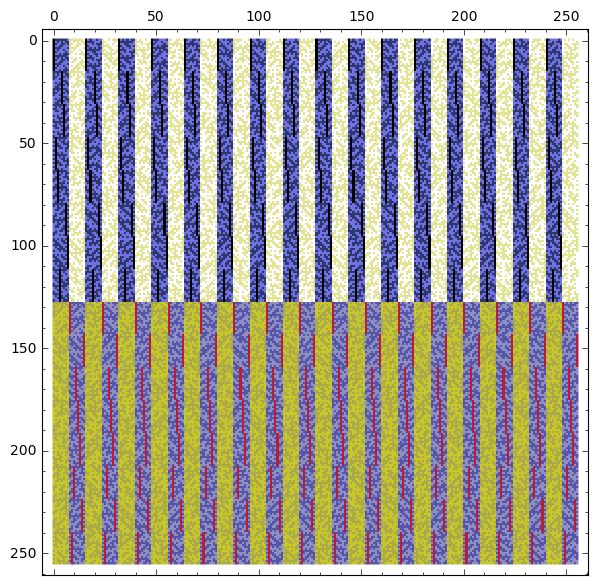}
  \captionof{figure}{$[f_{8,10}]$: extended Cayley classes.}
  \label{fig:c8_10_bent_cayley_graph_index_matrix}
\end{minipage}
~~
\begin{minipage}{.48\textwidth}
  \centering
  \includegraphics[width=.9\linewidth]{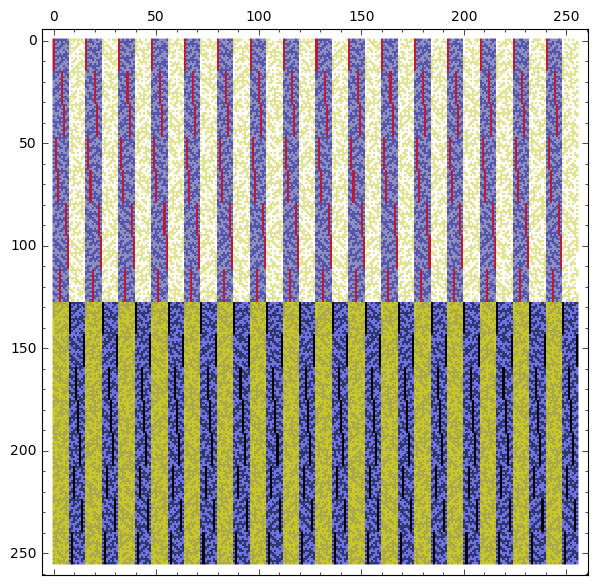}
  \captionof{figure}{$[f_{8,10}]$: extended Cayley classes of dual bent functions.}
  \label{fig:c8_10_dual_cayley_graph_index_matrix}
\end{minipage}%
\end{figure}
\newpage

\subsection{Two sequences of bent functions}

As stated in the introduction, in a recent paper \cite{Leo17Hurwitz},
the author found an example of two infinite series of bent functions whose
Cayley graphs have the same strongly regular parameters at each dimension,
but are not isomorphic if the dimension is 8 or more.
The sequences are $\sigma_m$ and $\tau_m$ for $m \geqslant 1$,
whose definitions are reproduced here from \cite{Leo15Twin}.

\begin{definition}
\label{def-sigma-tau}
The sign-of-square function $\sigma_m : \Z_2^{2 m} \To \Z_2$ is defined as follows:

For $i \in \Z_{2^{2m}},$ $\sigma_m(i) = 1$ if and only if the number of
1 digits in the base 4 representation of $i$ is odd.

The non-diagonal-symmetry function $\tau_m \colon \Z_{2^{2 m}} \To \Z_2$
is defined as follows.

For $i$ in $\Z_2^2$:
\begin{align*}
&\tau_1(i) :=
\begin{cases}
1 &\text{if~}i = 10,
\\
0 &\text{otherwise}.
\end{cases}
\end{align*}

For $i$ in $\Z_2^{2 m - 2}$:
\begin{align*}
&\tau_m (00 \odot i) := \tau_{m-1}(i),
\\
&\tau_m (01 \odot i) := \sigma_{m-1}(i),
\\
&\tau_m (10 \odot i) := \sigma_{m-1}(i) + 1,
\\
&\tau_m (11 \odot i) := \tau_{m-1}(i).
\end{align*}
where $\odot$ denotes concatenation of bit vectors, and $\sigma$ is the sign-of-square function, as above.
\end{definition}

As shown in  \cite{Leo15Twin}, both sequences produce Cayley graphs whose strongly regular parameters are
\begin{align*}
(v_m,k_m,\lambda_m=\mu_m) &= (4^m, 2^{2 m - 1} - 2^{m-1}, 2^{2 m - 2} - 2^{m-1}).
\end{align*}


\begin{figure}[!ht]
\centering
\begin{minipage}{.48\textwidth}
  \centering
  \includegraphics[width=.9\linewidth]{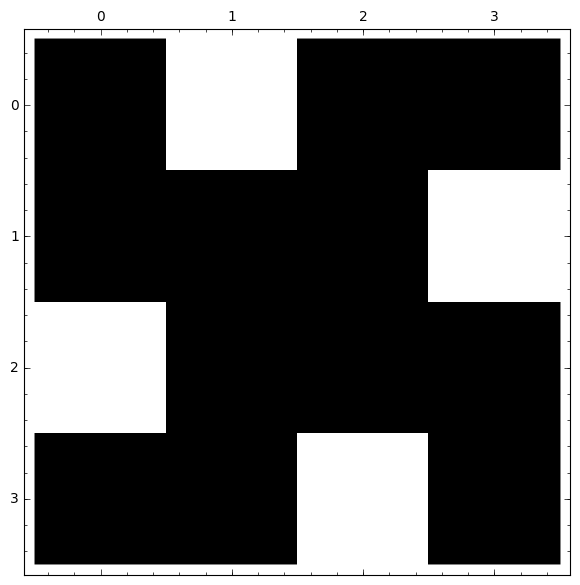}
  \captionof{figure}{$[\sigma_1]$:\\2 extended Cayley classes}
  \label{fig:sigma_1_bent_cayley_graph_index_matrix}
\end{minipage}%
\begin{minipage}{.48\textwidth}
  \centering
  \includegraphics[width=.9\linewidth]{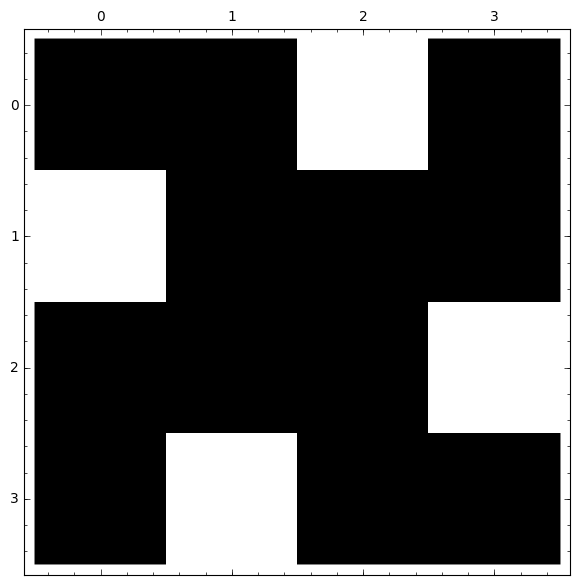}
  \captionof{figure}{$[\tau_1]$:\\2 extended Cayley classes}
  \label{fig:tau_1_bent_cayley_graph_index_matrix}
\end{minipage}
\end{figure}

\begin{figure}[!hb]
\centering
\begin{minipage}{.48\textwidth}
  \centering
  \includegraphics[width=.9\linewidth]{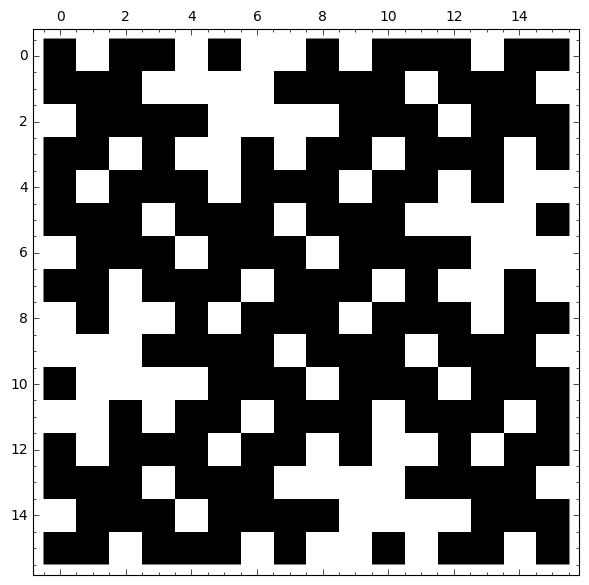}
  \captionof{figure}{$[\sigma_2]$:\\2 extended Cayley classes}
  \label{fig:sigma_2_bent_cayley_graph_index_matrix}
\end{minipage}%
\begin{minipage}{.48\textwidth}
  \centering
  \includegraphics[width=.9\linewidth]{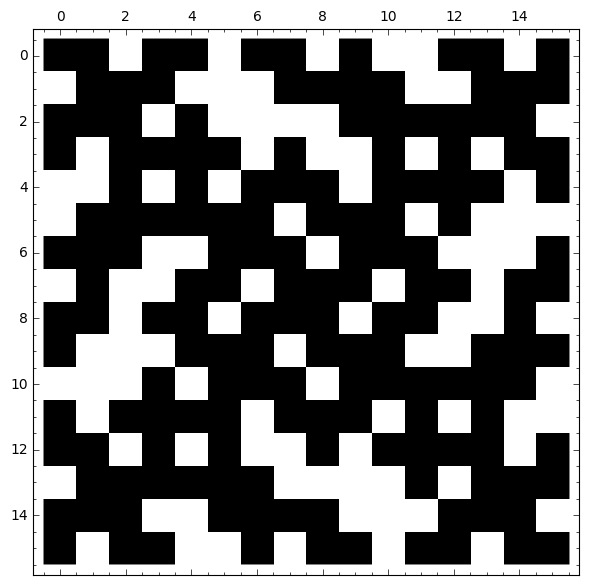}
  \captionof{figure}{$[\tau_2]$:\\2 extended Cayley classes}
  \label{fig:tau_2_bent_cayley_graph_index_matrix}
\end{minipage}
\end{figure}

Figures \ref{fig:sigma_1_bent_cayley_graph_index_matrix} to \ref{fig:tau_4_bent_cayley_graph_index_matrix}
illustrate the extended Cayley classes within the ET classes of each of function $\sigma_m$ and $\tau_m$ for $m$ from 1 to 4.
Note that $\tau_3$ has degree 3, and $\tau_4$ has degree 4.
As shown in \cite{Leo17Hurwitz},
The functions $\sigma_m$ and $\tau_m$ are extended Cayley equivalent for $m$ from 1 to 3, but are inequivalent for $m>3$.

\begin{figure}[!ht]
\centering
\begin{minipage}{.48\textwidth}
  \centering
  \includegraphics[width=.9\linewidth]{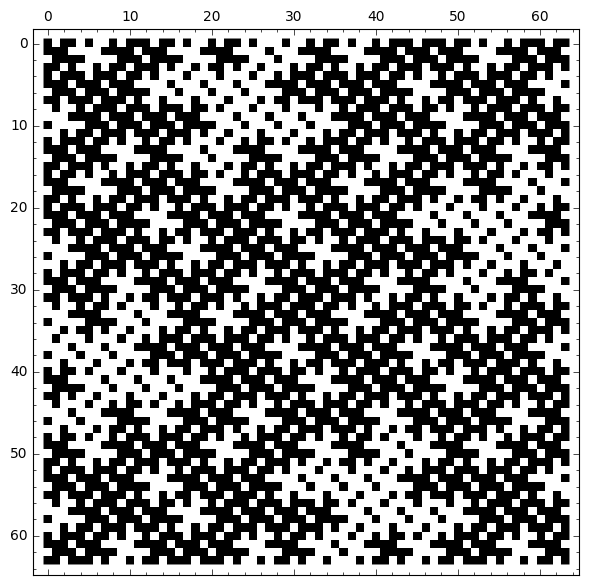}
  \captionof{figure}{$[\sigma_3]$:\\2 extended Cayley classes}
  \label{fig:sigma_3_bent_cayley_graph_index_matrix}
\end{minipage}%
\begin{minipage}{.48\textwidth}
  \centering
  \includegraphics[width=.9\linewidth]{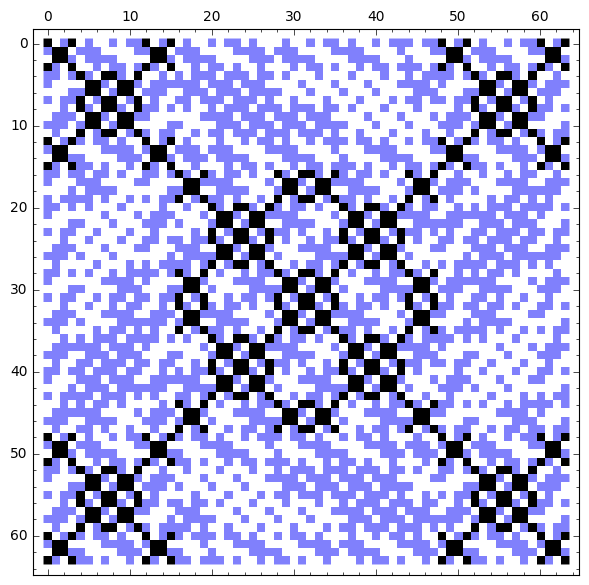}
  \captionof{figure}{$[\tau_3]$:\\3 extended Cayley classes}
  \label{fig:tau_3_bent_cayley_graph_index_matrix}
\end{minipage}
\end{figure}

\begin{figure}[!hb]
\centering
\begin{minipage}{.48\textwidth}
  \centering
  \includegraphics[width=.9\linewidth]{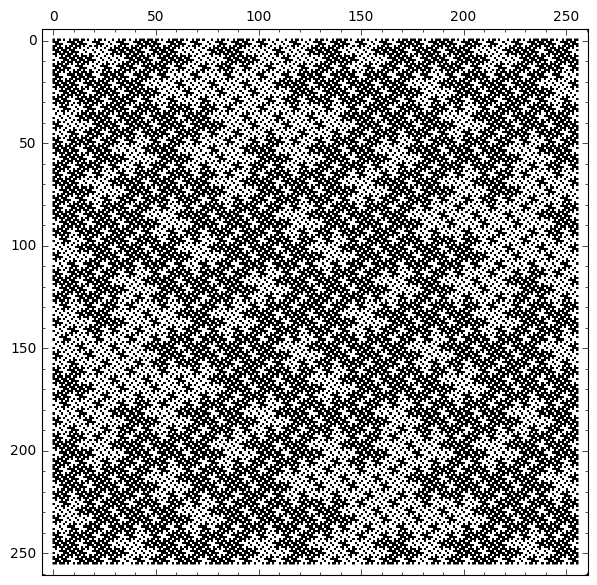}
  \captionof{figure}{$[\sigma_4]$:\\2 extended Cayley classes}
  \label{fig:sigma_4_bent_cayley_graph_index_matrix}
\end{minipage}%
\begin{minipage}{.48\textwidth}
  \centering
  \includegraphics[width=.9\linewidth]{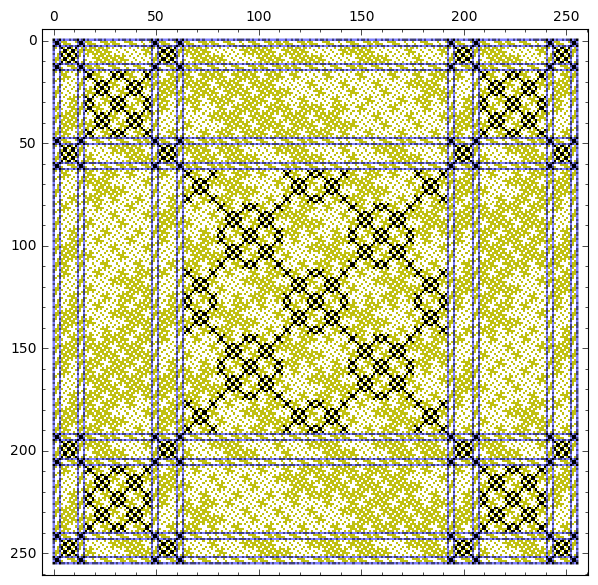}
  \captionof{figure}{$[\tau_4]$:\\5 extended Cayley classes}
  \label{fig:tau_4_bent_cayley_graph_index_matrix}
\end{minipage}
\end{figure}
~
\newpage
~
\newpage
~
\newpage
\subsection{CAST-128 S-boxes}
The CAST-128 encryption algorithm is used in PGP and elsewhere \cite{Ada97}.
The algorithm uses 8 S-boxes, each of which consists of 32 Boolean bent functions in 8 dimensions,
with degree 4, making 256 bent functions in total.
The full CAST-128 algorithm, including the contents of the S-boxes,
is published as IETF Request For Comments 2144 \cite{RFC2144}.

The bent function $cast128_{1,0}$ is the first bent function of S-box number 1 of CAST-128.
Its definition by algebraic normal form is
\footnotesize{}
\begin{align*}
cast128_{1,0} :=
\begin{array}{l}
x_{0} x_{1} x_{2} x_{3} + x_{0} x_{1} x_{2} x_{4} + x_{0} x_{1} x_{2} x_{5} + x_{0} x_{1} x_{2} + x_{0} x_{1} x_{3} x_{5} + x_{0} x_{1} x_{3} x_{6}\, +
\\
x_{0} x_{1} x_{3} + x_{0} x_{1} x_{5} x_{6} + x_{0} x_{1} x_{6} + x_{0} x_{1} x_{7} + x_{0} x_{2} x_{3} x_{4} + x_{0} x_{2} x_{3}\, +
\\
x_{0} x_{2} x_{4} x_{5} + x_{0} x_{2} x_{5} x_{7} + x_{0} x_{2} x_{6} + x_{0} x_{2} x_{7} + x_{0} x_{2} + x_{0} x_{3} x_{4} x_{5}\, +
\\
x_{0} x_{3} x_{4} x_{6} + x_{0} x_{3} x_{5} x_{6} + x_{0} x_{3} + x_{0} x_{4} x_{5} x_{6} + x_{0} x_{4} x_{5} x_{7} + x_{0} x_{4} x_{5}\, +
\\
x_{0} x_{4} x_{6} + x_{0} x_{4} + x_{0} x_{5} x_{6} + x_{0} x_{5} x_{7} + x_{0} x_{6} + x_{0} x_{7} + x_{0} + x_{1} x_{2} x_{4} x_{6}\, +
\\
x_{1} x_{2} x_{4} x_{7} + x_{1} x_{2} x_{5} x_{6} + x_{1} x_{2} x_{7} + x_{1} x_{3} x_{4} x_{6} + x_{1} x_{3} x_{4} x_{7} + x_{1} x_{3} x_{5}\, +
\\
x_{1} x_{3} x_{7} + x_{1} x_{4} x_{5} x_{7} + x_{1} x_{4} x_{6} + x_{1} x_{5} x_{6} + x_{1} + x_{2} x_{3} x_{4} x_{6} + x_{2} x_{3} x_{4}\, +
\\
x_{2} x_{3} x_{5} x_{6} + x_{2} x_{3} x_{5} + x_{2} x_{3} x_{7} + x_{2} x_{4} + x_{2} x_{5} x_{6} + x_{2} x_{5} + x_{2} + x_{3} x_{4} x_{5} x_{6}\, +
\\
x_{3} x_{5} x_{6} + x_{3} x_{5} x_{7} + x_{3} x_{6} + x_{3} + x_{4} + x_{6} x_{7}.
\end{array}
\end{align*}
\normalsize{}

This bent function $cast128_{1,0}$ is prolific:
the ET class $[cast128_{1,0}]$ contains the maximum possible number of Cayley classes,
that is $65\,536$.
The duals of the bent functions $[cast128_{1,0}]$ give another $65\,536$ extended Cayley classes.
In other words, no bent function in $[cast128_{1,0}]$ is EA equivalent to its dual.
The two weight classes of $[cast128_{1,0}]$ are
shown in Figure~\ref{fig:cast128_1_0_weight_class_matrix}.

\begin{figure}[!hb]
\centering
\begin{minipage}{.48\textwidth}
  \centering
\includegraphics[width=.9\linewidth]{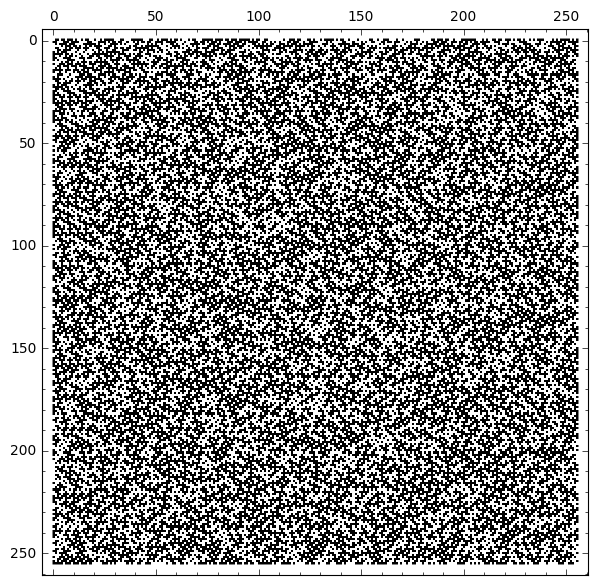}
  \captionof{figure}{$[cast128_{1,0}]$:\\Weight classes ~~~~~~ ~~~~~~~~\\~\\Colormap: gist\_stern.}
  \label{fig:cast128_1_0_weight_class_matrix}
\end{minipage}
\begin{minipage}{.48\textwidth}
  \centering
\includegraphics[width=.9\linewidth]{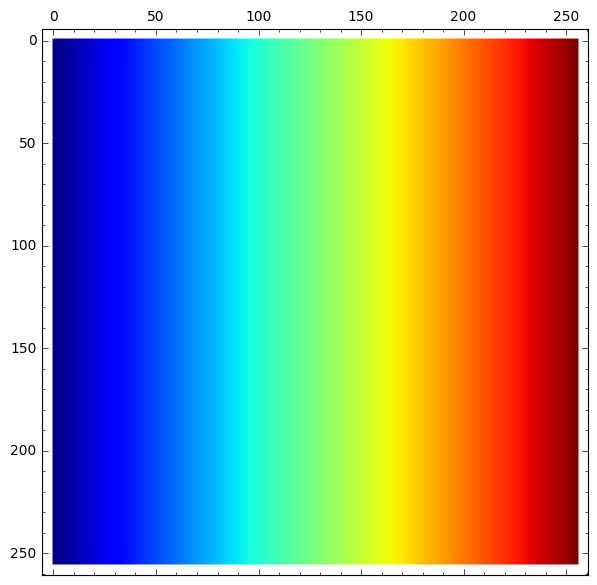}
  \captionof{figure}{$[cast128_{1,0}]$:\\$65\,536$ extended Cayley classes.\\Total including duals is $131\,072$.
\\Colormap: jet.}
  \label{fig:cast128_1_0_bent_cayley_graph_index_matrix}
\end{minipage}%
\end{figure}

\begin{figure}[!ht]
\centering
\begin{minipage}{.48\textwidth}
  \centering
\includegraphics[width=.9\linewidth]{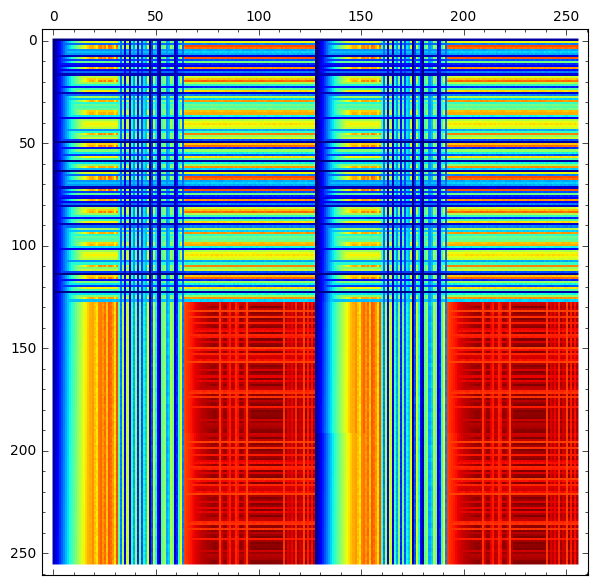}
  \captionof{figure}{$[cast128_{2,1}]$:\\$8\,256$ extended Cayley classes.\\Total including duals is $8\,256$.
\\Colormap: jet.}
  \label{fig:cast128_2_1_bent_cayley_graph_index_matrix}
\end{minipage}%
\begin{minipage}{.48\textwidth}
  \centering
\includegraphics[width=.9\linewidth]{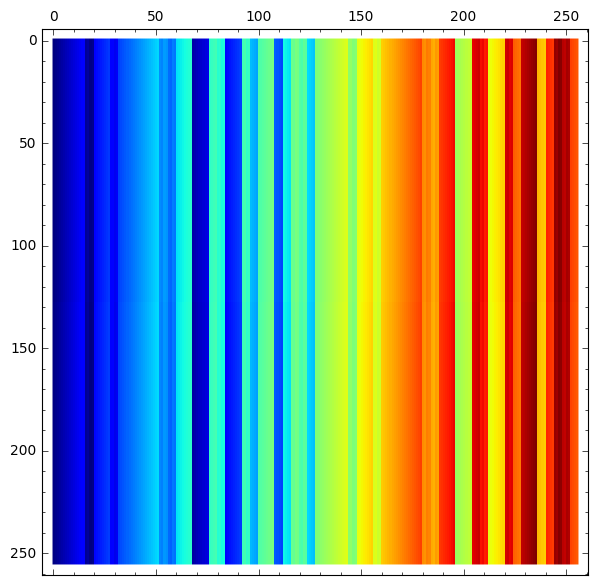}
  \captionof{figure}{$[cast128_{2,16}]$:\\$32\,768$ extended Cayley classes.\\Total including duals is $65\,536$.
\\Colormap: jet.}
  \label{fig:cast128_2_16_bent_cayley_graph_index_matrix}
\end{minipage}%
\end{figure}
\begin{figure}[!hb]
\centering
\begin{minipage}{.48\textwidth}
  \centering
\includegraphics[width=.9\linewidth]{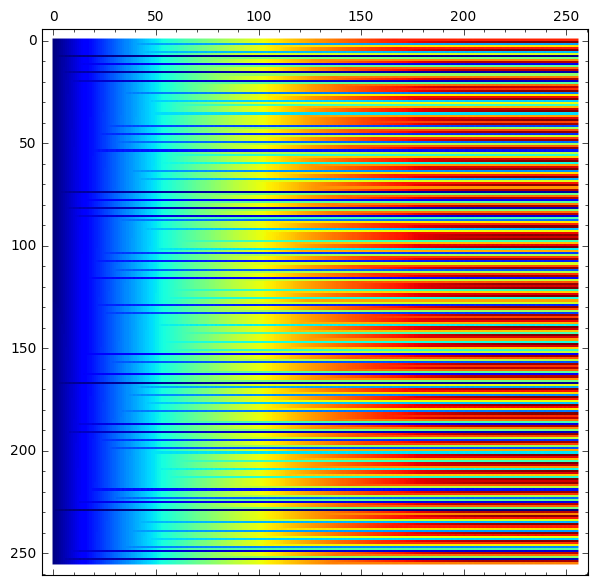}
  \captionof{figure}{$[cast128_{4,27}]$:\\$65\,536$ extended Cayley classes.\\Total including duals is $65\,536$.
\\Colormap: jet.}
  \label{fig:cast128_4_27_bent_cayley_graph_index_matrix}
\end{minipage}%
\begin{minipage}{.48\textwidth}
  \centering
\includegraphics[width=.9\linewidth]{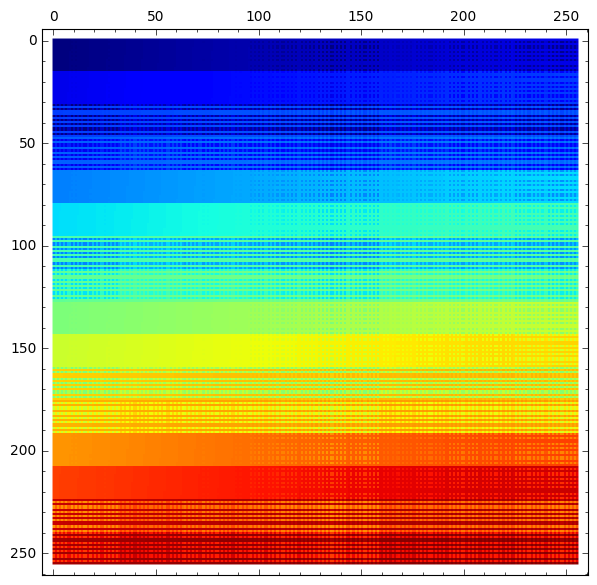}
  \captionof{figure}{$[cast128_{5,16}]$:\\$33\,280$ extended Cayley classes.\\Total including duals is $66\,560$.
\\Colormap: jet.}
  \label{fig:cast128_5_16_bent_cayley_graph_index_matrix}
\end{minipage}%
\end{figure}
\begin{figure}[!ht]
\centering
\begin{minipage}{.48\textwidth}
  \centering
\includegraphics[width=.9\linewidth]{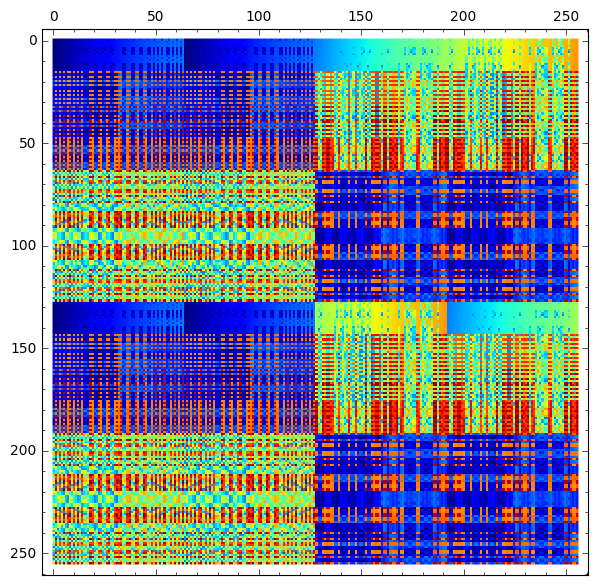}
  \captionof{figure}{$[cast128_{5,27}]$:\\$6\,144$ extended Cayley classes.\\Total including duals is $6\,144$.
\\Colormap: jet.}
  \label{fig:cast128_5_27_bent_cayley_graph_index_matrix}
\end{minipage}%
\begin{minipage}{.48\textwidth}
  \centering
\includegraphics[width=.9\linewidth]{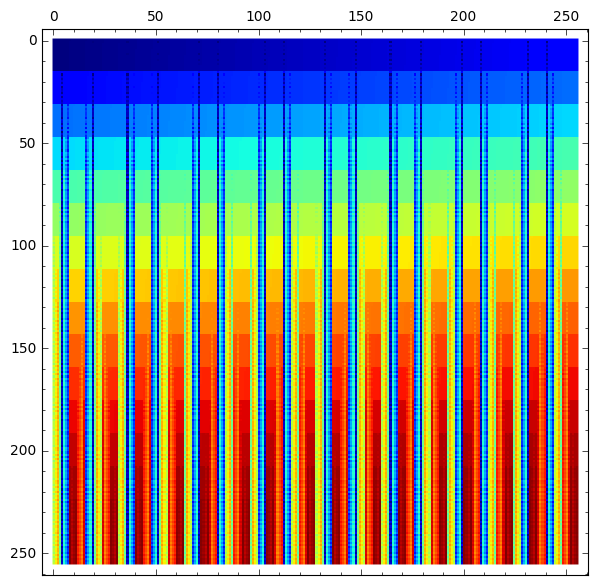}
  \captionof{figure}{$[cast128_{6,17}]$:\\$65\,536$ extended Cayley classes.\\Total including duals is $65\,536$.
\\Colormap: jet.}
  \label{fig:cast128_6_17_bent_cayley_graph_index_matrix}
\end{minipage}%
\end{figure}
\begin{figure}[!ht]
\centering
\begin{minipage}{.48\textwidth}
  \centering
\includegraphics[width=.9\linewidth]{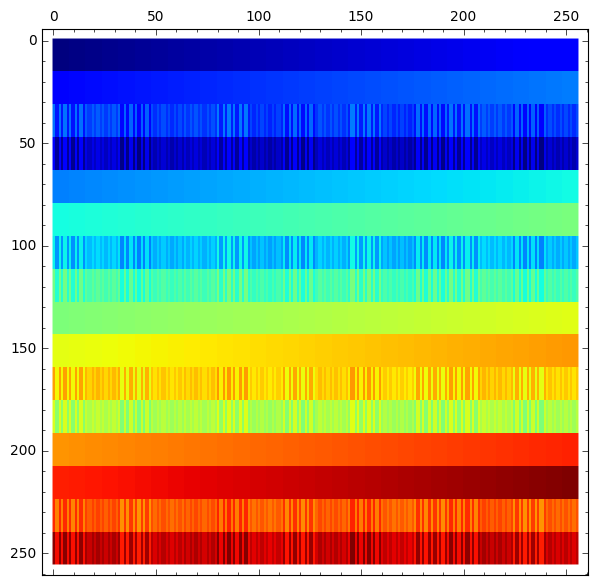}
  \captionof{figure}{$[cast128_{7,15}]$:\\$32\,768$ extended Cayley classes.\\Total including duals is $65\,536$.
\\Colormap: jet.}
  \label{fig:cast128_7_15_bent_cayley_graph_index_matrix}
\end{minipage}%
\begin{minipage}{.48\textwidth}
  \centering
\includegraphics[width=.9\linewidth]{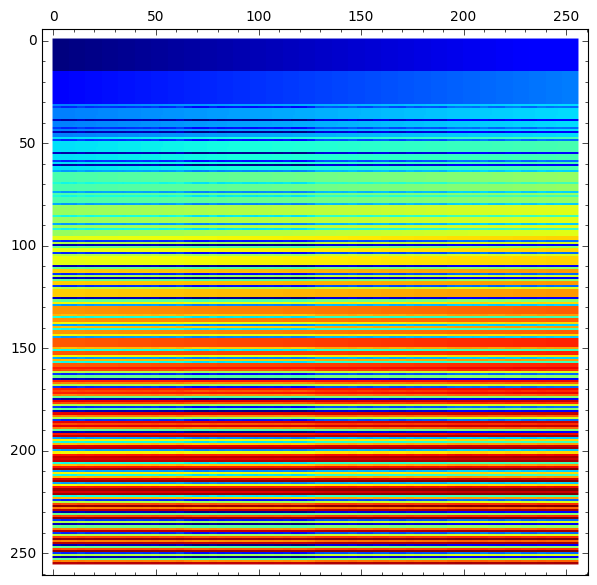}
  \captionof{figure}{$[cast128_{7,21}]$:\\$32\,768$ extended Cayley classes.\\Total including duals is $65\,536$.
\\Colormap: jet.}
  \label{fig:cast128_7_21_bent_cayley_graph_index_matrix}
\end{minipage}%
\end{figure}

Of the 256 bent functions that make up the 8 S-boxes of CAST-128, 248 are like $cast128_{1,0}$,
they are prolific and are not EA equivalent to their duals.
The remaining 8 bent functions are exceptional.
Both $cast128_{4,27}$ and $cast128_{6,17}$ are prolific, but both are EA equivalent to their duals.
\newpage

The other 6 bent functions are not prolific and the number of Cayley classes for each is given in the captions to
Figures \ref{fig:cast128_2_1_bent_cayley_graph_index_matrix} to \ref{fig:cast128_7_21_bent_cayley_graph_index_matrix}.
\subsection{Partial spread bent functions.}

According to Langevin and Hou \cite{LanH11counting}
there are $70\,576\,747\,237\,594\,114\,392\,064 \approx 2^{75.9}$ \Emph{partial spread} bent functions in
dimension 8,
contained in $14\,758$ EA classes.
The EA class representatives are listed at Langevin's web page \cite{Lan10psf}.
On this web page,
the file \texttt{psf-8.txt} contains details of $9316$ representatives that are
$\mathcal{PS}^{(-)}$ bent functions, all of degree 4; and
the file \texttt{psf-9.txt} contains details of $5442$ representatives that are
$\mathcal{PS}^{(+)}$ bent functions, of which $5440$ are of degree 4.

Preliminary calculations using SageMath on the NCI Raijin supercomputer indicate that
\begin{enumerate}
 \item
The $5442$ EA classes of $\mathcal{PS}^{(+)}$ bent functions of dimension 8
contain $296\,594\,720$ extended Cayley classes, assuming that each extended Cayley class appears in only one of the EA classes.
 \item
If the duals of the $5442$ representatives, and their corresponding EA classes are included,
the total number of extended Cayley classes is $541\,700\,450$, under the same assumption.
 \item
Of the $5442$ representatives,
$3434$ are prolific and not EA equivalent to their dual,
$582$ are prolific and are EA equivalent to their dual,
and the EA classes of the remaining $1426$ each contain less than $65\,536$ extended Cayley classes.
\end{enumerate}
The classifications of these $5442$ EA classes take up $2.1\,$TB of space on the NCI Raijin supercomputer.
The classifications of the $9316$ $\mathcal{PS}^{(-)}$ bent functions are currently being generated on Raijin.
In total, the classifications of the $\mathcal{PS}^{(+)}$ and $\mathcal{PS}^{(-)}$
bent functions are estimated to take up $6\,$TB of space on Raijin. 
It is intended that these classifications will be incorporated into a public database, if support can be found to maintain it
\cite{Leo18Database}. 



One example classification from $\mathcal{PS}^{(+)}$ in dimension 8 is that of $psf_{9,5439}$.
The bent function $psf_{9,5439}$ is listed as function number $5439$ in \texttt{psf-9.txt},
and is a $\mathcal{PS}^{(+)}$ bent function of degree 4.
Its ET class $[psf_{9,5439}]$ contains 16 extended Cayley classes,
of which 6 form three duality pairs similar to those seen in $[f_{8,9}]$ and $[f_{8,10}]$ above.

The 16 extended Cayley classes are distributed
as shown in Figures~\ref{fig:psf_9_5439_bent_cayley_graph_index_matrix} (classes) and~\ref{fig:psf_9_5439_dual_cayley_graph_index_matrix}
(classes of duals).

\begin{figure}[!ht] 
\centering
\begin{minipage}{.48\textwidth}
  \centering
  \includegraphics[width=.9\linewidth]{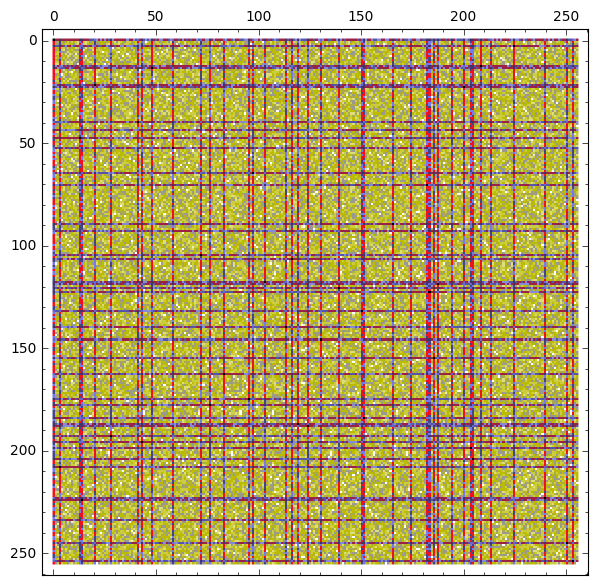}
  \captionof{figure}{$[psf_{9,5439}]$:\\extended Cayley classes ~~ ~~~~ ~~~~\\~~~~~~~~~}
  \label{fig:psf_9_5439_bent_cayley_graph_index_matrix}
\end{minipage}
~~
\begin{minipage}{.48\textwidth}
  \centering
  \includegraphics[width=.9\linewidth]{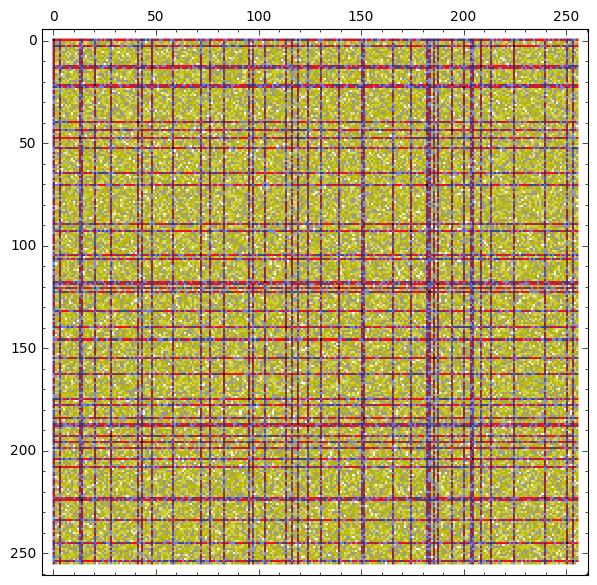}
  \captionof{figure}{$[psf_{9,5439}]$:\\extended Cayley classes of dual bent\\functions}
  \label{fig:psf_9_5439_dual_cayley_graph_index_matrix}
\end{minipage}%
\end{figure}




\newpage



\input{Leopardi-Bent-functions-Cayley-graphs.bbl}

%% file: Leopardi-Bent-functions-Cayley-graphs.bbl
\begin{thebibliography}{10}

\bibitem{RFC2144}
C.~M. Adams.
\newblock The {CAST-128} encryption algorithm.
\newblock RFC 2144, RFC Editor, May 1997.

\bibitem{Ada97}
C.~M. Adams.
\newblock Constructing symmetric ciphers using the {CAST} design procedure.
\newblock In E.~Kranakis and P.~Van~Oorschot, editors, {\em Selected Areas in
  Cryptography},  71--104, Boston, MA, (1997). Springer US.

\bibitem{BerC99}
A.~Bernasconi and B.~Codenotti.
\newblock Spectral analysis of {Boolean} functions as a graph eigenvalue
  problem.
\newblock {\em IEEE Transactions on Computers}, 48(3):345--351, (1999).

\bibitem{BerCV01}
A.~Bernasconi, B.~Codenotti, and J.~M. VanderKam.
\newblock A characterization of bent functions in terms of strongly regular
  graphs.
\newblock {\em IEEE Transactions on Computers}, 50(9):984--985, (2001).

\bibitem{BorD2008}
J.~Borwein and K.~Devlin.
\newblock {\em The computer as crucible: An introduction to experimental
  mathematics}.
\newblock AK Peters/CRC Press, (2008).

\bibitem{Bos63}
R.~C. Bose.
\newblock Strongly regular graphs, partial geometries and partially balanced
  designs.
\newblock {\em Pacific J. Math}, 13(2):389--419, (1963).

\bibitem{BouFFWW2006}
I.~Bouyukliev, V.~Fack, W.~Willems, and J.~Winne.
\newblock Projective two-weight codes with small parameters and their
  corresponding graphs.
\newblock {\em Designs, Codes and Cryptography}, 41(1):59--78, (2006).

\bibitem{Bra06thesis}
A.~Braeken.
\newblock {\em Cryptographic Properties of Boolean Functions and S-Boxes}.
\newblock {PhD} thesis, Katholieke Universiteit Leuven, Leuven-Heverlee,
  Belgium, (2006).

\bibitem{BroCN89}
A.~Brouwer, A.~Cohen, and A.~Neumaier.
\newblock {\em Distance-Regular Graphs}.
\newblock Ergebnisse der Mathematik und Ihrer Grenzgebiete, 3 Folge / A Series
  of Modern Surveys in Mathematics. Springer London, (2011).

\bibitem{Brov92}
A.~E. Brouwer and C.~A. Van~Eijl.
\newblock On the p-rank of the adjacency matrices of strongly regular graphs.
\newblock {\em Journal of Algebraic Combinatorics}, 1(4):329--346, (1992).

\bibitem{CalK1986}
R.~Calderbank and W.~M. Kantor.
\newblock The geometry of two-weight codes.
\newblock {\em Bulletin of the London Mathematical Society}, 18(2):97--122,
  (1986).

\bibitem{Cam2003}
P.~J. Cameron.
\newblock Random strongly regular graphs?
\newblock {\em Discrete Mathematics}, 273(1):103--114, (2003).
\newblock EuroComb'01.

\bibitem{CamVL91}
P.~J. Cameron and J.~H. Van~Lint.
\newblock {\em Designs, graphs, codes and their links}, volume~3.
\newblock Cambridge University Press, (1991).

\bibitem{Car10boolean}
C.~Carlet.
\newblock Boolean functions for cryptography and error correcting codes.
\newblock In {\em Boolean Models and Methods in Mathematics, Computer Science,
  and Engineering}, volume~2,  257--397. Cambridge University Press, (2010).

\bibitem{CarDPS10self}
C.~Carlet, L.~E. Danielsen, M.~G. Parker, and P.~Sol{\'e}.
\newblock Self-dual bent functions.
\newblock {\em International Journal of Information and Coding Theory},
  1(4):384--399, (2010).

\bibitem{CarM2016four}
C.~Carlet and S.~Mesnager.
\newblock Four decades of research on bent functions.
\newblock {\em Designs, Codes and Cryptography}, 78(1):5--50, (2016).

\bibitem{CheTZ11}
Y.~M. Chee, Y.~Tan, and X.~D. Zhang.
\newblock Strongly regular graphs constructed from p-ary bent functions.
\newblock {\em Journal of Algebraic Combinatorics}, 34(2):251--266, (2011).

\bibitem{Com80}
J.~D. Comerford.
\newblock Affine and general linear equivalences of {Boolean} functions.
\newblock {\em Information and Control}, 45(2):156--169, (1980).

\bibitem{CusS2017}
T.~W. Cusick and P.~Stanica.
\newblock {\em Cryptographic Boolean functions and applications}.
\newblock Academic Press, 2nd edition, (2017).

\bibitem{Del72weights}
P.~Delsarte.
\newblock Weights of linear codes and strongly regular normed spaces.
\newblock {\em Discrete Mathematics}, 3(1-3):47--64, (1972).

\bibitem{Dil74}
J.~F. Dillon.
\newblock {\em Elementary {Hadamard} Difference Sets}.
\newblock PhD thesis, University of Maryland College Park, Ann Arbor, USA,
  (1974).

\bibitem{DilS87block}
J.~F. Dillon and J.~R. Schatz.
\newblock Block designs with the symmetric difference property.
\newblock In {\em Proceedings of the NSA Mathematical Sciences Meetings},
  159--164. US Govt. Printing Office Washington, DC, (1987).

\bibitem{Din2015}
C.~Ding.
\newblock Linear codes from some 2-designs.
\newblock {\em IEEE Transactions on information theory}, 61(6):3265--3275,
  (2015).

\bibitem{DinMTX2018cyclic}
C.~Ding, S.~Mesnager, C.~Tang, and M.~Xiong.
\newblock Cyclic bent functions and their applications in codes, codebooks,
  designs, {MUBs} and sequences.
\newblock {\em arXiv preprint arXiv:1811.07725}, (2018).

\bibitem{FeuSSW2013}
T.~Feulner, L.~Sok, P.~Sol{\'e}, and A.~Wassermann.
\newblock Towards the classification of self-dual bent functions in eight
  variables.
\newblock {\em Designs, Codes and Cryptography}, 68(1):395--406, (2013).

\bibitem{Har64}
M.~A. Harrison.
\newblock On the classification of {Boolean} functions by the general linear
  and affine groups.
\newblock {\em Journal of the Society for Industrial and Applied Mathematics},
  12(2):285--299, (1964).

\bibitem{HoeL94}
C.~Hoede and X.~Li.
\newblock Clique polynomials and independent set polynomials of graphs.
\newblock {\em Discrete Mathematics}, 125(1):219 -- 228, (1994).

\bibitem{HuaY04}
T.~Huang and K.-H. You.
\newblock Strongly regular graphs associated with bent functions.
\newblock In {\em 7th International Symposium on Parallel Architectures,
  Algorithms and Networks, 2004. Proceedings.},  380--383, May 2004.

\bibitem{Jac64}
N.~Jacobson.
\newblock {\em Lectures in Abstract Algebra: III. Theory of Fields and Galois
  Theory (Graduate Texts in Mathematics)}.
\newblock Van Nostrand, (1964).

\bibitem{JoyEtAl13Sage}
D.~Joyner, O.~Geil, C.~Thomsen, C.~Munuera, I.~M{\'a}rquez-Corbella,
  E.~Mart{\'\i}nez-Moro, M.~Bras-Amor{\'o}s, R.~Jurrius, and R.~Pellikaan.
\newblock Sage: A basic overview for coding theory and cryptography.
\newblock In {\em Algebraic Geometry Modeling in Information Theory}, volume~8
  of {\em Series on Coding Theory and Cryptology},  1--45. World Scientific
  Publishing Company, (2013).

\bibitem{JunK07Bliss}
T.~Junttila and P.~Kaski.
\newblock Engineering an efficient canonical labeling tool for large and sparse
  graphs.
\newblock In D.~Applegate, G.~S. Brodal, D.~Panario, and R.~Sedgewick, editors,
  {\em Proceedings of the Ninth Workshop on Algorithm Engineering and
  Experiments and the Fourth Workshop on Analytic Algorithms and
  Combinatorics},  135--149, New Orleans, LA, (2007). Society for Industrial
  and Applied Mathematics.

\bibitem{JunK11conflict}
T.~Junttila and P.~Kaski.
\newblock Conflict propagation and component recursion for canonical labeling.
\newblock In {\em Theory and Practice of Algorithms in (Computer) Systems},
  151--162. Springer, (2011).

\bibitem{Kan75symplectic}
W.~M. Kantor.
\newblock Symplectic groups, symmetric designs, and line ovals.
\newblock {\em Journal of Algebra}, 33(1):43--58, (1975).

\bibitem{Kan83exponential}
W.~M. Kantor.
\newblock Exponential numbers of two-weight codes, difference sets and
  symmetric designs.
\newblock {\em Discrete Mathematics}, 46(1):95--98, (1983).

\bibitem{Lan10psf}
P.~Langevin.
\newblock Classification of partial spread functions in eight variables,
  (2010).
\newblock \url{http://langevin.univ-tln.fr/project/spread/psp.html}.

\bibitem{LanH11counting}
P.~Langevin and X.-D. Hou.
\newblock Counting partial spread functions in eight variables.
\newblock {\em IEEE Transactions on Information Theory}, 57(4):2263--2269,
  (2011).

\bibitem{LanL11counting}
P.~Langevin and G.~Leander.
\newblock Counting all bent functions in dimension eight
  99270589265934370305785861242880.
\newblock {\em Designs, Codes and Cryptography}, 59(1-3):193--205, (2011).

\bibitem{LanLM08Kasami}
P.~Langevin, G.~Leander, and G.~McGuire.
\newblock Kasami bent functions are not equivalent to their duals.
\newblock In G.~Mullen, D.~Panario, and I.~Shparlinski, editors, {\em Finite
  Fields and Applications: Eighth International Conference on Finite Fields and
  Applications, July 9-13, 2007, Melbourne, Australia},  187--198. American
  Mathematical Society, (2008).

\bibitem{Lem1975matrix}
A.~Lempel.
\newblock Matrix factorization over gf(2) and trace-orthogonal bases of
  gf(2\^{}n).
\newblock {\em SIAM Journal on Computing}, 4(2):175--186, (1975).

\bibitem{Leo18Database}
P.~Leopardi.
\newblock A database of {Cayley} graphs of bent {Boolean} functions.
\newblock In preparation.

\bibitem{Leo15Twin}
P.~Leopardi.
\newblock Twin bent functions and {Clifford} algebras.
\newblock In {\em Algebraic Design Theory and Hadamard Matrices},  189--199.
  Springer, (2015).

\bibitem{Leo16GitHub}
P.~Leopardi.
\newblock Boolean-cayley-graphs, (2016).
\newblock \url{https://github.com/penguian/Boolean-Cayley-graphs} GitHub
  repository.

\bibitem{Leo17CoCalc}
P.~Leopardi.
\newblock Boolean-cayley-graphs, (2017).
\newblock \url{http://tinyurl.com/Boolean-Cayley-graphs} CoCalc public folder.

\bibitem{Leo17Hurwitz}
P.~Leopardi.
\newblock Twin bent functions, strongly regular {Cayley} graphs, and
  {Hurwitz-Radon} theory.
\newblock {\em Journal of Algebra Combinatorics Discrete Structures and
  Applications}, 4(3):271--280, (2017).

\bibitem{MacS77}
F.~J. MacWilliams and N.~J.~A. Sloane.
\newblock {\em The theory of error-correcting codes}.
\newblock Elsevier, (1977).

\bibitem{Mai91}
J.~A. Maiorana.
\newblock A classification of the cosets of the {Reed-Muller} code {R}(1, 6).
\newblock {\em Mathematics of Computation}, 57(195):403--414, (1991).

\bibitem{McKP13nauty}
B.~D. McKay and A.~Piperno.
\newblock {\em Nauty and Traces user's guide (Version 2.5)}.
\newblock Computer Science Department, Australian National University,
  Canberra, Australia, (2013).

\bibitem{McKP14practical}
B.~D. McKay and A.~Piperno.
\newblock Practical graph isomorphism, {II}.
\newblock {\em Journal of Symbolic Computation}, 60:94--112, (2014).

\bibitem{MeiS90}
W.~Meier and O.~Staffelbach.
\newblock Nonlinearity criteria for cryptographic functions.
\newblock In J.-J. Quisquater and J.~Vandewalle, editors, {\em Advances in
  Cryptology --- EUROCRYPT '89: Workshop on the Theory and Application of
  Cryptographic Techniques}, volume 434 of {\em Lecture Notes in Computer
  Science},  549--562, Berlin, Heidelberg, (1990). Springer.

\bibitem{Mes2016}
S.~Mesnager.
\newblock {\em Bent functions}.
\newblock Springer, (2016).

\bibitem{Mul54}
D.~E. Muller.
\newblock Application of {Boolean} algebra to switching circuit design and to
  error detection.
\newblock {\em Transactions of the IRE Professional Group on Electronic
  Computers}, (3):6--12, (1954).

\bibitem{Neu06bent}
T.~Neumann.
\newblock {\em Bent functions}.
\newblock PhD thesis, University of Kaiserslautern, (2006).

\bibitem{OlSW1975}
J.~Olsen, R.~Scholtz, and L.~Welch.
\newblock Bent-function sequences.
\newblock {\em IEEE Transactions on Information Theory}, 28(6):858--864,
  November 1982.

\bibitem{PostgreSQL}
{PostgreSQL Global Development Group}.
\newblock {PostgreSQL}, (1996).
\newblock \url{https://www.postgresql.org}.

\bibitem{Rot76}
O.~S. Rothaus.
\newblock On ``bent'' functions.
\newblock {\em Journal of Combinatorial Theory, Series A}, 20(3):300--305,
  (1976).

\bibitem{Roy08normal}
G.~F. Royle.
\newblock A normal non-cayley-invariant graph for the elementary abelian group
  of order 64.
\newblock {\em Journal of the Australian Mathematical Society},
  85(03):347--351, (2008).

\bibitem{Rue1986}
R.~Rueppel.
\newblock {\em Analysis and design of stream ciphers}.
\newblock Communications and control engineering series. Springer, (1986).

\bibitem{CoCalc}
{SageMath, Inc.}
\newblock {\em CoCalc - Collaborative Calculation in the Cloud}, (2017).
\newblock \url{https://cocalc.com/}.

\bibitem{Sei79}
J.~J. Seidel.
\newblock Strongly regular graphs.
\newblock In {\em Surveys in combinatorics ({P}roc. {S}eventh {B}ritish
  {C}ombinatorial {C}onf., {C}ambridge, 1979)}, volume~38 of {\em London
  Mathematical Society Lecture Note Series},  157--180, Cambridge-New York,
  (1979). Cambridge Univ. Press.

\bibitem{SQLite}
{SQLite Consortium}.
\newblock {SQLite}, (2000).
\newblock \url{http://sqlite.org}.

\bibitem{Sta07}
P.~Stanica.
\newblock Graph eigenvalues and {Walsh} spectrum of {Boolean} functions.
\newblock {\em Integers: Electronic Journal Of Combinatorial Number Theory},
  7(2):A32, (2007).

\bibitem{Sti07combinatorial}
D.~R. Stinson.
\newblock {\em Combinatorial designs: constructions and analysis}.
\newblock Springer Science \& Business Media, (2007).

\bibitem{SageMath7517}
{The Sage Developers}.
\newblock {\em {S}ageMath, the {S}age {M}athematics {S}oftware {S}ystem
  ({V}ersion 7.5)}, (2017).
\newblock \url{http://www.sagemath.org}.

\bibitem{SageMath8418}
{The Sage Developers}.
\newblock {\em {S}ageMath, the {S}age {M}athematics {S}oftware {S}ystem
  ({V}ersion 8.4)}, (2018).
\newblock \url{http://www.sagemath.org}.

\bibitem{Tok15bent}
N.~Tokareva.
\newblock {\em Bent functions: results and applications to cryptography}.
\newblock Academic Press, (2015).

\bibitem{Ton96uniformly}
V.~D. Tonchev.
\newblock The uniformly packed binary [27, 21, 3] and [35, 29, 3] codes.
\newblock {\em Discrete Mathematics}, 149(1-3):283--288, (1996).

\bibitem{Ton07codes}
V.~D. Tonchev.
\newblock Codes.
\newblock In C.~Colbourne and J.~Dinitz, editors, {\em Handbook of
  combinatorial designs}, chapter VII.1,  677--701. CRC press, second edition,
  (2007).

\end{thebibliography}
